\documentclass{amsart}
\usepackage{amssymb}

\usepackage[left=1.5in,right=1.5in,bottom=1.5in,top=1.5in]{geometry}

\newcommand{\map}[1]{\xrightarrow{#1}}
\newcommand{\mil}{\varprojlim}

\newcommand{\iso}{\cong}
\newcommand{\define}{\stackrel{\mathrm{def}}{=}}

\newcommand{\Gal}{\mathrm{Gal}}
\newcommand{\Hom}{\mathrm{Hom}}
\newcommand{\Aut}{\mathrm{Aut}}
\newcommand{\End}{\mathrm{End}}
\newcommand{\Spec}{\mathrm{Spec}}
\newcommand{\Spf}{\mathrm{Spf}}
\newcommand{\Q}{\mathbb Q}
\newcommand{\Z}{\mathbb Z}

\newcommand{\C}{\mathbb C}
\newcommand{\F}{\mathbb F}

\newcommand{\co}{\mathcal O}
\newcommand{\alg}{\mathrm{alg}}
\newcommand{\ord}{\mathrm{ord}}
\newcommand{\length}{\mathrm{length}}

\newcommand{\Art}{\mathbf{Art}}
\newcommand{\ProArt}{ \mathbf{ProArt}  }
\newcommand{\vertical}{\mathrm{ver} }
\newcommand{\univ}{\mathrm{univ}}
\newcommand{\ZZ}{ \mathbb{Z}^\circ  }
\newcommand{\QQ}{ \mathbb{Q}^\circ }
\newcommand{\OO}{ \mathcal{O}^\circ }
\newcommand{\cmorder}{\mathcal{R}}

\input xy
\xyoption{all}

\begin{document}

\author{Benjamin Howard}
\title{Deforming  endomorphisms of supersingular Barsotti-Tate groups}

\email{howardbe@bc.edu}
\address{Department of Mathematics, Boston College, Chestnut Hill, MA, 02467}

\date{November 10, 2009}

\begin{abstract}
The formal deformation space of a supersingular $p$-Barsotti-Tate group over 
$\mathbb{F}_p^\alg$ of dimension two equipped with an action of $\mathbb{Z}_{p^2}$ 
is known to be isomorphic to the formal spectrum of a power series ring in two variables over the 
Witt ring of $\mathbb{F}_p^\mathrm{alg}$.  If one chooses an extra  $\mathbb{Z}_{p^2}$-linear endomorphism of the $p$-divisible
 group then the locus in the formal deformation space formed by those deformations for which the extra endomorphism lifts is a 
 closed formal subscheme of codimension two.  We give a complete description of the irreducible components of this formal 
 subscheme, compute the multiplicities of these components, and compute the intersection numbers of the components with a 
 distinguished closed formal subscheme of codimension one.  These calculations, which extend the Gross-Keating theory of quasi-
 canonical lifts,  are used in the companion article \emph{Intersection theory on Shimura surfaces II} to compute global intersection 
 numbers of special cycles on the integral model of a Shimura surface.
\end{abstract}

\maketitle

\theoremstyle{plain}
\newtheorem{Thm}{Theorem}[subsection]
\newtheorem{Prop}[Thm]{Proposition}
\newtheorem{Lem}[Thm]{Lemma}
\newtheorem{Cor}[Thm]{Corollary}
\newtheorem{Conj}[Thm]{Conjecture}
\newtheorem{BigTheorem}{Theorem}

\theoremstyle{definition}
\newtheorem{Def}[Thm]{Definition}
\newtheorem{Hyp}[Thm]{Hypothesis}

\theoremstyle{remark}
\newtheorem{Rem}[Thm]{Remark}
\newtheorem{Ques}[Thm]{Question}

\renewcommand{\labelenumi}{(\alph{enumi})}
\renewcommand{\theBigTheorem}{\Alph{BigTheorem}}


\section{Introduction}


Let $\mathfrak{g}_0$ be a connected $p$-Barsotti-Tate group over $\F=\F_p^\alg$ of height two and dimension one; that is, 
$\mathfrak{g}_0$ is the $p$-power torsion of  a supersingular elliptic curve over $\F$.   Fix   a quadratic field extension 
$E_0/\Q_p$.  As the endomorphism ring of $\mathfrak{g}_0$ is the maximal order in a quaternion division algebra over $\Q_p$, 
we may fix an action of $\co_{E_0}$ on $\mathfrak{g}_0$.  Let $\mathfrak{M}_0$ be the formal scheme over the Witt ring 
$$\ZZ_p=W(\F)$$ which represents the functor of deformations of $\mathfrak{g}_0$ to complete local Noetherian $\ZZ_p$-algebras
 with residue field $\F$.   Suppose now that we fix some $\gamma_0\in\co_{E_0}$ not contained in $\Z_p$ and ask for the locus in 
 $\mathfrak{M}_0$ of deformations for which the endomorphism $\gamma_0\in\mathrm{End}(\mathfrak{g}_0)$ lifts.  This locus, a 
 closed formal subscheme of $\mathfrak{M}_0$ denoted  $\mathfrak{Y}_0$,  depends only on the integer $c_0$ defined by 
$$
\Z_p[\gamma_0]=\Z_p+p^{c_0}\co_{E_0}
$$
 and not on $\gamma_0$ itself.   One knows that $\mathfrak{M}_0\iso \Spf(\ZZ_p[[x]])$, and the Gross-Keating theory of
  quasi-canonical lifts recalled in \S \ref{ss:quasi-canonical} gives a complete understanding of the closed  formal 
  subscheme $\mathfrak{Y}_0$.    For example, one knows that  $\mathfrak{Y}_0$ has $c_0+1$ irreducible components 
  which we label as $\mathfrak{C}(0),\ldots,\mathfrak{C}(c_0)$.  The component $\mathfrak{C}(s)$ is isomorphic to
   $\Spf(W_s)$ where $W_s$ is the ring of integers in a finite abelian  extension of the fraction field of 
$$
W_0 =  \begin{cases}
\ZZ_p & \mathrm{if\ } E_0/\Q_p\mathrm{\ is\ unramified} \\
\co_{E_0}\otimes_{\Z_p}\ZZ_p & \mathrm{if\ }  E_0/\Q_p\mathrm{\ is\ ramified}
\end{cases}
$$
 with Galois group isomorphic to $\co_{E_0}^\times / (\Z_p+p^s\co_{E_0})^\times$.  The restriction of the universal 
 deformation of $\mathfrak{g}_0$ to the component $\mathfrak{C}(s)$, a deformation of $\mathfrak{g}_0$ to $W_s$, 
 is the \emph{quasi-canonical lift of level $s$}, and has endomorphism ring $\Z_p+p^s\co_{E_0}$.

This general  type of problem, in which one begins with the formal deformation space of a $p$-Barsotti-Tate group  
and studies the locus in the deformation space obtained by imposing additional endomorphisms, is central to 
Kudla's program  \cite{kudla97,kudla04b,KRY} to relate intersection multiplicities of so-called \emph{special cycles} 
on Shimura varieties to Fourier coefficients of modular forms.  In such generality it is very difficult to determine the structure of this 
locus (for example, to determine its irreducible components and their multiplicities).  
In this article we extend the  one dimensional theory of 
Gross and Keating to the case of two dimensional  $p$-Barsotti-Tate groups.   More specifically, we will study those deformation 
spaces which arise by formally completing a Hilbert-Blumenthal surface (or Shimura surface) at a closed point in characteristic $p$.
  The determination of these deformation spaces is used in the companion article \cite{howardC} to relate the intersection 
  multiplicities of special cycles on a Shimura surface to Fourier coefficients of a Hilbert modular form.

 Let $\Z_{p^2}$ be the integer ring in the unramified quadratic extension of $\Q_p$ and define 
 $$
 \mathfrak{g}=\mathfrak{g}_0\otimes\Z_{p^2},
 $$ a $p$-Barsotti-Tate group of height four and dimension two which is equipped with an action of  
 $$
 \co_E = \co_{E_0}\otimes_{\Z_p}\Z_{p^2}.
 $$  
 Let $\mathfrak{M}$ be the formal $\ZZ_p$-scheme classifying deformations of $\mathfrak{g}$, with its $\Z_{p^2}$-action, to complete
  local Noetherian $\ZZ_p$-algebras with residue field $\F$.  Let $\mathfrak{Y}\map{}\mathfrak{M}$ be the closed formal subscheme 
  classifying those deformations for which the endomorphism $\gamma = \gamma_0\otimes 1\in \co_E$  lifts (equivalently, for which 
  the action of $\Z_{p^2}[\gamma]\iso \Z_{p^2}+p^{c_0}\co_E$ lifts).    One knows  that $\mathfrak{M}\iso \Spf(\ZZ_p[[x_1,x_2]])$, and 
  the problem is to determine the structure of $\mathfrak{Y}$.    To get started, one may note that there is cartesian diagram
$$
\xymatrix{
{\mathfrak{Y}_0} \ar[r] \ar[d]  &  {\mathfrak{Y}} \ar[d] \\
{\mathfrak{M}_0} \ar[r]  & {\mathfrak{M}} 
}
$$
in which all arrows are closed immersions and the horizontal arrows represent the functor which takes a deformation 
$\mathfrak{G}_0$ of $\mathfrak{g}_0$ to the deformation $\mathfrak{G}_0\otimes\Z_{p^2}$ of $\mathfrak{g}$.  Using the closed 
immersion $\mathfrak{Y}_0\map{}\mathfrak{Y}$ above, we view $\mathfrak{C}(s)$ also as a closed formal subscheme of 
$\mathfrak{Y}$.  There is a natural action of $\Aut_{\Z_{p^2}}(\mathfrak{g})$ on the deformation space $\mathfrak{M}$, and this action 
restricts to an action of $\co_E^\times$ on $\mathfrak{M}$ and on $\mathfrak{Y}$.   In particular $\co_E^\times$ acts on the set of all 
closed formal subschemes of $\mathfrak{Y}$, and for each $\xi\in\co_E^\times$ we set
$$
\mathfrak{C}(s,\xi) = \xi* \mathfrak{C}(s).
$$

Our first theorem describes the irreducible components of $\mathfrak{Y}$ when $E_0/\Q_p$ is unramified.  The horizontal 
components arise by taking $\co_E^\times$-orbits of components  of $\mathfrak{Y}_0$.  The main difference between the dimension 
two case under consideration and the dimension one case considered by Gross and Keating is the presence of two vertical 
components.  Define  a subgroup of $\co_E^\times$ 
$$ H_s=\co_{E_0}^\times  \cdot (\Z_{p^2} + p^s\co_E)^\times.$$

\begin{BigTheorem}
Assume that $E_0/\Q_p$ is unramified.  
\begin{enumerate}
\item
Each closed formal subscheme $\mathfrak{C}(s,\xi)$ constructed above is an irreducible component of $\mathfrak{Y}$, and every 
horizontal irreducible component of $\mathfrak{Y}$ has this form for a unique $0\le s\le c_0$ and a unique $\xi\in \co_E^\times/H_s$.  
\item
If $c_0>0$ then $\mathfrak{Y}$ has two vertical irreducible components; if $c_0=0$  there are no such components.
\end{enumerate}
\end{BigTheorem} 

\begin{proof}
The first claim is Proposition \ref{Prop:unr orbits},  the second is Proposition \ref{Prop:two components}.
\end{proof}

If $E_0/\Q_p$ is ramified then the situation is similar, but slightly more complicated.  Let $E=E_0\otimes_{\Q_p}\Q_{p^2}$, the fraction 
field of $\co_{E}$.    As $E/\Q_p$ is biquadratic there is a unique quadratic subfield $E_0'\subset E$ which is neither $E_0$ nor 
$\Q_{p^2}$, and we may fix an action of $\co_{E'_0}$ on $\mathfrak{g}_0$.  As $\co_E\iso \co_{E_0'}\otimes_{\Z_p} \Z_{p^2}$ 
canonically this choice determines a new action of $\co_E$ on $\mathfrak{g}\iso \mathfrak{g}_0\otimes\Z_{p^2}$ which is conjugate 
by some $w\in\Aut_{\Z_{p^2}}(\mathfrak{g})$ to the action used in the definition of the deformation space $\mathfrak{Y}$.  Let 
$\mathfrak{Y}_0'$ be the locus in $\mathfrak{M}_0$ of deformations for which the new action of $\Z_p+p^{c_0}\co_{E_0'}$ lifts, so that 
the irreducible components of $\mathfrak{Y}_0'$ are indexed, again by the theory of quasi-canonical lifts, as 
$\mathfrak{C}'(0),\ldots,\mathfrak{C}'(c_0)$.  Viewing each $\mathfrak{C}'(s)$ as a closed formal subscheme of $\mathfrak{M}$, we 
then set
$$
\mathfrak{C}'(s,\xi) = (\xi\circ w) * \mathfrak{C}'(s)
$$
for every $\xi\in\co_E^\times$ and $0\le s\le c_0$.  One can show that  $\mathfrak{C}'(s,\xi)$ is contained in $\mathfrak{Y}$ (although 
$\mathfrak{C}'(s)$ itself is not).   Define $$H_s'=\co_{E_0'}^\times\cdot (\Z_{p^2} + p^s\co_E)^\times.$$

\begin{BigTheorem}
Assume that $E_0/\Q_p$ is ramified.  
\begin{enumerate}
\item
Each closed formal subscheme $\mathfrak{C}(s,\xi)$ and $\mathfrak{C}'(s,\xi)$ constructed above is an irreducible component of 
$\mathfrak{Y}$.   Every horizontal irreducible component of $\mathfrak{Y}$ is either equal to  $\mathfrak{C}(s,\xi)$  for a unique 
$0\le s\le c_0$ and a unique $\xi\in \co_E^\times/H_s$, or is equal to $\mathfrak{C}'(s,\xi)$ for a unique $0\le s\le c_0$ and a unique 
$\xi\in\co_E^\times/H_s'$.
\item
If $c_0>0$ then $\mathfrak{Y}$ has two vertical irreducible components; if $c_0=0$ there are no such components.
\end{enumerate}
\end{BigTheorem}

\begin{proof}
The first claim is Proposition \ref{Prop:ram orbits},  the second is Proposition \ref{Prop:two components}.
\end{proof}

Having determined all irreducible components of $\mathfrak{Y}$, we next turn to the problem of determining the multiplicity 
$\mathrm{mult}_\mathfrak{Y}(\mathfrak{C})$ (in the sense of Definition \ref{Def:components}) of each such component 
$\mathfrak{C}$.   For the horizontal components this will be done using global techniques, i.e.~ by identifying $\mathfrak{Y}$ with the 
formal completion at a point of a Hilbert-Blumenthal surface.  For the vertical components the argument uses global techniques (the 
determination of the \emph{Hasse-Witt locus} of $\mathfrak{M}$ in the sense of \S \ref{ss:hilbert-blumenthal}) and a heavy dose of 
explicit calculations using Zink's theory of \emph{windows} for $p$-divisible groups.

\begin{BigTheorem}
Every horizontal irreducible component of $\mathfrak{Y}$ appears with multiplicity one.  If $c_0>0$ and $p$ is odd then each of the 
two vertical irreducible components (called $\mathfrak{C}_1^\vertical$ and $\mathfrak{C}_2^\vertical$)  of $\mathfrak{Y}$ appears 
with multiplicity 
$$
\mathrm{mult}_\mathfrak{Y}(\mathfrak{C}_i^\vertical)  =  2p^{c_0-1}+ 4p^{c_0-2}+ 6p^{c_0-3}+8p^{c_0-4}+\cdots+(2c_0)p^0.
$$
\end{BigTheorem}

\begin{proof}
If $E_0/\Q_p$ is unramified this is Proposition \ref{Prop:unr components}, if $E_0/\Q_p$ is ramified this is Proposition 
\ref{Prop:ram components}.
\end{proof}

The final problem is to determine the relative position of $\mathfrak{M}_0$ and $\mathfrak{Y}$ inside of $\mathfrak{M}$, or, more 
precisely, the intersection number of $\mathfrak{M}_0$ with each irreducible component of $\mathfrak{Y}$.     Let us say that an 
irreducible component $\mathfrak{C}$ of $\mathfrak{Y}$ is \emph{proper} if it is not contained in $\mathfrak{M}_0$ (and hence  meets 
$\mathfrak{M}_0$ properly).   We wish to compute the intersection multiplicity $I_\mathfrak{M}(\mathfrak{C},\mathfrak{M}_0)$ of 
$\mathfrak{C}$ and $\mathfrak{M}_0$ in the sense of Definition \ref{Def:components}.  Our notion of intersection does not take 
multiplicities into account, so really it is the value 
$\mathrm{mult}_\mathfrak{Y}(\mathfrak{C})\cdot I_\mathfrak{M}(\mathfrak{C},\mathfrak{M}_0)$ in which we are interested.  We have 
done this for every proper component of $\mathfrak{Y}$, under the hypothesis $p>2$.  Summing the results over all proper 
components yields the following results on the \emph{total proper intersection} of $\mathfrak{Y}$ and $\mathfrak{M}_0$.

\begin{BigTheorem}\label{ThmD}
Assume that $p$ is odd.  If $E_0/\Q_p$ is unramified then 
$$
\sum_{\mathfrak{C} \mathrm{\ proper} } \mathrm{mult}_{\mathfrak{Y}}(\mathfrak{C}) \cdot I_{\mathfrak{M}}(\mathfrak{C},\mathfrak{M}_0) 
=  c_0  \cdot \left(  \frac{p^{c_0+1}-1}{p-1}  +  \frac{p^{c_0}-1}{p-1}  \right).
$$
If $E_0/\Q_p$ is ramified then
$$
\sum_{\mathfrak{C} \mathrm{\ proper} } \mathrm{mult}_{\mathfrak{Y}}(\mathfrak{C}) \cdot I_{\mathfrak{M}}(\mathfrak{C},\mathfrak{M}_0) 
=  (2 c_0 +1) \frac{p^{c_0+1}-1}{p-1} .
$$
In both cases the sum is over all proper irreducible components of $\mathfrak{Y}$.
\end{BigTheorem}

\begin{proof}
If $E_0/\Q_p$ is unramified the contribution of all proper horizontal irreducible components is computed in Corollary 
\ref{Cor:horizontal reduced} (for the multiplicities) and Corollary \ref{Cor:unr proper horizontal} (for the intersection numbers), while 
the contribution of the vertical components is computed in  Proposition \ref{Prop:two components} (for the intersection numbers) and 
Proposition \ref{Prop:unr components} (for the multiplicities).  

If $E_0/\Q_p$ is ramified the contribution of all proper horizontal irreducible components is computed in Corollary 
\ref{Cor:horizontal reduced} (for the multiplicities) and Corollary \ref{Cor:ram proper horizontal} (for the intersection numbers), while 
the contribution of the vertical components is computed in Proposition \ref{Prop:two components} (for the intersection numbers) and 
Proposition \ref{Prop:ram components} (for the multiplicities).  
\end{proof}

It is in fact the calculation of the total proper intersection found in Theorem \ref{ThmD} which motivated this project.  This calculation of 
local intersection multiplicities is one part of a larger global calculation which relates the intersection multiplicities of special cycles on 
the integral model of a Shimura surface  to the Fourier coefficients of a Hilbert modular form of weight $3/2$.  This global application 
of the local calculations carried out herein is the content of  \cite{howardC}, which contains a higher-dimensional version of the 
Kudla-Rapoport-Yang \cite{KRY} intersection theory on  integral models of  Shimura curves.

\subsection{Acknowledgements}

This research was supported in part by NSF grant DMS-0556174, and by a Sloan Foundation Research Fellowship.  Portions of this 
work were carried out during a visit to the Mathematisches Institut der Universit\"at Bonn, and the author thanks the Institut and its 
members for their  hospitality.   The author especially thanks  Michael  Rapoport for  helpful conversations during that visit.

\subsection{Notation}

The following notation will be used throughout the article.  Let $\F=\F_p^\alg$ be an algebraic closure of the field of $p$ elements and 
let $\ZZ_p=W(\F)$ denote the ring of Witt vectors of $\F$.   Equivalently, $\ZZ_p$ is the completion of the strict Henselization of $\Z_p$ 
with respect to the unique $\Z_p$-algebra homomorphism $\Z_p\map{}\F$.    Let $\QQ_p$ be the fraction field of $\ZZ_p$ and let 
$\Q_{p^2}$ be the unique quadratic extension of $\Q_p$ contained in $\QQ_p$.  Denote by $\Z_{p^2}$ the ring of integers of 
$\Q_{p^2}$, and let $\sigma$ be the nontrivial Galois automorphism of $\Q_{p^2}/\Q_p$.

  When $R$ is a local ring we denote by $\mathfrak{m}_R$ the maximal ideal of $R$.  Denote by $\ProArt$  the category of  complete 
  (meaning complete and separated) local Noetherian  $\ZZ_p$-algebras $R$ equipped with an isomorphism  
  $R/\mathfrak{m}_R\map{}\F$.  Morphisms in $\ProArt$ are local  $\ZZ_p$-algebra homomorphisms. Let $\Art$ be the full 
  subcategory of $\ProArt$ whose objects are Artinian local $\ZZ_p$-algebras.    Any set-valued functor $\mathcal{F}$ on $\Art$ 
  extends naturally to a functor on $\ProArt$ by
$$
\mathcal{F}(R)\define \mil\mathcal{F}(R/\mathfrak{m}_R^k).
$$
We say that $\mathcal{F}$ is \emph{pro-representable} if there is an isomorphism of functors
$$
\mathcal{F} ( -) \iso \Hom_{\ProArt} (R_\mathcal{F}, - )
$$
for some object $R_\mathcal{F}$ of $\ProArt$.  When this is the case we often confuse $\mathcal{F}$ with the formal $\ZZ_p$-scheme 
$\Spf(R_\mathcal{F})$.


\section{Preliminaries}
\label{s:prelims}


In \S \ref{ss:functors} we define formal schemes over $\Spf(\ZZ_p)$ which represent certain deformation functors on $\Art$. In \S 
\ref{ss:hilbert-blumenthal} we relate these form schemes to the completed local rings of Hilbert-Blumenthal  surfaces and derive some 
consequences.


\subsection{Deformation functors}
\label{ss:functors}


  Let $E_0$ be a quadratic field extension of $\Q_p$ with ring of integers $\co_{E_0}$ and fix a $\Z_p$-algebra homomorphism
  \begin{equation}\label{psi fix}
  \psi:\co_{E_0}\map{}\F.
  \end{equation}
  Let $\mathfrak{g}_0$ be a connected  $p$-Barsotti-Tate group of dimension one and height two over $\F$.  Up to isomorphism there 
  is a unique such $\mathfrak{g}_0$ \cite[p.~93]{demazure}, and $\mathfrak{g}_0$ is isomorphic to the $p$-Barsotti-Tate group of any 
  supersingular elliptic curve over $\F$.  The $\Z_p$-algebra $\End(\mathfrak{g}_0)$ is the maximal order in the unique (up to 
  isomorphism) quaternion division algebra over $\Q_p$.   We may choose an embedding 
  $$
  j_0: \co_{E_0}\map{}\End(\mathfrak{g}_0)
  $$ 
  in such a way that the action of $\co_{E_0}$ on $\mathrm{Lie}(\mathfrak{g}_0)$ is through the homomorphism $\psi$.  Such an 
  embedding is unique up to $\Aut(\mathfrak{g}_0)$-conjugacy  by Corollary \ref{Cor:normalized action} below.  Pick any 
  $\gamma_0\in \co_{E_0}$ not contained in $\Z_p$, define $c_0$ by 
  $$
  \Z_p[\gamma_0]=\Z_p+p^{c_0}\co_{E_0},
  $$
   and view $\gamma_0$ as an endomorphism of $\mathfrak{g}_0$.    For an object $R$ of $\ProArt$,  a \emph{deformation} of 
   $\mathfrak{g}_0$ to $R$ is a pair $  (\mathfrak{G}_0,\rho_0)$ in which   $\mathfrak{G}_0$ is a $p$-Barsotti-Tate group over $R$ and  
   $$
   \rho_0 : \mathfrak{g}_0\map{} \mathfrak{G}_{0/\F}
   $$ 
   is an isomorphism.    We denote by $\mathfrak{M}_0$ the functor on $\Art$ which assigns to an object $R$ the set 
   $\mathfrak{M}_0(R)$ of isomorphism classes of  deformations of $\mathfrak{g}_0$ to $R$.  Let $\mathfrak{Y}_0$ be the subfunctor 
   of $\mathfrak{M}_0$ which assigns to $R$ the set  $\mathfrak{Y}_0(R)$  of isomorphism classes of deformations  $
   (\mathfrak{G}_0,\rho_0)$ of $\mathfrak{g}_0$ to $R$ for which the endomorphism $\gamma_0\in \End(\mathfrak{g}_0)$ lifts  to an 
   endomorphism of $\mathfrak{G}_0$ (uniquely, by the rigidity theorem  \cite[Theorem 2.1]{waterhouse72}). 
  
   Given a $p$-Barsotti-Tate group $\mathfrak{G}_0$ over an object $R$ of $\ProArt$  we define $\mathfrak{G}_0\otimes\Z_{p^2}$ to 
   be the $p$-Barsotti-Tate group representing the functor on $\Art$ defined by 
   $$
   S\mapsto \mathfrak{G}_0(S)\otimes_{\Z_p}\Z_{p^2}.
   $$  
   The representability of this functor is seen by choosing a $\Z_p$-basis of $\Z_{p^2}$, defining 
   $$
   \mathfrak{G}_0\otimes\Z_{p^2}=\mathfrak{G}_0\times\mathfrak{G}_0,
   $$ 
   and letting $\Z_{p^2}$ act on the right hand side through the   ring homomorphism $\Z_{p^2}\map{}M_2(\Z_p)$ induced by our 
   choice of basis.   The rule $\mathfrak{G}_0\mapsto\mathfrak{G}_0\otimes\Z_{p^2}$ defines a functor from the category of $p$-
   Barsotti-Tate groups over $R$ to the category of $p$-Barsotti-Tate groups over $R$ with $\Z_{p^2}$-action.    In particular we define 
   a $p$-Barsotti-Tate group 
    $$
 \mathfrak{g}=\mathfrak{g}_0\otimes\Z_{p^2}
 $$ 
   of height four and dimension two over $\F$.  Set 
   $$
   E=E_0\otimes_{\Q_p} \Q_{p^2},
   $$ let  
   $$\co_E=\co_{E_0}\otimes_{\Z_p}\Z_{p^2}$$ be the maximal order in $E$, and set $\gamma=\gamma_0\otimes 1\in\co_E$.  The 
   action of $\co_{E_0}$ on $\mathfrak{g}_0$ fixed above determines an action 
$$
j: \co_E\map{}\End(\mathfrak{g}),
$$ 
and in particular we may view $\gamma$ as an endomorphism of $\mathfrak{g}$. If $R$ is an object of $\ProArt$,  a 
\emph{deformation} of $\mathfrak{g}$ to $R$ is a pair  $(\mathfrak{G},\rho)$ in which   $\mathfrak{G}$ is a $p$-Barsotti-Tate group 
over $R$ equipped with an action $\Z_{p^2}\map{}\End(\mathfrak{G})$ and  $\rho:\mathfrak{g} \map{}\mathfrak{G}_{/\F}$ is a 
$\Z_{p^2}$-linear isomorphism.  Let $\mathfrak{M}$ be the  functor  on $\Art$ which assigns to each object $R$ the set of 
isomorphism classes of deformations of $\mathfrak{g}$  to $R$, and let $\mathfrak{Y}$ be the subfunctor of deformations of 
$\mathfrak{g}$ for which the endomorphism $\gamma$ lifts.  
   There is a  cartesian diagram of functors 
\begin{equation} \label{functor diagram}
\xymatrix{
{\mathfrak{Y}_0} \ar[r]^{\otimes \Z_{p^2}} \ar[d]  &  {\mathfrak{Y}} \ar[d] \\
{\mathfrak{M}_0} \ar[r]_{\otimes \Z_{p^2}}  & {\mathfrak{M}} 
}
\end{equation}
in which the horizontal arrows are defined by
   $$
(\mathfrak{G}_0,\rho_0)\mapsto   (\mathfrak{G}_0,\rho_0)\otimes\Z_{p^2}\define
 (\mathfrak{G}_0 \otimes\Z_{p^2},\rho_0\otimes\Z_{p^2}),
   $$
 where $\rho_0\otimes\Z_{p^2}$ is the isomorphism
   $$
   \mathfrak{g}=\mathfrak{g}_0\otimes\Z_{p^2}\map{\rho_0\otimes\mathrm{id} } 
   \mathfrak{G}_{0/\F}\otimes\Z_{p^2}\iso \mathfrak{G}_{/\F}.
   $$
 All  functors in (\ref{functor diagram}) are pro-representable by objects of $\ProArt$  and all arrows  correspond to surjections 
 between the pro-representing objects (this follows from \cite[Proposition 2.9]{rapoport96}).   Expressed differently, the diagram 
 (\ref{functor diagram}) can be identified with a cartesian diagram of formal schemes over $\ZZ_p$
\begin{equation*} 
\xymatrix{
{\Spf(R_{\mathfrak{Y}_0})  } \ar[r]^{\otimes \Z_{p^2}} \ar[d]  &  { \Spf( R_{\mathfrak{Y} } )} \ar[d] \\
{   \Spf( R_{\mathfrak{M}_0 }  )    } \ar[r]_{\otimes \Z_{p^2}}  & {  \Spf(  R_{\mathfrak{M}} )  }
}
\end{equation*}
in which all arrows are closed immersions.

 For any $\xi\in\co_E^\times$, viewed as an element of $\Aut(\mathfrak{g})$, and any $(\mathfrak{G},\rho )  \in \mathfrak{Y}(R)$ we 
 define a new deformation
$$
\xi* (\mathfrak{G},\rho) \define (\mathfrak{G},\rho\circ \xi^{-1}  )  \in \mathfrak{Y}(R).
$$
This defines an action of $\co_E^\times$ on the functor $\mathfrak{Y}$.  Using the isomorphism
$$
 \mathfrak{Y}(  -  ) \iso \Hom_{\ProArt}( R_{\mathfrak{Y}},  - )
$$
we find an action of $\co_E^\times$ on $\Hom_{\ProArt}( R_{\mathfrak{Y}},  -  )$ of the form $\xi*f=f\circ \xi^{-1}$ for some 
homomorphism $$\co_E^\times\map{}\Aut_{\ProArt}(R_{\mathfrak{Y}}).$$  The action of $\co_{E_0}^\times$ preserves the subfunctor 
$\mathfrak{Y}_0$ and similarly determines a homomorphism 
\begin{equation}\label{E_0 action}
\co_{E_0}^\times\map{}\Aut_{\ProArt}(R_{\mathfrak{Y}_0}).
\end{equation}

\begin{Def}\label{Def:components}
A \emph{component} of $\mathfrak{Y}$ is a closed formal subscheme  $\mathfrak{C}\map{}\mathfrak{Y}$ of the form 
$$
\mathfrak{C}=\Spf(R_{\mathfrak{Y} }/\mathfrak{p})\map{}\Spf(R_{\mathfrak{Y} })
$$ 
for some minimal prime ideal $\mathfrak{p}\subset R_{\mathfrak{Y}}$.    The component $\mathfrak{C}$ is \emph{horizontal} if 
$p\not\in\mathfrak{p}$, and is \emph{vertical} if $p\in\mathfrak{p}$. 
The \emph{multiplicity} of a component $\mathfrak{C}$ is  the length of the local ring at $\mathfrak{p}$
$$
\mathrm{mult}_{\mathfrak{Y}}(\mathfrak{C}) =\length_{R_{\mathfrak{Y},\mathfrak{p}} }( R_{\mathfrak{Y},\mathfrak{p}} )
$$
and the \emph{intersection number} is defined as the (possibly infinite) length
$$
I_{\mathfrak{M}}(\mathfrak{C},\mathfrak{M}_0) =
 \length_{R_{\mathfrak{M} }}( R_{\mathfrak{Y} }/\mathfrak{p}\otimes_{R_\mathfrak{M} }  R_{\mathfrak{M}_0} ).
$$
The component  $\mathfrak{C}$ is \emph{improper} if the closed immersion $\mathfrak{C}\map{}\mathfrak{Y}$ factors as 
$\mathfrak{C} \map{} \mathfrak{Y}_0\map{} \mathfrak{Y}$ and is \emph{proper} otherwise.  Intuitively, the phrase  ``$\mathfrak{C}$ is 
proper" should be interpreted as shorthand for ``$\mathfrak{C}$ meets $\mathfrak{M}_0$ properly in $\mathfrak{M}$." 
 \end{Def}

 Given any ring homomorphism $\phi:\co_{E_0}\map{}R$ we denote by $\overline{\phi}: \co_{E_0}\map{}R$ the homomorphism 
 obtained by precomposing $\phi$ with the nontrivial Galois automorphism of $E_0/\Q_p$.

 \begin{Lem}\label{Lem:normalized lie}
Let $\Delta_0$ be the maximal order in a ramified quaternion algebra over $\Q_p$ with uniformizing parameter $\Pi$.  Two 
embeddings $i_1,i_2: \co_{E_0}\map{}\Delta_0$ are $\Delta_0^\times$-conjugate if and only if they reduce to the same 
$\Z_p$-algebra homomorphism $\co_{E_0}\map{}\Delta_0/\Pi\Delta_0\iso\F_{p^2}$.
\end{Lem}

\begin{proof}
Easy exercise.
\end{proof}

\begin{Cor}\label{Cor:normalized action}
Let $\mathfrak{G}$ and $\mathfrak{G}'$ be $p$-Barsotti-Tate groups of dimension one and height two over $\F$, each equipped with 
an action of $\co_{E_0}$.  Then $\mathfrak{G}$ and $\mathfrak{G}'$ are $\co_{E_0}$-linearly isomorphic if and only if their Lie 
algebras are isomorphic as $\co_{E_0}$-modules.
\end{Cor}

\begin{proof}
By \cite[p.~93]{demazure} $\mathfrak{G}$ and $\mathfrak{G}'$ are isomorphic as $p$-Barsotti-Tate groups, and so we may assume 
that $\mathfrak{G}'=\mathfrak{G}$.  The two actions of $\co_{E_0}$ on $\mathfrak{G}$ are determined by homomorphisms 
$\phi_1,\phi_2:\co_{E_0}\map{}\Delta_0$ where $\Delta_0=\End(\mathfrak{G})$ is the maximal order in a ramified quaternion algebra 
over $\Q_p$.  The action of $\Delta_0$ on $\mathrm{Lie}(\mathfrak{G})$ is through some embedding $\Delta_0/\Pi\Delta_0\map{}\F$, 
and the claim reduces to Lemma \ref{Lem:normalized lie}.
\end{proof}


\subsection{Hilbert-Blumenthal surfaces}
\label{ss:hilbert-blumenthal}


It will be useful to identify $R_\mathfrak{M}$  with a  completed  local ring of a Hilbert-Blumenthal surface.  Choose number fields 
$\mathcal{E}_0$, $\mathcal{F}$, and $\mathcal{E}$ such that the following properties hold:
\begin{enumerate}
\item 
$\mathcal{F}$ is real quadratic, $\mathcal{E}_0$ is imaginary quadratic, and $\mathcal{E}\iso \mathcal{E}_0\otimes_\Q\mathcal{F}$,
\item
there are isomorphisms (which we now fix)
$$
\co_\mathcal{F}\otimes_\Z\Z_p\iso \Z_{p^2}  \qquad
\co_{\mathcal{E}_0} \otimes_\Z\Z_p \iso \co_{E_0} 
$$
\end{enumerate}
Set $\co_{\mathcal{E}_0,\mathcal{F}}=\co_{\mathcal{E}_0} \otimes_\Z\co_{\mathcal{F}}$ and 
let   ${\cmorder}\subset\co_{\mathcal{E}_0,\mathcal{F} }$ be the $\co_{\mathcal{F}}$-order defined by
$$
{\cmorder}\otimes_{\Z}\Z_\ell  =  \begin{cases}
\Z_{p^2}[\gamma] &\mathrm{if\ }\ell=p \\
\co_{\mathcal{E}_0,\mathcal{F}}\otimes_\Z\Z_\ell & \mathrm{if\ }\ell\not=p
\end{cases}
$$
for every rational prime $\ell$.   Choose an elliptic curve $\mathfrak{a}_0$ over $\F$ and an action 
$\co_{\mathcal{E}_0}\map{}\End(\mathfrak{a}_0)$ in such a way that the  $p$-Barsotti-Tate group of $\mathfrak{a}_0$ is $\co_{E_0}$-
linearly isomorphic to $\mathfrak{g}_0$ (this is possible by Corollary \ref{Cor:normalized action}).  The abelian surface 
$\mathfrak{a}=\mathfrak{a}_0\otimes\co_{\mathcal{F}}$ then carries an action of $\co_{\mathcal{E}_0,\mathcal{F}}$  and the 
$p$-Barsotti-Tate group of $\mathfrak {a}$ is $\co_{E}$-linearly isomorphic to  $\mathfrak{g}$.    

Fix a Zariski open neighborhood $U\map{}\Spec(\Z)$ of $p$.  

\begin{Def}
An \emph{abelian surface with real multiplication} (RM) over a $U$-scheme $S$ is an abelian scheme $ A\map{}S$  with an action 
$\co_\mathcal{F}\map{}\End( A)$ satisfying the \emph{Rapoport condition}:  every $s\in S$ has an  open affine neighborhood   over 
which  the coherent sheaf $\mathrm{Lie}( A)$ is a free $\co_\mathcal{F}\otimes_\Z\co_S$-module of rank one.
\end{Def}

\begin{Def}
Let $\mathcal{D}_{\mathcal{F}}$ be the different of $\mathcal{F}/\Q$.
Suppose $S$ is a $U$-scheme and $ A\map{}S$ is an RM abelian surface.  A \emph{Deligne-Pappas}  polarization of $ A$ is an 
$\co_\mathcal{F}$-linear polarization $\lambda: A\map{} A^\vee$ whose kernel is $ A[\mathcal{D}_\mathcal{F}]$.
\end{Def}

\begin{Def}
A \emph{polarized RM abelian surface} over a $U$-scheme $S$ is a pair $( A,\lambda)$ in which $ A$ is an RM abelian 
surface over $S$, and $\lambda: A\map{} A^\vee$ is a Deligne-Pappas polarization.
\end{Def}

\begin{Rem}
See \cite{vollaard} for an extensive discussion of  Deligne-Pappas polarizations and the Rapoport condition.  What we have called a 
Deligne-Pappas polarization is equivalent to what Vollaard calls a $\mathcal{D}_{\mathcal{F}}^{-1}$-polarization.  
\end{Rem}

Let $\mathcal{M}$ be the Deligne-Mumford stack over $U$ classifying polarized RM abelian surfaces over $U$-schemes.  Let 
$\mathcal{Y}$ be the Deligne-Mumford stack classifying triples $( A,\lambda,i)$ where $( A,\lambda)$ is a polarized RM abelian 
surface over a $U$-scheme and $i:{\cmorder}\map{}\End( A)$ is an action of ${\cmorder}$ which extends the action of 
$\co_{\mathcal{F}}$ on $ A$.  The stack $\mathcal{M}$ is smooth of relative dimension two over $U$.  
See \cite{deligne-pappas, rapoport78,vollaard}.   After shrinking $U$ and adding rigidifying  \'etale level structure at primes not in
 $U$ to these moduli problems the resulting stacks are schemes.  From now on we assume that such \'etale level structure has been 
 imposed, but make no explicit mention of it.

Suppose $ A_0$ is any elliptic curve over a scheme $S$ and set $A=A_0\otimes \co_\mathcal{F}$.  Fix a basis $\{x,y\}$ of 
$\co_\mathcal{F}$ as a $\Z$-module and let $\{x^\vee,y^\vee\}$ be the dual basis (relative to the trace form) of the inverse different 
$\mathcal{D}^{-1}_\mathcal{F}$.  These bases determine  two homomorphisms $i,i^\vee : \co_{\mathcal{F}} \map{}M_2(\Z)$  which 
are interchanged by transposition in $M_2(\Z)$.   Thus if we identify  $ A\iso  A_0\times A_0$ in such a way that the action of  
$\co_{\mathcal{F}}$  on the right hand side is through $i$,  the induced action of $\co_\mathcal{F}$ on 
$A^\vee \iso  A_0^\vee \times A_0^\vee$ is through $i^\vee$.  In other words there is an $\co_\mathcal{F}$-linear isomorphism 
$A^\vee\iso  A_0^\vee\otimes\mathcal{D}_F^{-1}$.  Moreover, if $\lambda_0: A_0\map{} A_0^\vee$ is the unique principal 
polarization of $ A_0$ then as in \cite[\S 3.1]{howardA} the isogeny
$$
 A \iso  A_0\otimes\co_\mathcal{F} \map{\lambda_0\otimes\iota}  A_0^\vee\otimes\mathcal{D}_\mathcal{F}^{-1} \iso  A^\vee
$$
is a polarization with  kernel $ A[\mathcal{D}_{\mathcal{F}}]$, and  does not depend on the choice of basis $\{x,y\}$.   Here we have 
used  $\iota$ to denote the inclusion $\co_\mathcal{F}\map{}\mathcal{D}_\mathcal{F}^{-1}$.     The above construction equips every 
abelian surface of the form $ A_0\otimes\co_\mathcal{F}$  with a canonical Deligne-Pappas polarization.  In particular the abelian 
surface $\mathfrak{a}=\mathfrak{a}_0\otimes\co_\mathcal{F}$ defined above has a Deligne-Pappas polarization 
$\lambda:\mathfrak{a}\map{}\mathfrak{a}^\vee$, and, as 
$\mathrm{Lie}(\mathfrak{a})\iso \mathrm{Lie}(\mathfrak{a}_0)\otimes_{\Z}\co_{\mathcal{F}}$, the abelian surface $\mathfrak{a}$ 
satisfies the Rapoport condition.  Thus the pair $(\mathfrak{a},\lambda)$ determines an $\F$-valued point of $\mathcal{M}$.  Our 
chosen action of $\co_{\mathcal{E}_0}$ on $\mathfrak{a}_0$ determines an action 
$i:\co_{\mathcal{E}_0,\mathcal{F}}\map{}\End_{\co_{\mathcal{F}}}(\mathfrak{a})$, and  the triple $(\mathfrak{a},\lambda,i)$ defines a 
point of $\mathcal{Y}(\F)$.

For an object $R$ of $\Art$  let $\mathfrak{M}^\mathfrak{a}(R)$ denote the set of isomorphism classes of deformations $( A,\rho)$ of 
$\mathfrak{a}$, with its $\co_{\mathcal{F}}$-action, to $R$.  Thus $ A$ is an abelian surface over $R$ equipped with an action of 
$\co_{\mathcal{F}}$, and $\rho:\mathfrak{a}\map{} A_{/\F}$ is an $\co_{\mathcal{F}}$-linear isomorphism.    The deformation $ A$ 
automatically satisfies the Rapoport condition, and  by the corollary to \cite[Theorem 3]{vollaard}  the Deligne-Pappas polarization of 
$\mathfrak{a}$ lifts uniquely to $ A$.  Similarly let $\mathfrak{Y}^\mathfrak{a}(R)$ denote the set of isomorphism classes of 
deformations $( A,  \rho)$ of $\mathfrak{a}$ for which the action  ${\cmorder}\map{}\End( A)$ lifts (necessarily uniquely, by 
\cite[Corollary 6.2]{mumford65}) to   an action of ${\cmorder}$ on $ A$.  For any abelian scheme $ A$ over a base scheme $S$ let 
$ A_{p^\infty}$ be the $p$-Barsotti-Tate group of $ A$. If we let $x\in \mathcal{M}(\F)$ be the geometric point corresponding to the 
polarized RM abelian surface $(\mathfrak{a},\lambda)$ then it follows from the discussion above and the Serre-Tate theorem  that 
there are isomorphisms of functors on $\Art$
$$
\Hom_{\ProArt}(  \co_{\mathcal{M},x}^\circ   , - )\iso  \mathfrak{M}^\mathfrak{a}(-)  \iso \mathfrak{M}(-)
$$
in which $\co_{\mathcal{M},x}^\circ$ is the completion of the strictly Henselian local ring of $\mathcal{M}$ at $x$, and  the second 
arrow is defined by passage to $p$-Barsotti-Tate groups $( A,\rho)\mapsto ( A_{p^\infty} , \rho )$.    Similarly, if we let 
$y\in \mathcal{Y}(\F)$ be the point corresponding to $(\mathfrak{a},\lambda)$ with its above $\co_{\mathcal{E}_0,\mathcal{F}}$-action 
then  there are isomorphisms
$$
\Hom_{\ProArt}(  \co_{\mathcal{Y},y}^\circ   , - )\iso  \mathfrak{Y}^\mathfrak{a}(-)  \iso \mathfrak{Y}(-).
$$
In particular there is a commutative diagram in $\ProArt$
$$
\xymatrix{
{  \OO_{\mathcal{Y},y} } \ar[r]\ar[d] &     { R_\mathfrak{Y}  }  \ar[d]\\
 {  \OO_{\mathcal{M},x} }  \ar[r] &   {  R_\mathfrak{M} }
}
$$
in which the horizontal arrows are isomorphisms.

By definition of $\mathfrak{M}$, for an object $R$ of $\ProArt$ an element of 
$$
\mathfrak{M}(R) = \mil\mathfrak{M}(R/\mathfrak{m}_R^k)
$$
is a compatible family $(\mathfrak{G}^{(k)},\rho^{(k)})$ of deformations of $\mathfrak{g}$ to $R/\mathfrak{m}_R^k$.  One would like to 
know that such a family comes from a single deformation $(\mathfrak{G},\rho)$ of $\mathfrak{g}$ to $R$. This is true in great 
generality (see \cite[Lemma 2.4.4]{dejong95}), but in this particular case one can use the bijection
$$
\Hom_{\ProArt}(\co_{\mathcal{M},x}^\circ, R)\iso \mathfrak{M}(R)
$$
to see that the $p$-divisible group of the pullback of the universal Hilbert-Blumenthal moduli via
$$
\Spec(R)\map{}\Spec(\co_{\mathcal{M},x}^\circ)\map{}\mathcal{M}
$$
gives the desired deformation of $\mathfrak{g}$ to $R$.  Similarly any element of $\mathfrak{Y}(R)$, $\mathfrak{M}^\mathfrak{a}(R)$, 
or $\mathfrak{Y}^\mathfrak{a}(R)$, \emph{a priori} defined as a compatible family of deformations  to Artinian quotients of $R$, in fact 
determines  a  deformation to  $R$ (necessarily unique  by \cite[Corollary 8.4.6]{FGA}).

We will exploit the isomorphism $\co_{\mathcal{Y},y}^\circ \iso R_\mathfrak{Y}$ to deduce properties about the deformation space 
$\mathfrak{Y}$ which seem difficult to obtain by working purely in the context of $p$-Barsotti-Tate groups.  The following lemma is a 
good example of this.

\begin{Lem}\label{Lem:etale completion}
The $\QQ_p$-algebra $R_{\mathfrak{Y}}[1/p]$ is a finite product of field extensions  of finite degree.
\end{Lem}

\begin{proof}
Exactly as in \cite[Lemma 3.1.3]{howardA} one can choose a prime $\ell\in U$ for which $\cmorder\otimes_\Z\Z_\ell\iso \Z_\ell^4$, and 
use the Serre-Tate deformation theory of ordinary abelian varieties to prove that $\mathcal{Y}\times_{U} \Spec( W(\F_\ell^\alg))$ is 
isomorphic to a disjoint union of copies of $\Spec(W(\F_\ell^\alg))$.  It follows that $\mathcal{Y}_{/\Q}$ is a disjoint union of spectra of 
number fields.

Let $\Spec(R)\map{}\mathcal{Y}_{/\ZZ_p}$ be an open affine neighborhood of the point $y\in \mathcal{Y}_{/\ZZ_p}(\F)$  
corresponding to the triple $(\mathfrak{a},\lambda,i)$.  By the previous paragraph   $R\otimes_{\ZZ_p}\QQ_p$   is  a finite product of 
field extensions of $\QQ_p$ of finite degree.     Let $R_y$ be the local ring of $R$ at $y$ and let $\widehat{R}_y$ be the completion of 
$R_y$ with respect to the topology induced by its maximal ideal.    We let  $I\subset R$ be the ideal of $\ZZ_p$-torsion elements and 
set $S=R/I$.  The  local ring $S$ is then free of finite rank as a $\ZZ_p$-module, and so admits a decomposition 
$$
S\iso \prod_{\mathfrak{m}} S_\mathfrak{m}
$$
as a product of complete and separated local rings, where $\mathfrak{m}$ runs over the finitely many maximal ideals of $S$.  There 
are two possibilities to consider: either the $\ZZ_p$-algebra map $R\map{}\F$ determined by $y$ factors through $S$, or it does not.   
If $R\map{}\F$ does not factor through $S$ then the local ring  $R_y$ contains an invertible $\ZZ_p$-torsion element, and hence 
$\widehat{R}_y$ is  $\ZZ_p$-torsion.  If $R\map{}\F$ does factor through $S$ then  there is a unique factor $S_\mathfrak{m}$ in the 
above decomposition  for which the composition  $R\map{}S_\mathfrak{m}$ extends to a (necessarily surjective) homomorphism of 
local rings $R_y\map{}S_\mathfrak{m}$ with $\ZZ_p$-torsion kernel $I\otimes_R R_y$.  As $S_\mathfrak{m}$ is complete, this map  
extends uniquely to a homomorphism $\widehat{R}_y\map{}S_\mathfrak{m}$, still with $\ZZ_p$-torsion kernel.  We deduce that 
$$
\widehat{R}_y\otimes_{\ZZ_p}\QQ_p \iso S_\mathfrak{m}\otimes_{\ZZ_p}\QQ_p
$$
is a  direct factor of the product of fields
$$
R\otimes_{\ZZ_p}\QQ_p\iso S\otimes_{\ZZ_p}\QQ_p.
$$
In either case we see that $\widehat{R}_y\otimes_{\ZZ_p}\QQ_p$ is a finite product of  field extensions of  $\QQ_p$ of finite degree.   
Using  $\widehat{R}_y  \iso \co_{Y,y}^\circ \iso R_\mathfrak{Y}$ completes the proof.
\end{proof}

\begin{Cor}\label{Cor:horizontal reduced}
Let $\mathfrak{p}$ be a  prime ideal of $R_\mathfrak{Y}$ with $p\not\in\mathfrak{p}$.  Then the local ring 
$R_{\mathfrak{Y},\mathfrak{p}}$ is a field extension of $\QQ_p$ of finite degree. In particular, every horizontal component of 
$\mathfrak{Y}$ has multiplicity one.
\end{Cor}

\begin{proof}
This is immediate from Lemma \ref{Lem:etale completion}.
\end{proof}

 For each $\xi\in\co_{E}^\times$ let  $I_\xi \subset\co_{\mathcal{E}_0,\mathcal{F}}$ be the proper fractional ${\cmorder}$-ideal defined 
 by
$$
I_\xi \otimes_\Z\Z_\ell =  \begin{cases}
\xi \cdot ( {\cmorder}\otimes_\Z\Z_p )  &\mathrm{if\ } \ell=p \\
{\cmorder}\otimes_{\Z}\Z_\ell &\mathrm{if\ }\ell\not=p
\end{cases}
$$
for all primes $\ell$.  Let  $( A,\rho) \in  \mathfrak{Y}^\mathfrak{a}(R)$ be a deformation of $\mathfrak{a}$ to some object $R$ of $\Art$ 
for which  the action of ${\cmorder}$ lifts.   Denoting by $ A\mapsto A^\sim$ the reduction from $R$ to $\F$, there are canonical 
isomorphisms
\begin{equation}\label{twisted reduction}
(A\otimes_{{\cmorder}} I_\xi)^\sim  \iso \mathfrak{a}\otimes_{{\cmorder}} I_\xi \iso \mathfrak{a}
\end{equation}
in which the first isomorphism is  
$$
( A\otimes_{{\cmorder}} I_\xi )^\sim \iso
  A^\sim \otimes_{{\cmorder}} I_\xi \map{  \rho^{-1}\otimes \mathrm{id}  } \mathfrak{a} \otimes_{{\cmorder}} I_\xi 
$$
and the second is given on points by  
$$
\mathfrak{a}(S)\otimes_{{\cmorder}}I_\xi \map{} \mathfrak{a}(S) \qquad P\otimes \alpha\mapsto \alpha\cdot P
$$
for any $\F$-scheme $S$ (the latter makes sense because the action of ${\cmorder}$ on $\mathfrak{a}$ extends to an action of 
$\co_{\mathcal{E}_0,\mathcal{F}}$).   Expressed differently, we are making use of the canonical isomorphism
$$
\mathfrak{a}\otimes_{{\cmorder}}I_\xi\iso \mathfrak{a}\otimes_{\co_{\mathcal{E}_0,\mathcal{F}  }} (I_\xi\otimes_{{\cmorder}} 
\co_{\mathcal{E}_0,\mathcal{F}}) \iso \mathfrak{a}\otimes_{\co_{\mathcal{E}_0,\mathcal{F}  }} \co_{\mathcal{E}_0,\mathcal{F}} \iso 
\mathfrak{a}
$$
arising from $I_\xi\co_{\mathcal{E}_0,\mathcal{F}} = \co_{\mathcal{E}_0,\mathcal{F}}$.   Denote by $\rho\otimes_{{\cmorder}} I_\xi$  
the inverse of the composition (\ref{twisted reduction})  and  note that while $ A\otimes_{{\cmorder}}I_\xi$ depends only on the class of 
$I_\xi$ in $\mathrm{Pic}({\cmorder})$ the isomorphism $\rho \otimes_{{\cmorder}}I_\xi$ depends on the fractional ideal $I_\xi$ itself.   
We obtain a new deformation of $\mathfrak{a}$, with its ${\cmorder}$-action, to $R$
$$
( A,\rho) \otimes_{{\cmorder}}I_\xi  \define ( A\otimes_{\cmorder} I_\xi , \rho\otimes_{\cmorder} I_\xi ) \in    \mathfrak{Y}^\mathfrak{a}(R) 
$$
and the operation $\otimes_{{\cmorder}} I_\xi$ defines an automorphism of  the functor   $\mathfrak{Y}^\mathfrak{a}$.

\begin{Lem}\label{Lem:pic swindle}
For any  $\xi\in\co_{E}^\times$  the diagram
$$
\xymatrix{
{\mathfrak{Y}^\mathfrak{a} }  \ar[r]\ar[d]_{ \otimes_{{\cmorder}}I_\xi }  &   {\mathfrak{Y} }  \ar[d]^{\xi*}  \\ 
{ \mathfrak{Y}^\mathfrak{a} } \ar[r] &   {\mathfrak{Y} }
}
$$
commutes (the action on the right is that defined in \S \ref{ss:functors}).
\end{Lem}

\begin{proof}
Fix an object $R$ of $\Art$ and a deformation $( A,\rho) \in \mathfrak{Y}^\mathfrak{a}(R)$.   
The $p$-Barsotti-Tate group of $ A\otimes_{{\cmorder}} I_\xi$ represents the functor on $R$-schemes
$$
S \mapsto
 A_{p^\infty}(S)\otimes_{\Z_{p^2}[\gamma]} \xi \Z_{p^2}[\gamma],
 $$ 
and so there is an isomorphism of $p$-Barsotti-Tate groups
\begin{equation}\label{twisting deformation match}
 A_{p^\infty} \map{} ( A\otimes_{{\cmorder}}I_\xi)_{p^\infty}
\end{equation}
which on $S$-points is given  by $P\mapsto P\otimes\xi$.  Reducing this isomorphism modulo $\mathfrak{m}_R$ one finds the 
commutative diagram
$$
\xymatrix{
{  A^\sim_{p^\infty}    } \ar[r]    &  {    ( A \otimes_{{\cmorder}} I_\xi  )^\sim_{p^\infty}  }   \\
{  \mathfrak{a}_{p^\infty} } \ar[r]_{\xi}  \ar[u]^{\rho}  & {   \mathfrak{a}_{p^\infty} } \ar[u]_{\rho\otimes_{{\cmorder}} I_\xi}.
}
$$
In other words (\ref{twisting deformation match})  defines an isomorphism of  deformations 
$$
( A_{p^\infty} ,\rho\circ\xi^{-1} ) \iso 
( ( A\otimes_{{\cmorder}}I_\xi)_{p^\infty},  \rho\otimes_{{\cmorder}} I_\xi  )
$$
of $\mathfrak{a}_{p^\infty}\iso \mathfrak{g}$.
 \end{proof}

 Suppose $R$ is the ring of integers in a finite extension of $\QQ_p$ and fix some $(\mathfrak{G}, \rho)\in\mathfrak{Y}(R)$.  Let 
 $\mathrm{Ta}_p(\mathfrak{G})$ be the $p$-adic Tate module of $\mathfrak{G}$.  Note that the action of $\Z_{p^2}$ and  the 
 endomorphism $\gamma \in \co_E$ make $\mathrm{Ta}_p(\mathfrak{G})$ into a $\Z_{p^2}[\gamma]$-module, and that 
 $\mathrm{Ta}_p(\mathfrak{G})\otimes_{\Z_p}\Q_p$ is free of rank one over $E$.

\begin{Def}\label{Def:invariants}
Define the \emph{geometric CM-order} of $(\mathfrak{G} ,\rho)$ by
 $$
 \co(\mathfrak{G}) = \{ x\in E\mid x \cdot \mathrm{Ta}_p(\mathfrak{G})\subset \mathrm{Ta}_p(\mathfrak{G})\} 
 $$
and define the \emph{reflex type} of $(\mathfrak{G}, \rho)$ to be the isomorphism class of   $\mathrm{Lie}(\mathfrak{G}) $ as a module 
over  $\Z_{p^2}[\gamma]\otimes_{\Z_p}R$. 
\end{Def}

\begin{Rem}\label{Rem:reflex pairs}
The action of $\Z_p^2\otimes_{\Z_p} R\iso R\times R$ on $\mathrm{Lie}(\mathfrak{G})$ induces a decomposition 
$$
\mathrm{Lie}(\mathfrak{G}) \iso \Lambda_1\oplus \Lambda_2
$$
in such a way that  each $\Lambda_i$ a free $R$-module of rank one, and so that $\Z_{p^2}$ acts through the embedding 
$\Z_{p^2}\map{}\ZZ_p \map{} R$ on $\Lambda_1$ and through the conjugate embedding on $\Lambda_2$.  The action of  
$\Z_p[\gamma_0]$  preserves this decomposition and determines a pair of embeddings $(\phi_1,\phi_2)$ of  $\Z_p[\gamma_0]$ 
into $\End_R(\Lambda_i)\iso R$.  This pair of embeddings completely determines the reflex type of the deformation 
$(\mathfrak{G},\rho)\in\mathfrak{Y}(R)$.
\end{Rem}

The following proposition will be crucial in the determination of the horizontal components of $\mathfrak{Y}$ carried out in 
\S \ref{s:horizontal}.

\begin{Prop}\label{Prop:CM orbits}
Let $R$ be the integer ring of a finite extension of $\QQ_p$.  Two deformations 
$$
(\mathfrak{G}, \rho), (\mathfrak{G}' ,\rho') \in\mathfrak{Y}(R)
$$
lie in the same $\co_E^\times$-orbit if and only if they have the same geometric CM-order and the same reflex type.
\end{Prop}

\begin{proof}
One implication  is obvious: if there is a $\xi\in\co_{E}^\times$  for which 
$$
(\mathfrak{G}', \rho') \iso (\mathfrak{G}, \rho\circ\xi^{-1})
$$ 
then in particular there is a  $\Z_{p^2}[\gamma]$-linear isomorphism $\mathfrak{G}'\iso \mathfrak{G}$, which implies that the 
geometric CM-orders and  reflex types agree.
For the other implication we use Lemma \ref{Lem:pic swindle}.  Fix an embedding  $\mathrm{Frac}(R)\map{} \C$.  It suffices to prove 
that if we are given  two deformations 
$$
( A, \rho), ( A', \rho') \in\mathfrak{Y}^\mathfrak{a}(R)
$$
of $\mathfrak{a}$   for which $\co( A_{p^\infty})=\co( A'_{p^\infty})$ and 
$$
\mathrm{Lie}( A) \iso \mathrm{Lie}( A' )
$$
as $\Z_{p^2}[\gamma]\otimes_{\Z_p} R$-modules,  then there is a $\xi\in \co_{E}^\times$ such that
$$
( A', \rho') \iso ( A , \rho  )  \otimes_{\cmorder}  I_\xi   
$$
as deformations of $\mathfrak{a}$.  The pro-representability of $\mathfrak{Y}^\mathfrak{a}$ implies that the map 
$\mathfrak{Y}^\mathfrak{a}(R)\map{}\mathfrak{Y}^\mathfrak{a}(R')$ is injective whenever $R'$ is the ring of integers of a finite 
extension of the fraction field of $R$, thus it suffices to prove the existence of such a $\xi$ after enlarging $R$.  Therefore we are free 
to assume that
$$
\End_{{\cmorder}}( A) =\End_{{\cmorder}}( A_{/\mathrm{Frac}(R)} )  = \End_{{\cmorder} }( A_{/\C}).
$$
The theory of complex multiplication implies that  $\End_{{\cmorder}}( A)$   is an $\co_\mathcal{F}$-order in $\mathcal{E}$ which 
satisfies 
$$
\End_{{\cmorder}}( A) \otimes_{\Z}\Z_\ell \iso \End_{{\cmorder}\otimes_{\Z}\Z_\ell } (\mathrm{Ta}_\ell( A) )
$$
for every prime $\ell$.  Using 
$$
\mathrm{Ta}_\ell( A)\iso\mathrm{Ta}_\ell(\mathfrak{a}_0)\otimes_{\Z} \co_{\mathcal{F}}
\iso (\co_{\mathcal{E}_0}\otimes_\Z\Z_\ell ) \otimes_{\Z_\ell}  (\co_{\mathcal{F}}\otimes_\Z\Z_\ell)
$$
for $\ell\not=p$  we find that  
$$
\End_{{\cmorder}}( A)\otimes_{\Z}\Z_\ell = \begin{cases}
\co( A_{p^\infty} ) & \mathrm{if\ }\ell=p \\
\co_{\mathcal{E}_0 ,\mathcal{F}}\otimes_\Z\Z_\ell & \mathrm{if\ }\ell\not=p.
\end{cases}
$$
Applying the same reasoning to $ A'$ shows that $ A_{/\C}$ and $ A'_{/\C}$  have the same endomorphism ring.  As the CM abelian 
surfaces $A_{/\C}$ and $A'_{/\C}$ have the same reflex types (in the classical sense) and endomorphism rings,  the theory of 
complex multiplication implies that there is some $I\in\mathrm{Pic}(\End_{\cmorder}(A))$ such that
$$
 A_{/\C}' \iso  A_{/\C} \otimes_{\End_{\cmorder}(A)  } I.
$$
Replacing $I$ by any one of its preimages under $\mathrm{Pic}({\cmorder})\map{}\mathrm{Pic}(\End_{\cmorder}(A) )$ we find an 
$I\in\mathrm{Pic}({\cmorder})$ such that
$$
 A_{/\C}' \iso  A_{/\C} \otimes_{{\cmorder}} I,
$$
and after possibly further enlarging $R$ we may assume that
\begin{equation}\label{pic twist}
 A_{/\mathrm{Frac}(R)}' \iso  A_{/\mathrm{Frac}(R)} \otimes_{{\cmorder}} I .
\end{equation}
Applying the theory of N\'eron models over the discrete valuation ring $R$ we find that (\ref{pic twist})  extends uniquely to an 
isomorphism
$$
 A'  \iso  A \otimes_{ {\cmorder} } I
$$
of abelian schemes over $R$.  As the reductions of $ A'$ and $ A$ to $\F$ are  isomorphic to $\mathfrak{a}$, we deduce that there 
exists an isomorphism $\mathfrak{a}\iso\mathfrak{a}\otimes_{{\cmorder} } I$ of abelian varieties over $\F$.  Hence we obtain  
isomorphisms of $\co_{ \mathcal{E}_0 , \mathcal{F} }$-modules
\begin{eqnarray*}
\Hom_{\co_{\mathcal{E}_0,\mathcal{F}}}( \mathfrak{a} , \mathfrak{a} )
& \iso &
\Hom_{\co_{\mathcal{E}_0,\mathcal{F}}}( \mathfrak{a}, \mathfrak{a}  \otimes_{{\cmorder}} I ) \\
& \iso &
\Hom_{\co_{\mathcal{E}_0,\mathcal{F}}}( \mathfrak{a}, \mathfrak{a} )  \otimes_{{\cmorder}} I .
\end{eqnarray*}
Using
$$
\Hom_{\co_{\mathcal{E}_0,\mathcal{F}}}( \mathfrak{a} , \mathfrak{a} ) \iso 
\End_{\co_{\mathcal{E}_0}} (\mathfrak{a}_0)\otimes\co_{\mathcal{F}}\iso \co_{\mathcal{E}_0,\mathcal{F}}
$$
 we find 
$$
\co_{\mathcal{E}_0,\mathcal{F}}\iso \co_{\mathcal{E}_0,\mathcal{F}}\otimes_{{\cmorder}} I
$$
as $\co_{\mathcal{E}_0, \mathcal{F} }$-modules.  In other words $I$ lies in the kernel of 
$\mathrm{Pic}({\cmorder})\map{} \mathrm{Pic}(\co_{\mathcal{E}_0,\mathcal{F}} )$.   Every such $I$ has the form $I_\xi$ for some 
$\xi\in\co_{E}^\times$, and we have now found an isomorphism $ A'\map{}  A\otimes_{{\cmorder}} I_\xi$ of abelian schemes over 
$R$.  Reducing this isomorphism modulo $\mathfrak{m}_R$ yields a commutative diagram
$$
\xymatrix{
 { (A')^\sim }  \ar[r]   &   { ( A\otimes_{{\cmorder}}I_\xi )^\sim }  \\
 {\mathfrak{a} } \ar[r] \ar[u]^{\rho'}   &  { \mathfrak{a}  } \ar[u]_{  \rho\otimes_{{\cmorder}} I_\xi  }
 }
 $$
in which the bottom horizontal arrow is some  $\co_{\mathcal{E}_0,\mathcal{F}}$-linear automorphism of $\mathfrak{a}$.  Every such 
automorphism is given by the action of some $\zeta\in\co_{\mathcal{E}_0,\mathcal{F}}^\times$, and if we replace $\xi$ by 
$\xi\zeta^{-1}$ then the bottom horizontal arrow is the identity automorphism of $\mathfrak{a}$.   This shows that there is an 
isomorphism of deformations
 $$
 ( A',\rho') \iso ( A,\rho)\otimes_{{\cmorder}} I_{\xi},
 $$
completing the proof.
\end{proof}

\begin{Lem}\label{Lem:mostly ss}
All but finitely many triples  $(A,\lambda, i) \in \mathcal{Y}(\F)$  are supersingular.
\end{Lem}

\begin{proof}
The only possible slope sequences for the $p$-Barsotti-Tate group of an abelian surface over $\F$ are $\{0,\frac{1}{2},1\}$,  
$\{0,0,1,1\}$,  or $\{\frac{1}{2},\frac{1}{2}\}$.  A $p$-Barsotti-Tate group with slopes $\{0,\frac{1}{2},1\}$ cannot admit an action of 
$\Z_{p^2}$, thus every point of $\mathcal{Y}(\F)$ is either ordinary or supersingular.  If $\mathcal{Y}^\ord(\F)\subset \mathcal{Y}(\F)$ is 
the subset of ordinary points then taking Serre-Tate canonical lifts   gives an injection 
$\mathcal{Y}^\ord(\F)\map{}\mathcal{Y}(\ZZ_p)$.   The proof of Lemma \ref{Lem:etale completion} implies that $\mathcal{Y}(\ZZ_p)$ is 
finite, completing the proof.
\end{proof}

Suppose $R$ is an object of $\ProArt$ with $pR=0$ and that  $\mathfrak{G}$ is a $p$-Barsotti-Tate group of dimension $k$ over $R$.   
The \emph{Hasse-Witt invariant} of $\mathfrak{G}$ is the homomorphism of free rank one $R$-modules
$$
\wedge^k\mathrm{Ver}:\wedge^k\mathrm{Lie}(\mathfrak{G}^{(p)}) \map{}\wedge^k\mathrm{Lie}(\mathfrak{G})
$$ 
where $\mathrm{Ver}:\mathfrak{G}^{(p)}\map{}\mathfrak{G}$ is the usual Verschiebung morphism.   Consider in particular the 
universal deformation $(\mathfrak{g}^\univ,\rho^\univ)\in \mathfrak{M}(R_\mathfrak{M})$ and its reduction $(\mathfrak{G},\rho)$ to 
$R=R_{\mathfrak{M}}\otimes _{\ZZ_p}\F$.   By choosing an isomorphism $ R\iso \wedge^2\mathrm{Lie}(\mathfrak{G})$ we may 
identify the kernel of the Hasse-Witt invariant  with an ideal $I\subset R$.  The kernel of 
$$
R_{\mathfrak{M}}\map{} R\map{}R/I
$$
is denoted $I_{\mathrm{HW}}$, and the closed formal subscheme 
$$
\mathfrak{M}_{\mathrm{HW}} \define \Spf(R_{\mathfrak{M}} / I_{\mathrm{HW}} ) \map{}\mathfrak{M}
$$
is the \emph{Hasse-Witt locus} of $\mathfrak{M}$.

\begin{Prop}\label{Prop:hasse-witt}
Suppose  $\mathfrak{C}$ is a vertical component of $\mathfrak{Y}$ in the sense of Definition \ref{Def:components}.  The closed 
immersion $\mathfrak{C}\map{}\mathfrak{M}$ factors as 
$$
\mathfrak{C}\map{}\mathfrak{M}_{\mathrm{HW}} \map{}\mathfrak{M}.
$$
\end{Prop}

\begin{proof}
Fix a vertical component $\mathfrak{C}=R_{\mathfrak{Y}}/\mathfrak{p}$ and let $(\mathfrak{G},\rho)$ be the pullback to 
$R_{\mathfrak{Y}}/\mathfrak{p}$ of the universal deformation of $\mathfrak{g}$.  We must show that the Hasse-Witt invariant of 
$\mathfrak{G}$ is $0$.  Let $y\in \mathcal{Y}(\F)$ be the point corresponding to the polarized RM abelian surface 
$(\mathfrak{a},\lambda)$ with the action of $\co_{\mathcal{E}_0,\mathcal{F} }$ induced by our fixed action of $\co_{\mathcal{E}_0}$ 
on $\mathfrak{a}_0$.  Let $V\map{}\mathcal{Y}_{/\F}$ be an open affine neighborhood of $y$ and let $A$ be the  restriction to $V$ of 
the universal RM abelian surface over $\mathcal{Y}_{/\F}$.    After shrinking $V$ we may assume that $\mathrm{Lie}( A)$ is free of 
rank two over $\Gamma(V)$, and (by Lemma \ref{Lem:mostly ss}) that the fiber of $ A$ at every point of $V(\F)$ is supersingular.

 After choosing generators for the exterior squares of the Lie algebras of $ A$ and $ A^{(p)}$ the Verschiebung 
$$
\wedge^2\mathrm{Ver}:\wedge^2\mathrm{Lie}( A^{(p)}) \map{}\wedge^2\mathrm{Lie}( A)
$$
is given by multiplication by some $\beta \in \Gamma(V)$.  At any point $z\in V(\F)$ the supersingularity of the fiber $ A_z$ implies 
that 
$$
\mathrm{Ver}:\mathrm{Lie}( A_z^{(p)} )   \map{}  \mathrm{Lie} ( A_z)
$$
is not invertible, and thus $\beta$ vanishes at every $z\in V(\F)$ (in general, an abelian variety over $\F$ is ordinary if and only if the 
Verschiebung map induces an isomorphism of Lie algebras).    As  $\Gamma(V)$ is of finite type over $\F$ its Jacobson radical is 
equal to its nilradical by \cite[Theorem 5.5]{matsumura}, and we deduce that  $\beta^n=0$ for some positive integer $n$.  If we let 
$\co_{V,y}^\circ$ denote the completion of the   local ring of $V$ at $y$ then $\co_{V,y}^\circ\iso R_{\mathfrak{Y}}\otimes_{\ZZ_p} \F$,  
the $p$-Barsotti-Tate group $\mathfrak{G}$ is the base change of the $p$-Barsotti-Tate group $ A_{p^\infty}$ via
$$
\Gamma(V)\map{}\co_{V,y}^\circ\map{} R_\mathfrak{Y}/\mathfrak{p},
$$
and the Hasse-Witt invariant of $\mathfrak{G}$ is (for appropriate choice of generators of the exterior squares of the Lie algebras of 
$\mathfrak{G}$ and $\mathfrak{G}^{(p)}$) multiplication by the image of  $\beta$ in $R_{\mathfrak{Y}}/\mathfrak{p}$.    As the image of 
$\beta$ is $0$, we are done.
\end{proof}


\section{Vertical components}
\label{s:vertical}


In this section we will determine all vertical components of $\mathfrak{Y}$.  The final answer is simple: if $c_0=0$ then there are no 
vertical components, while if $c_0>0$ there are exactly two, each meeting $\mathfrak{M}_0$ transversely (in the sense that 
the intersection  number of Definition \ref{Def:components} is equal to $1$).  More difficult is the problem of determining 
the multiplicities of these 
components.  We will accomplish this (at least for $p>2$, a hypothesis which seems essential to the method)  by making heavy use of 
Zink's theory of \emph{windows} \cite{zink01}, an alternative to the theory of displays \cite{messing07,zink02} which is more 
amenable to explicit calculation.  The calculations are inspired by a method used by Kudla-Rapoport for studying  special cycles on 
unitary Shimura varieties \cite{kudla08}.


\subsection{Dieudonn\'e theory}
\label{ss:dieudonne}


  Let $W_0$ denote the completion of the strict henselization of $\co_{E_0}$ with respect to the homomorphism 
  $\psi:\co_{E_0}\map{}\F$ fixed in (\ref{psi fix}).  More concretely, 
$$
W_0\iso \begin{cases}
\ZZ_p & \mathrm{if\ } E_0/\Q_p\mathrm{\ is\ unramified} \\
\co_{E_0}\otimes_{\Z_p}\ZZ_p & \mathrm{if\ }  E_0/\Q_p\mathrm{\ is\ ramified}.
\end{cases}
$$
In either case $W_0$ is integrally closed in its fraction field $M_0$,  and we are given an embedding $\Psi:\co_{E_0}\map{}W_0$ 
which lifts $\psi$.    For any ring homomorphism $\phi:\co_{E_0}\map{}R$ we denote by $\overline{\phi}$ the map obtained by 
precomposing $\phi$ with the nontrivial Galois automorphism of $E_0/\Q_p$.   If  $R$ is an object of $\ProArt$ we let $W(R)$ be the 
ring of Witt vectors of $R$ and let $I_R\subset W(R)$ be the kernel of the natural surjection $W(R)\map{}R$.  For any $r\in R$ denote 
by $[r]\in W(R)$ the Teichmuller lift.  The ring $W(R)$ is equipped with a ring endomorphism $r\mapsto {}^F r$ and an endomorphism 
of  the underlying additive group $r\mapsto{} {}^V r$.

 We will use Zink's theory of \emph{displays} as in \cite{zink02}; a quick summary of the basics can be found in \cite{goren}.     A 
 display over an object $R$ of $\ProArt$ consists of a quadruple $\mathbf{D}=(P,Q,F,V^{-1})$ in which $P$ is finitely generated free 
 $W(R)$-module, $Q\subset P$ is a $W(R)$-submodule, and $F:P\map{}P$  and $V^{-1}:Q\map{}P$ are ${}^F$-linear operators.  
 There are addition properties which the quadruple $(P,Q,F,V^{-1})$ is to satisfy;   see \cite[Definitions 1 and 2]{zink02}.   If $(D,F,V)$ 
 is a Dieudonn\'e module over $\F$  with the property that $V$ is topologically nilpotent then  the quadruple $(D,VD,F,V^{-1})$ is a 
 display, and this construction establishes   an equivalence of categories between displays over $\F$ and Dieudonn\'e modules over 
 $\F$ on  which $V$ is topologically nilpotent.  By \cite[Theorem 9]{zink02} for any object $R$ of $\Art$ there is a (covariant) 
 equivalence of categories between displays over $R$ and $p$-Barsotti-Tate groups over $R$ with connected special fiber.   By the 
 comments after  \cite[Theorem 4]{zink02}  the Lie algebra of a connected $p$-Barsotti-Tate group is isomorphic to the Lie algebra of 
 its display $\mathbf{D}=(P,Q,F,V^{-1})$,  defined as the free $R$-module 
 $$
 \mathrm{Lie}(\mathbf{D})\define P/Q.
 $$
 The \emph{height} of $\mathbf{D}$ is $\mathrm{rank}_{W(R)}( P )$ and the \emph{dimension} of $\mathbf{D}$ is 
 $\mathrm{rank}_R(P/Q)$.  The Lie algebra of a Dieudonn\'e module $(D,F,V)$ is defined to be the Lie algebra of its associated 
 display $\mathrm{Lie}(D)=D/VD$.

Consider the \emph{standard supersingular Dieudonn\'e  module}  $(D_0,F,V)$ over $\F$ whose  underlying $\ZZ_p$-module is free 
on two generators $\{e_0,f_0\}$ with the operators $F$ and $V$ defined by
$$
Fe_0=f_0\qquad F f_0=pe_0
$$
and
$$
Ve_0 =f_0 \qquad V f_0 = pe_0.
$$
 If $E_0/\Q_p$ is unramified then $\Psi$ takes values in $W_0=\ZZ_p$, and  the \emph{normalized action} of $ \co_{E_0}$ on $D_0$ 
 is  defined by
$$
r * e_0=\Psi(r) e_0\qquad r*f_0=\overline{\Psi}(r) f_0
$$
for all $r \in\co_{E_0}$.  The \emph{anti-normalized}  action of $\co_{E_0}$ on $D_0$  is defined by 
$$
r* e_0=\overline{\Psi}(r) e_0\qquad r*f_0 =\Psi (r) f_0.
$$
The  action of $\co_{E_0}$ on $\mathrm{Lie}(D_0)$ induced by the normalized action is through $\psi$, and the action of $\co_{E_0}$ 
induced by the anti-normalized action is through $\overline{\psi} $.   If instead  $E_0/\Q_p$ is ramified let $\varpi_{E_0} \in\co_{E_0}$ 
be a root of an Eisenstein polynomial in $\Z_p[x]$, so that $\co_{E_0}=\Z_p\oplus \Z_p\varpi_{E_0}$.  Using the surjectivity of the 
trace $\Z_{p^2}\map{}\Z_p$ and norm $\Z_{p^2}^\times\map{}\Z_p^\times$ this Eisenstein polynomial has the form
$$
x^2-(a+a^\sigma)x+( a a^\sigma  -  b b^\sigma p)
$$
for some $a\in p\Z_{p^2}$ and $b\in\Z_{p^2}^\times$.  Having fixed such $a$ and $b$ we now let $\varpi_{E_0}$ act on $D_0$ by
$$
\varpi_{E_0}* e_0=a e_0+ b^\sigma f_0 \qquad \varpi_{E_0}*f _0= pb e_0 + a^\sigma f_0.
$$
This determines an action of $\co_{E_0}$ on $D_0$ which we call the \emph{normalized} action of $\co_{E_0}$.  The \emph{anti-
normalized} action is obtained by precomposing the normalized action with the nontrivial Galois automorphism of $\co_{E_0}$.  
Under either action $\co_{E_0}$ acts on $\mathrm{Lie}(D_0)$ through $\psi=\overline{\psi}$ (the unique $\Z_p$-algebra 
homomorphism $\co_{E_0}\map{}\F$.)

\begin{Lem}\label{Lem:standard display}
The covariant Dieudonn\'e module of the $p$-Barsotti-Tate group $\mathfrak{g}_0$ fixed in \S \ref{ss:functors} is $\co_{E_0}$-linearly 
isomorphic to the standard supersingular Dieudonn\'e module $D_0$ with its normalized $\co_{E_0}$-action.
\end{Lem}

\begin{proof}
As the action of $\co_{E_0}$ on the Lie algebra of $D_0$ is through $\psi$, this is a consequence of Corollary 
\ref{Cor:normalized action}.
\end{proof}

\begin{Lem}\label{Lem:flip action}
Suppose $\mathfrak{G}_0'$ and $\mathfrak{G}_0''$ are $p$-Barsotti-Tate groups of height two and dimension one over $\F$, each 
equipped with an $\co_{E_0}$-action.  Suppose further that there is an $\co_{E_0}$-linear isogeny  
$f:\mathfrak{G}_0'\map{} \mathfrak{G}_0''$ of degree $p^k$. Let $\phi:\co_{E_0}\map{}\F$ be the homomorphism giving the action of
 $\co_{E_0}$ on $\mathrm{Lie}(\mathfrak{G}_0'')$.
\begin{enumerate}
\item
If $k$ is even then $\co_{E_0}$ acts on $\mathrm{Lie}(\mathfrak{G}_0')$ through $\phi$.
\item
If $k$ is odd then $\co_{E_0}$ acts on $\mathrm{Lie}(\mathfrak{G}_0')$ through $\overline{\phi}$.
\end{enumerate}
\end{Lem}

\begin{proof}
If $E_0/\Q_p$ is ramified then there is a  unique  $\Z_p$-algebra homomorphism $\co_{E_0}\map{} \F$, and so there is nothing to 
prove.  Thus we assume that $E_0/\Q_p$ is unramified and let $D_0'$ and $D_0''$ be the covariant Dieudonn\'e modules of 
$\mathfrak{G}_0'$ and $\mathfrak{G}_0''$.   For simplicity we assume $\phi=\psi$ and fix, using Corollary \ref{Cor:normalized action}, 
an $\co_{E_0}$-linear isomorphism from $D_0''$ to the standard supersingular Dieudonn\'e module $D_0$ with its normalized action 
of $\co_{E_0}$.   The isogeny $f$ then identifies $D_0'$ with an $\co_{E_0}$-stable sub-Dieudonn\'e module $D_0'\subset D_0$ with 
$D_0/D_0'$ of length $k$ as a $\ZZ_p$-module.  Every such submodule has the form 
$$
D_0'= p^a\ZZ_p e_0 + p^{a+\epsilon} \ZZ_p f_0
$$
with $a\ge 0$, $\epsilon\in\{0,1\}$, and $k=2a+\epsilon$.  If $k$ is even then $\epsilon=0$ and 
$\mathrm{Lie}(\mathfrak{G}_0')=D_0'/VD_0'$ is generated by the image of $p^ae_0$.  Hence $\co_{E_0}$ acts on 
$\mathrm{Lie}(\mathfrak{G}_0')$ through $\psi$.    If $k$ is odd then $\epsilon=1$ and $\mathrm{Lie}(\mathfrak{G}_0')=D_0'/VD_0'$ is 
generated by the image of $p^{a+1}f_0$.  Hence $\co_{E_0}$ acts on $\mathrm{Lie}(\mathfrak{G}_0')$ through $\overline{\psi}$.  
\end{proof}

Let $\mathbf{d}_0=(P_0,Q_0,F,V^{-1})$ be the display associated associated to $(D_0,F,V)$, the 
\emph{standard supersingular display}.    Thus $P_0$ is the free $\ZZ_p$-module on $\{e_0,f_0\}$, $Q_0\subset P_0$ is free on 
the generators $\{pe_0,f_0\}$, and the operators $F$ and $V^{-1}$ are completely determined by the relations
$$
Fe_0=f_0\qquad V^{-1} f_0=e_0.
$$
As in \cite[(9)]{zink02} we encode this information in the \emph{displaying matrix} 
$\left(\begin{matrix} & 1 \\ 1  \end{matrix}\right)$ of $\mathbf{d}_0$.   Define a display $\mathbf{d}_0^\univ=(P_0,Q_0,F,V^{-1})$ over 
$R=\ZZ_p[[x_0]]$ as follows.  Let $P_0$ be the free $W(R)$ module on two generators $\{ e_0^\univ , f_0^\univ\}$,  let  
$Q_0\subset P_0$ be the $W(R)$-submodue 
$$
Q_0=I_R e_0^\univ+W(R) f_0^\univ,
$$  
and take the displaying matrix of $\mathbf{d}^\univ$ with respect to this basis to be
$\left(\begin{matrix}
[ x_0 ] & 1\\
1
\end{matrix}\right).$ Comparing with the displaying matrix of the standard supersingular display we see that the reduction of 
$\mathbf{d}_0^\univ$ to $\F$ is isomorphic to $\mathbf{d}_0$.

  For any  display $\mathbf{D}_0=(P_0,Q_0,F,V^{-1})$ over an object  $R$ of $\ProArt$ define a new display 
  $\mathbf{D}_0\otimes \Z_{p^2}= (P,Q, F, V^{-1})$ over $R$ in the obvious way;  thus
 $$
 P=P_0\otimes_{\Z_p}\Z_{p^2}\qquad Q=Q_0\otimes_{\Z_p}\Z_{p^2}
 $$  
and   $F$ and $V^{-1}$ are the  $\Z_{p^2}$-linear extensions  of the corresponding operators of $\mathbf{D}_0$.  In particular we 
define
$$
\mathbf{d}=\mathbf{d}_0\otimes\Z_{p^2}.
$$
Using the homomorphism $\ZZ_p\map{}W(\ZZ_p)$ of  \cite[\S 17.6]{hazewinkel}   we obtain a canonical $\Z_p$-algebra 
homomorphism $$\ZZ_{p}\map{}W(R)$$ for any $\ZZ_p$-algebra (and in particular any object of $\ProArt$) $R$, and hence a 
canonical  homomorphism $\Z_{p^2} \map{}W(R)$.  Let  $\epsilon_1$ and $\epsilon_2$ be the usual idempotents in
$$
  W(R)  \otimes_{\Z_p}  \Z_{p^2}\iso W(R)\times W(R)
$$
 indexed so that 
 $$ 
 (1 \otimes \alpha)\epsilon_1 =(\alpha \otimes 1) \epsilon_1
 \qquad
 (1 \otimes \alpha)\epsilon_2 =(\alpha^\sigma \otimes 1) \epsilon_2
 $$ 
 for all $\alpha \in\Z_{p^2}$.    Recalling that the underlying $\ZZ_p$-module of $\mathbf{d}_0$ has basis $\{e_0,f_0\}$, define a basis 
 $\{e_1,e_2,f_1,f_2\}$ of the underlying $\ZZ_p$-module of $\mathbf{d}$ by 
 $$
 e_i=\epsilon_i e_0\qquad f_i=\epsilon_i f_0.
 $$
 With respect to this basis $\mathbf{d}$ has displaying matrix
 \begin{equation}\label{ss display III}
\left( \begin{matrix}
 & & & 1 \\
 & & 1 \\ 
 & 1 \\
 1
\end{matrix} \right).
\end{equation}
The normalized action of $\co_{E_0}$  on  $\mathbf{d}_0$ determines an action of $\co_E$ on  $\mathbf{d}$.    If  $(D ,F,V)$ denotes 
the  Dieudonn\'e module associated to  $\mathbf{d}$ then  $D$ is the free  $\ZZ_p$-module on the generators $\{e_1,e_2,f_1,f_2\}$,  
the operators $F$ and $V$ both satisfy
$$
e_1\mapsto f_2 \qquad e_2\mapsto f_1 \qquad f_1\mapsto pe_2 \qquad f_2\mapsto pe_1,
$$
and the action of $\Z_{p^2}$ on $D$ is given by
\begin{equation}\label{unr display action}
\alpha* e_1 = \alpha e_1  \qquad \alpha* e_2 = \alpha^\sigma  e_2 \qquad \alpha* f_1=\alpha f_1  \qquad 
 \alpha* f_2=\alpha^\sigma f_2.
\end{equation}

Now abbreviate $R=\ZZ_p[[x_1,x_2]]$ and define a display  $\mathbf{d}^\univ  = (P,Q,F,V^{-1})$  over $R$ as follows.  The underlying  
$W(R)$-module $P$ of $\mathbf{d}^\univ$ is free on four generators  $\{ e^\univ_1, e^\univ_2 , f^\univ_1 , f^\univ_2\}$.  The  $W(R)$-
submodule   $Q\subset P$ is defined by
$$
Q=I_{R} e^\univ_1+I_{R} e^\univ_2+W(R) f^\univ_1+W(R) f^\univ_2, 
$$
and the displaying matrix of $\mathbf{d}^\univ$ with respect to this basis is
\begin{equation}\label{ss display II}
\left(\begin{matrix}
 & [x_1] & &\ 1\  \\
 [ x_2] & & 1 \\
  & 1 \\
 1
\end{matrix}\right).
\end{equation}
We endow $\mathbf{d}^\univ$ with an action of $\Z_{p^2}$ using the formulas (\ref{unr display action}), replacing $e_i$ with 
$e_i^\univ$ and $f_i$ with $f_i^\univ$.  One can check that setting $x_2=x_1=x_0$ in (\ref{ss display II}) recovers the displaying 
matrix of $\mathbf{d}_0^\univ\otimes\Z_{p^2}$, and hence
  $\mathbf{d}_0^\univ\otimes\Z_{p^2}$ is $\Z_{p^2}$-linearly  isomorphic to  the base change of $\mathbf{d}^\univ$ through the 
  isomorphism
$$
 \ZZ_p[[x_1,x_2]]/(x_1-x_2) \map{x_i\mapsto x_0}  \ZZ_p[[x_0]].
$$
Comparing (\ref{ss display III}) with (\ref{ss display II}), the base change of $\mathbf{d}^\univ$ to the residue field of $\Z_p[[x_1,x_2]]$ 
is $\Z_{p^2}$-linearly isomorphic to $\mathbf{d}=\mathbf{d}_0\otimes\Z_{p^2}$,  and hence (by  Lemma \ref{Lem:standard display})  
is $\Z_{p^2}$-linearly  isomorphic to the  display of the $p$-Barsotti-Tate group $\mathfrak{g}\iso \mathfrak{g}_0\otimes\Z_{p^2}$.  
Thus we have a commutative diagram of functors on $\Art$
$$
\xymatrix{
{\Spf(\ZZ_p[[x_0]]) }  \ar[d]  \ar[r] & {\mathfrak{M}_0} \ar[d]^{\otimes\Z_{p^2}} \\
{\Spf(\ZZ_p[[x_1,x_2]])  } \ar[r]  &  \mathfrak{M}
}
$$
in which the top horizontal arrow sends a morphism $\ZZ_p[[x_0]]\map{} R$ in $\ProArt$ to the $p$-divisible group associated to the 
base change of $\mathbf{d}^\univ_{0}$ to $R$,  the bottom horizontal arrow sends a morphism $\ZZ_p[[x_1,x_2]]\map{} R$ to the $p$-
divisible group associated to the base change of  $\mathbf{d}^\univ$ to $R$, and the vertical arrow on the left is $x_i\mapsto x_0$.

\begin{Prop}\label{Prop:universal display I}
The horizontal arrows in the above diagram are isomorphisms.
\end{Prop}

\begin{proof}
Consider the top  horizontal arrow.  By \cite[Theorem 1.5.3(3)]{dok} it suffices to prove that the induced map on tangent spaces
$$
\Hom_{\ProArt}(\ZZ_p[[x_0]], \F[\epsilon]) \map{} \mathfrak{M}_0(\F[\epsilon])
$$
is a bijection, and this is proved  by Zink  \cite[\S 2.2]{zink02} using his  deformation theory of displays.  The proof for the bottom 
horizontal arrow is similar; details can be found in \cite[\S 6.11]{goren} and \cite{goren00}.
\end{proof}

To paraphrase the proposition: we  may identify $R_{\mathfrak{M}}\iso \ZZ_p[[x_1,x_2]]$ in such a way that the closed immersion 
$\mathfrak{M}_0\map{} \mathfrak{M}$  is identified with the closed immersion
$$
\Spf(\ZZ_p[[x_0]]) \map{}\Spf(\ZZ_p[[x_1,x_2]])
$$
defined by $x_i\mapsto x_0$.


\subsection{The Hasse-Witt locus}
\label{ss:hasse-witt}


Using Proposition \ref{Prop:universal display I} we  identify $$R_\mathfrak{M}\iso \ZZ_p[[x_1,x_2]],$$
so that the ideal defining  the closed formal subscheme $\mathfrak{M}_0$ is generated by $x_1-x_2$. The display of the universal 
deformation of $\mathfrak{g}$ is given by the displaying matrix (\ref{ss display II}), and the Lie algebra of the universal deformation is 
the free $\ZZ_p[[x_1,x_2]]$-module on the generators $\{e_1^\univ, e_2^\univ\}$.  Using \cite[Example 23]{zink02}, after base change 
from $\ZZ_p[[x_1,x_2]]$ to $\F[[x_1,x_2]]$ the Verschiebung morphism on Lie algebras can be read off from the matrix 
(\ref{ss display II}), and is given by 
$$
e_1^\univ\mapsto x_2 e_2^\univ \qquad e_2^\univ\mapsto x_1 e_1^\univ.
$$
Thus the Hasse-Witt locus $\mathfrak{M}_{\mathrm{HW}}\map{}\mathfrak{M}$ (in the sense of \S \ref{ss:hilbert-blumenthal}) is seen to 
be the closed formal subscheme defined by the  ideal $I_{\mathrm{HW}}=(p,x_1x_2)$.    We find that $\mathfrak{M}_{\mathrm{HW}}$ 
has two irreducible components which we denote by 
\begin{eqnarray*}
\mathfrak{C}^\vertical_1&=&\Spf(R_{\mathfrak{M}}/(p,x_2))  \iso \Spf(\F[[x_1]]) \\
\mathfrak{C}^\vertical_2&=&\Spf(R_{\mathfrak{M}} / (p,x_1) ) \iso \Spf(\F[[x_2]]).
\end{eqnarray*}
Each of $\mathfrak{C}_i^\vertical$ meets $\mathfrak{M}_0$ transversely in the sense that 
\begin{equation}\label{transverse}
\mathfrak{C}_i^\vertical \times_{\mathfrak{M}}\mathfrak{M}_0\iso \Spf(\F).
\end{equation}  
According to Proposition \ref{Prop:hasse-witt} any vertical component of $\mathfrak{Y}$ is contained in the Hasse-Witt locus, and so 
must be one of $\mathfrak{C}_i^\vertical$.

We will examine the component $\mathfrak{C}_1^\vertical$.  Let $(\mathfrak{G}_1^\vertical,\rho_1^\vertical)\in \mathfrak{M}(\F[[x_1]])$ 
be the restriction to $\mathfrak{C}_1^\vertical$ of the universal deformation of $\mathfrak{g}$ over $\mathfrak{M}$.  We wish to 
determine which endomorphisms of $\mathfrak{g}$ lift to endomorphisms of $\mathfrak{G}_1^\vertical$.   Let 
$$\mathbf{D}_1^\vertical=(P_1^\vertical,Q_1^\vertical,F,V^{-1})$$ be the display over $\F[[x_1]]$ corresponding to 
$\mathfrak{G}_1^\vertical$.  According to \S \ref{ss:dieudonne} the $W(\F[[x_1]])$-module $P_1^\vertical$ is free on the generators 
$\{e_1,e_2,f_1,f_2\}$, the submodule $Q_1^\vertical$ is 
$$
I_{\F[[x_1]]} e_1 + I_{\F[[x_1]] }e_2 + W(\F[[x_1]]) f_1 + W(\F[[x_1]]) f_2,
$$
and the operators $F$ and $V^{-1}$  are determined by the displaying matrix over $W(\F[[x_1]])$ with respect to the basis 
$\{e_1,e_2,f_1,f_2\}$ 
\begin{equation*}
\left(\begin{matrix}
 & [x_1] & &\ 1\  \\
 0 & & 1 \\
  & 1 \\
 1
\end{matrix}\right).
\end{equation*}
Let $\mathrm{Fr}:\ZZ_p[[x_1]]\map{}\ZZ_p[[x_1]]$ be the unique ring homomorphism which is the Frobenius on $\ZZ_p$ and satisfies 
$x_1\mapsto x_1^p$.  The \emph{window} \cite{zink01} associated to $\mathbf{D}_1^\vertical$ with respect to the \emph{frame} 
$\ZZ_p[[x_1]]\map{}\F[[x_1]]$ is the triple $(M_1^\vertical,N_1^\vertical,\Phi)$ where $M_1^\vertical$ is the free $\ZZ_p[[x_1]]$-module 
on the generators  $\{e_1,e_2,f_1,f_2\}$, the submodule $N_1^\vertical$ is generated by $\{p e_1,pe_2,f_1,f_2\}$, and 
$\Phi :M_1^\vertical\map{}M_1^\vertical$ is the $\mathrm{Fr}$-linear map satisfying
$$
e_1\mapsto f_2 \qquad e_2\mapsto x_1 e_1+ f_1\qquad f_1\mapsto pe_2 \qquad f_2\mapsto pe_1.
$$
The action of $\Z_{p^2}$ on $(M_1^\vertical,N_1^\vertical,\Phi)$ is by the rule (\ref{unr display action}).  
The window of $\mathfrak{g}$ with respect to the frame $\ZZ_p\map{}\F$ is obtained as the base change $(M,N,\Phi)$ of the triple 
$(M_1^\vertical, N_1^\vertical,\Phi)$ via the  map $\ZZ_p[[x_1]] \map{} \ZZ_p$ defined by $x_1\mapsto 0$, and every 
$\Gamma \in \End_{\Z_{p^2}}(\mathfrak{g})$ is determined by its matrix with respect to the ordered basis $\{e_1,f_1,e_2,f_2\}$ of $M$.  
The condition that $\Gamma$ commutes with the $\Z_{p^2}$-action is equivalent to this matrix having the block diagonal form
$$
\Gamma=\left(\begin{matrix} Y  \\ & Z \end{matrix}\right),
$$
while the condition that $\Gamma$ commutes with $\Phi$ and preserves the submodule $N$ is equivalent to $Y$ and $Z$ having the 
form
$$
Y = \left( \begin{matrix} a& pb \\ c& d  \end{matrix}\right)
\qquad
Z = \left( \begin{matrix}   d^\sigma & pc^\sigma \\ b^\sigma & a^\sigma    \end{matrix}\right)
$$
with $a,b,c,d\in\Z_{p^2}$.  The subring $\End(\mathfrak{g}_0)\subset \End_{\Z_{p^2}}(\mathfrak{g})$ consists of those $\Gamma$ for 
which $d=a^\sigma$ and $c=b^\sigma$.

For each $k\ge 0$ set $$A[k]=\ZZ_p[[x_1]]/(x_1^{p^k}) \qquad R[k]=\F[[x_1]]/(x_1^{p^k})$$ and let $(M[k],N[k],\Phi)$ be the base 
change of $(M_1^\vertical , N_1^\vertical , \Phi)$ via $\ZZ_p[[x_1]]\map{}A[k]$, so that $(M[k],N[k],\Phi)$ is the window associated to 
$\mathfrak{G}^\vertical_{1/R[k]}$ with respect to the frame $A[k]\map{}R[k]$.    We will lift the endomorphism $\Gamma[0]=\Gamma$ of 
$(M[0],N[0],\Phi)=(M,N,\Phi)$ to a quasi-endomorphism $\Gamma_1^\vertical$ of $(M_1^\vertical,N_1^\vertical,\Phi)$ by successively 
lifting from $A[k]$ to $A[k+1]$.  With respect to the ordered basis $e_1,f_1,e_2,f_2$ any $\Z_{p^2}$-linear quasi-endomorphism 
$\Gamma[k]$ of $(M[k],N[k],\Phi)$ 
must have the form 
$$
\Gamma[k]=\left(\begin{matrix} Y[k]  \\ & Z[k] \end{matrix}\right),
$$
where $Y[k]$ and $Z[k]$ are $2\times 2$ matrices with entries in $A[k]\otimes_{\ZZ_p}\QQ_p$ which satisfy
\begin{equation}\label{prerecursion}
Y[k] \cdot \left(\begin{matrix} x_1 & p \\ 1 \end{matrix} \right) = \left(\begin{matrix}  x_1 & p \\ 1\end{matrix} \right)  \cdot \mathrm{Fr}(Z[k])
\qquad
Z[k] \cdot \left(\begin{matrix} & p \\ 1 \end{matrix} \right) = \left(\begin{matrix} & p \\ 1 \end{matrix} \right)  \cdot \mathrm{Fr}(Y[k])
\end{equation}
(these conditions are equivalent to $\Gamma[k]$ commuting with $\Phi$), and such a $\Gamma[k]$ is an endomorphism if and only if 
$Y[k]$ and $Z[k]$ have entries in $A[k]$ and each of the upper right entries is in $pA[k]$ (so that the submodule $N[k]\subset M[k]$ is 
preserved).  The crucial observation is the following: the ring homomorphism $\mathrm{Fr}:A[k+1]\map{}A[k+1]$ factors through the 
quotient map $A[k+1]\map{}A[k]$, so that there is a commutative diagram
$$
\xymatrix{
{A[k+1]} \ar[r]^{\mathrm{Fr}}  \ar[d] & {A[k+1]} \\
{A[k]} \ar[ur]_{\mathrm{Fr}}.
}
$$
Replacing $k$ by $k+1$ in (\ref{prerecursion}) and then solving for $Y[k+1]$ and $Z[k+1]$ we find that to successively lift 
$\Gamma[0]$ to  a quasi-endomorphism of each $(M[k],N[k],\Phi)$ is equivalent to solving the recursion relations
\begin{eqnarray}\label{recursion}
Y[k+1]  &=& \frac{1}{p} \left(\begin{matrix}  x_1& p \\ 1\end{matrix} \right)  \cdot
 \mathrm{Fr}(Z[k])  \cdot \left(\begin{matrix}  & p \\  1 & -x_1  \end{matrix} \right)   \\
Z[k+1]  &=& \frac{1}{p}\left(\begin{matrix} & p \\ 1 \end{matrix} \right)  \cdot
 \mathrm{Fr}(Y[k])  \cdot \left(\begin{matrix} & p \\ 1 \end{matrix} \right) \nonumber
\end{eqnarray}
 in $A[k]\otimes_{\ZZ_p}\QQ_p$.   By trial and error one can explicitly solve the recursion (\ref{recursion}). 
  If we define matrices $Y_1^\vertical$ and $Z_1^\vertical$ in $\ZZ_p[[x_1]]\otimes_{\ZZ_p}\QQ_p$ by 
\begin{eqnarray*}
Y_1^\vertical   &=&  \left(\begin{matrix}  a & pb   \\ c & d  \end{matrix} \right)  +  f(x_1)
 \left(\begin{matrix}  c & d-a  \\  0  & -c  \end{matrix} \right)   - g(x_1) \left(\begin{matrix} 0 & c\\ 0 & 0   \end{matrix}\right)  \\ 
Z_1^\vertical  &=&   \left(\begin{matrix} d^\sigma & pc^\sigma  \\   b^\sigma & a^\sigma \end{matrix} \right)  + 
 f(x_1^p)   \left(\begin{matrix} -c^\sigma &  0 \\   \frac{(d - a)^\sigma}{p}  & c^\sigma \end{matrix} \right) - g(x_1^p) 
 \left(\begin{matrix} 0& 0\\ \frac{c^\sigma}{p} & 0 \end{matrix}\right)\nonumber
\end{eqnarray*}
where
$$
f(x) = x^{p^0}+x^{p^2} + x^{p^4} +x^{p^6}+x^{p^8}+ \cdots
$$
and
\begin{eqnarray*}\lefteqn{
g(x) = x^{p^0} (x^{p^0})+x^{p^2}(2x^{p^0}+x^{p^2}) + x^{p^4}(2x^{p^0}+2x^{p^2}+x^{p^4}) } \\
& & + x^{p^6}(2x^{p^0}+2x^{p^2}+2x^{p^4} + x^{p^6}) + x^{p^8}(2x^{p^0}+2x^{p^2}+2x^{p^4} + 2x^{p^6} + x^{p^8}) + \cdots
\end{eqnarray*}
then $Y[k]$ and $Z[k]$ are the images of $Y_1^\vertical$ and $Z_1^\vertical$ under the quotient map 
$\ZZ_p[[x_1]]\otimes_{\ZZ_p}\QQ_p\map{} A[k]\otimes_{\ZZ_p}\QQ_p$.  From this it is clear that $\Gamma[0]=\Gamma$ lifts from an 
endomorphism of $\mathfrak{g}$ to a quasi-endomorphism of $\mathfrak{G}_1^\vertical$ given by
\begin{equation}\label{vertical quasi lifts}
\Gamma_1^\vertical = \left(\begin{matrix} Y_1^\vertical & \\ & Z_1^\vertical \end{matrix}\right)
\end{equation}
 and that this quasi-endomorphism is integral if and only if $a-d\in p\Z_{p^2}$ and $c\in p\Z_{p^2}$.

\begin{Prop}\label{Prop:two components}
Recall from \S \ref{ss:functors} that $c_0$ is defined by $\Z_p[\gamma_0]=\Z_p + p^{c_0}\co_{E_0}$.
\begin{enumerate}
\item
If $c_0=0$ then $\mathfrak{Y}$ has no vertical components.
\item
If $c_0>0$ then $\mathfrak{Y}$ has exactly two vertical components, $\mathfrak{C}_1^\vertical$ and $\mathfrak{C}_2^\vertical$, each 
satisfying $I_{\mathfrak{M}}(\mathfrak{C}_i^\vertical, \mathfrak{M}_0)=1$.
\end{enumerate}
\end{Prop}

\begin{proof}
Consider first the closed formal subscheme $\mathfrak{C}_1^\vertical=\Spf(\F[[x_1]])$ of $\mathfrak{M}$ defined by the ideal 
$(p,x_2)\subset R_\mathfrak{M}$.  Let $(\mathfrak{G}_1^\vertical,\rho_1^\vertical)$ be the restriction of the universal deformation of 
$\mathfrak{g}$  to $\mathfrak{C}_1$.  As above, the action of  the endomorphism $j(\gamma)\in\End_{\Z_{p^2}}(\mathfrak{g})$ is 
determined by its action on the window $(M,N,\Phi)$ associated to $\mathfrak{g}$, and is given by a matrix of the form
$$
\Gamma= \left(\begin{matrix}   a & pb \\ b^\sigma & a^\sigma \\ & & a & pb \\ & & b^\sigma & a^\sigma    \end{matrix}\right)
$$
with $a,b\in\Z_{p^2}$.  By what was said above, $\mathfrak{C}_1^\vertical$ is contained in $\mathfrak{Y}$ if and only if the unique lift 
of $j(\gamma)$ to a quasi-endomorphism of $\mathfrak{G}_1^\vertical$ is an endomorphism, which is equivalent to  $p$ dividing both 
$a- a^\sigma$ and $b$.  Using the fact that $\Z_p[\gamma_0]$ is isomorphic to the $\Z_p$-subalgebra of $M_2(\Z_{p^2})$ generated 
by $\left(\begin{matrix} a & pb \\b^\sigma & a^\sigma \end{matrix}\right)$ it is easy to see that 
$$
p\mid (a-a^\sigma)  \mathrm{\ and\ }p\mid b \iff \Z_p[\gamma_0]\subset \Z_p + p\co_{E_0}\iff c_0>0.
$$
The same argument with $\mathfrak{C}_1^\vertical$ replaced by $\mathfrak{C}_2^\vertical$ shows that each of the two components 
of the Hasse-Witt locus of $\mathfrak{M}$ is contained in $\mathfrak{Y}$ if and only if $c_0>0$.  The final claim concerning the 
intersection multiplicity is just a restatement of (\ref{transverse}).
\end{proof}


\subsection{Vertical multiplicities: $E_0$ unramified}
\label{ss:unramified multiplicity}


Assume that $E_0/\Q_p$ is unramified and fix some $\eta\in E_0$ such that $\co_{E_0}=\Z_p[\eta]$.  Thus
$$
\Z_p[\gamma_0] = \Z_p[p^{c_0}\eta] \qquad \Z_{p^2}[\gamma] = \Z_{p^2}[p^{c_0}\eta]
$$
and 
$$
U\define\Psi(\eta) - \overline{\Psi}(\eta) \in (\ZZ_p)^\times.
$$
Identify the window  of $\mathfrak{g}$ equipped with its  action induced by  $j: \co_E\map{}\End_{\Z_{p^2}}(\mathfrak{g})$ with the 
window $(M,N,\Phi)$ of  \S \ref{ss:hasse-witt} equipped with the normalized action in the sense of  \S \ref{ss:dieudonne}.   The action 
of $\eta$ on  $(M,N,\Psi)$ with respect to the  ordered basis $\{e_1,f_1,e_2,f_2\}$  is given by the matrix
$$
\Gamma= \left(\begin{matrix}   Y &  \\  & Z    \end{matrix}\right)
$$
where
$$
Y=  \left(\begin{matrix}   \Psi(\eta) &  \\  & \overline{\Psi}(\eta) \end{matrix} \right)
 \qquad Z = \left(\begin{matrix} \Psi(\eta) &  \\  & \overline{\Psi}(\eta)    \end{matrix}\right)
$$ 
and $\Psi,\overline{\Psi}:\co_{E_0}\map{}\ZZ_p$ are as in \S \ref{ss:dieudonne}.    As in \S \ref{ss:hasse-witt} we let 
$(\mathfrak{G}_1^\vertical,\rho_1^\vertical)\in \mathfrak{M}(\F[[x_1]])$ be the restriction of the universal $p$-Barsotti-Tate group over 
$\mathfrak{M}$ to the closed formal subscheme
$$
\mathfrak{C}_1^\vertical=\Spf(\F[[x_1]])
$$ 
and let $(M_1^\vertical,N_1^\vertical,\Phi)$ be the associated window over $\F[[x_1]]$ with respect to the  frame 
$\ZZ_p[[x_1]]\map{}\F[[x_1]]$.  We saw in \S \ref{ss:hasse-witt} (especially (\ref{vertical quasi lifts})) that  the endomorphism $\eta$ of 
$\mathfrak{g}$ lifts to the quasi-endomorphism of $\mathfrak{G}_1^\vertical$ whose action on $(M_1^\vertical,N_1^\vertical,\Phi)$ is 
given by the matrix
$$
\Gamma_1^\vertical= \left(\begin{matrix}  Y_1^\vertical \\ &Z_1^\vertical   \end{matrix}\right) 
$$
where
$$
Y_1^\vertical = \left(\begin{matrix} 
  \Psi(\eta) & -U f(x_1)   \\ 0 & \overline{\Psi}(\eta) \end{matrix} \right)  
  \qquad
 Z_1^\vertical = \left(\begin{matrix}  \Psi(\eta) & 0  \\  \frac{U}{p} f(x_1^p)  & \overline{\Psi}(\eta)    \end{matrix}\right) 
$$
and $f(x_1)=x_1^{p^0}+x_1^{p^2}+x_1^{p^4}+\cdots$.   Note in particular that $p\eta$ lifts to an endomorphism of 
$\mathfrak{G}_1^\vertical$.

We next attempt to lift the quasi-endomorphism $\Gamma_1^\vertical$ to a natural family of deformations of 
$\mathfrak{G}_1^\vertical$.  For every $k\ge 0$ set
\begin{equation}\label{thick-frame}
R[k]=\ZZ_p[[x_1,x_2]]/( p^{2k+1}, x_2^{p^k}) \qquad
A[k]=\ZZ_p[[x_1,x_2]]/(  x_2^{p^k}) .
\end{equation}
If we equip $A[k]$ with the unique continuous  ring homomorphism $\mathrm{Fr}:A[k]\map{}A[k]$ which satisfies $x_i\mapsto x_i^p$ 
and whose restriction to $\ZZ_p$ is the usual Frobenius   then $A[k]\map{}R[k]$ is a frame in Zink's sense. Denote by 
$\mathfrak{G}[k]$ the base change to $R[k]$ of the universal $p$-Barsotti-Tate group over $R_\mathfrak{M}\iso \ZZ_p[[x_1,x_2]]$, and 
note that  $\mathfrak{G}[0]$ is simply the $p$-Barsotti-Tate group  $\mathfrak{G}_1^\vertical$ over $\F[[x_1]]$.
 There is a canonical isomorphism $\mathfrak{G}[k+1]_{/R[k]}\iso \mathfrak{G}[k]$ and so  each $\mathfrak{G}[k]$ is naturally a 
 deformation of $\mathfrak{G}_1^\vertical$.     Starting from the display (\ref{ss display II}) of the universal deformation of 
 $\mathfrak{g}$ to $R_{\mathfrak{M}}$, base-changing from $R_\mathfrak{M}$ to $R[k]$, and applying Zink's equivalence of 
 categories between displays over $R[k]$ and windows over $R[k]$, we find that the window associated to $\mathfrak{G}[k]$ is the 
 triple $(M[k] , N[k],\Phi)$ in which $M[k]$ is the free $A[k]$-module on generators $\{e_1,e_2,f_1,f_2\}$,  the submodule 
 $N[k]\subset M[k]$ is generated by $\{p^{2k+1}e_1, p^{2k+1}e_2, f_1,f_2\}$, and $\Phi : M[k] \map{}M[k]$ is the $\mathrm{Fr}$-linear 
 map determined by
$$
e_1\mapsto x_2e_2+f_2\qquad e_2\mapsto x_1e_1+f_1\qquad f_1\mapsto p e_2\qquad f_2\mapsto p e_1.
$$
As always, the action of $\Z_{p^2}$ is by (\ref{unr display action}).  By the previous paragraph the quasi-endomorphism of 
$\mathfrak{G}[0]\iso \mathfrak{G}_1^\vertical$ induced by lifting $\eta$ from $\mathfrak{g}$ to $\mathfrak{G}_1^\vertical$ 
corresponds to the quasi-endomorphism of  the window $(M[0],N[0],\Phi)$  whose matrix with respect to the ordered basis 
$\{e_1,f_1,e_2,f_2\}$ is
$$
\Gamma[0]= \left(\begin{matrix} Y[0] & \\ & Z[0] \end{matrix}\right)
$$
where $Y[0], Z[0]\in M_2(A[0])\otimes_{\ZZ_p}\QQ_p$ are defined by 
\begin{equation}\label{initial conditions}
Y[0] = \left(\begin{matrix}  \Psi(\eta) & -U f(x_1) \\ 0 & \overline{\Psi}(\eta)  \end{matrix}\right) \qquad Z[0] 
= \left(\begin{matrix} \Psi(\eta) & 0 \\ \frac{U}{p} f(x_1^p) & \overline{\Psi}(\eta) \end{matrix}\right).
\end{equation}
By the argument used in \S \ref{ss:hasse-witt}, the quasi-endomorphism $\Gamma[0]$ of $(M[0],N[0],\Phi)$ lifts to the quasi-
endomorphism of $(M[k],N[k],\Phi)$ given by the matrix
$$
\Gamma[k]=\left(\begin{matrix} Y[k] & \\ & Z[k] \end{matrix}\right)
$$
where $Y[k], Z[k]\in M_2(A[k])\otimes_{\ZZ_p}\QQ_p$ satisfy the recursion relations
\begin{eqnarray}\label{recursion II}
Y [k+1] &=& \frac{1}{p} \left(\begin{matrix}  x_1& p \\ 1\end{matrix} \right)  \cdot \mathrm{Fr}(Z[k])  \cdot 
\left(\begin{matrix}  & p \\  1 & -x_1  \end{matrix} \right)   \\
Z [k+1]  &=& \frac{1}{p}\left(\begin{matrix} x_2 & p \\ 1 \end{matrix} \right)  \cdot \mathrm{Fr}(Y[k])  \cdot 
\left(\begin{matrix} & p \\ 1 & -x_2 \end{matrix} \right). \nonumber
\end{eqnarray}

The  matrices $Y[0]$ and $Z[0]$ admit obvious lifts to  $M_2(\ZZ_p[[x_1,x_2]])\otimes_{\ZZ_p} \QQ_p$, defined again by the equations 
(\ref{initial conditions}).  We then lift each $Y[k]$ and $Z[k]$ to a matrix in $M_2(\ZZ_p[[x_1,x_2]])\otimes_{\ZZ_p}\QQ_p$ in the unique 
way for which the relations (\ref{recursion II}) continue to hold.  For $k\ge 1$ set
\begin{equation}\label{init}
Y_k=p^k(Y[k]-Y[k-1])
\qquad
Z_k=p^k(Z[k]-Z[k-1])
\end{equation}
so that (by directly computing $Y[1]$ and $Z[1]$)
\begin{eqnarray*}
Y_1 &=& 0 \\
Z_1 &=& U\cdot \left(\begin{matrix} x_2f(x_1^p) & -px_2-x_2^2f(x_1^p) \\ & - x_2 f(x_1^p) \end{matrix} \right)
\end{eqnarray*}
and
\begin{eqnarray*}
Y_{k+1}  &=&  \left(\begin{matrix}  x_1& p \\ 1\end{matrix} \right)  \cdot \mathrm{Fr}(Z_k)  
\cdot \left(\begin{matrix}  & p \\  1 & -x_1  \end{matrix} \right)   \\
Z_{k+1}  &=&  \left(\begin{matrix} x_2 & p \\ 1 \end{matrix} \right)  \cdot \mathrm{Fr}(Y_k)
  \cdot \left(\begin{matrix} & p \\ 1 & -x_2 \end{matrix} \right).
\end{eqnarray*} 
For notational consistency we also define $Y_0=Y[0]$ and $Z_0=Z[0]$, so   that for $k\ge 0$
\begin{eqnarray*}
p^kY[k] &=&  p^k Y_0 + p^{k-1} Y_1 + p^{k-2}  Y_2 + \cdots + p Y_{k-1}+  Y_k \\
p^kZ[k]  &=&  p^k Z_0 + p^{k-1} Z_1 + p^{k-2}  Z_2 + \cdots + p Z_{k-1}+  Z_k.
\end{eqnarray*}
It is clear from the recursion relations and initial conditions that $Y_k$ and $Z_k$ have entries in $\ZZ_p[[x_1,x_2]]$ for $k\ge 1$, that 
$Y_k=0$ for $k$ odd, and that $Z_k=0$ for $k$ even.  Let $y_k$ and $z_k$ be the upper right entries of $Y_k$ and $Z_k$, 
respectively and  define (for $\ell\ge 0$)
$$
\epsilon_+(\ell) = \prod_{ \substack{  0 \le i <   \ell  \\ i\mathrm{\ even} } } x_2^{2p^i}
\qquad
\epsilon_-(\ell) = \prod_{ \substack{  0 \le i<   \ell  \\ i\mathrm{\ odd} } } x_2^{2p^i}.
$$
 From the fact that $Z_1$ is divisible by $x_2$ one deduces that $Z_\ell$ is divisible by $x_2^{p^{\ell-1}}$ for $\ell$ odd and positive, 
 and that $Y_\ell$ is divisible by $x_2^{p^{\ell-1}}$ for $\ell$ even and positive.   Furthermore one can easily compute the images of 
 $y_\ell$ and $z_\ell$  in $\F[[x_1,x_2]]$ by solving the above recursion relations modulo $p$.  One finds that
\begin{eqnarray}\label{structure I}
y_0 &=& u_0 \epsilon_-(0) \\ \nonumber
z_1 &=& u_1 \epsilon_+(1) + p x_2 g_1 \\ \nonumber
y_2 &=& u_2 \epsilon_-(2) + p x_2^p g_2 \\ \nonumber
z_3 &=& u_3 \epsilon_+(3) + p x_2^{p^2} g_3 \\\nonumber
y_4 &=& u_4 \epsilon_-(4) +p x_2^{p^3} g_4 \\\nonumber
z_5&=& u_5 \epsilon_+(5) + px_2^{p^4} g_5\\ \nonumber
& \vdots & 
\end{eqnarray}
for some $g_k\in \ZZ_p[[x_1,x_2]]$ and some $u_k\in\ZZ_p[[x_1]]$ with nonzero image in $\F[[x_1]]$.   As already noted
\begin{equation}\label{structure II}
\ell\mathrm{\ odd\ }\implies y_\ell=0\qquad \ell\mathrm{\ even\ }\implies z_\ell=0.
\end{equation}
Let $\alpha_k, \beta_k\in \ZZ_p[[x_1,x_2]]$ be the upper right entries of $p^kY[k]$ and $p^kZ[k]$, respectively, so that
\begin{eqnarray}\label{structure III}
\alpha_k &=&  p^ky_0  + p^{k-1} y_1 + p^{k-2}  y_2 + \cdots + p y_{k-1}+  y_k \\
\beta_k  &=&   p^k z_0+p^{k-1} z_1 + p^{k-2}  z_2 + \cdots + p z_{k-1}+  z_k . \nonumber
\end{eqnarray}

Suppose $k\ge 1$.  By what has been said above the endomorphism $p^k\eta$ of $\mathfrak{G}[0]$ lifts to a quasi-endomorphism  of 
$\mathfrak{G}[k]$ which we denote in the same way.  Our interest in the constants $\alpha_k,\beta_k\in \ZZ_p[[x_1,x_2]]$ comes from 
the following fundamental fact.

\begin{Prop}\label{Prop:small locus}
Suppose $k\ge 1$ and  let $R[k] \map{}S$ be the maximal quotient of $R[k]$ for which the quasi-endomorphism $p^k\eta$ of 
$\mathfrak{G}[k]_{/S}$  is an endomorphism.  Then  $$S=R[k]/(\alpha_k,\beta_k).$$
\end{Prop}

\begin{proof}
 As base change for displays is easier to work with than base change for windows, we first use Zink's equivalence \cite{zink01} 
 between windows and displays to construct the display associated to  $(M[k] ,N[k] ,\Phi)$.   As in the introduction to \cite{zink01} there 
 is a continuous ring homomorphism $A[k]\map{} W(R[k])$ lifting the map $A[k] \map{}R[k] $.  This map takes $x_i\in A[k]$ to the 
 Teichmuller lift $[x_i]\in W(R[k])$.  The display $(P,Q,F,V^{-1})$ associated to $(M[k],N[k],\Phi)$ consists of the free $W(R[k])$-module
$$
P=M[k] \otimes_{A[k] } W(R[k])
$$
the submodule 
$$
Q= I_{R[k]}e_1+W(R[k] ) f_1+ I_{R[k]}e_2 + W(R[k])f_2 \subset P,
$$
 and two operators $F:P\map{}P$ and $V^{-1}:Q\map{}P$ whose definition does not concern us.  If we let $J=(p^{2k+1})$ denote the 
 kernel of $A[k]\map{}R[k]$ then there is a canonical isomorphisms of $R[k]$-modules
$$
P / I_{R[k]} P\iso M[k] / J M[k]
$$ 
which identifies  $Q/I_{R[k]}P$ with $N[k] /JM[k]$.  The quasi-endomorphism 
$p^k\eta$ of $(M[k] ,N[k],\Phi)$ is represented by the matrix $p^k\Gamma[k]$ described above, which has entries in $A[k]$.  In 
particular $p^k\Gamma[k]$ is an endomorphism of the $A[k]$-module $M[k]$ and induces an endomorphism of the $W(R[k])$-module 
$P$.  The induced $R[k]$-module map $Q/IP\map{}P/Q$ (which, speaking intuitively, measures the extent to which the quasi-
endomorphism $p^k\eta$ of $\mathfrak{G}[k]$ fails to be an endomorphism) is determined by
$$
f_1\mapsto \alpha_k e_1 \qquad f_2\mapsto \beta_k e_2.
$$

First suppose that $R[k] \map{}S$ is any quotient over which the quasi-endomorphism $p^k\eta$ of  $\mathfrak{G}[k]_{/S}$  is an 
endomorphism.  The base change of $(P,Q,F,V^{-1})$ has $$P_{S}=P\otimes_{W(R[k])} W(S)$$ as its underlying 
$W(S)$-module, with 
submodule $Q_{S}\subset P_{S}$ equal to the image of $Q$ under $P\map{}P_{S}$.  By hypothesis the endomorphism 
$p^k\Gamma[k]$ of $P_{S}$ preserves $Q_{S}$ and hence (letting $I_S$ denote the kernel of $W(S)\map{}S$) the induced map 
\begin{equation}\label{flag}
 (Q/IP)\otimes_{R[k]} S \iso  Q_S/I_S P_S\map{p^k\Gamma[k]}P_S/Q_S \iso (P/Q)\otimes_{R[k]} S
\end{equation}
 is trivial.  This implies that $\alpha_k$ and $\beta_k$ lie in the kernel of $R[k] \map{}S$.

Now set $S=R[k] /(\alpha_k,\beta_k)$ so that (\ref{flag}) is trivial.  It follows that $p^k\Gamma[k]$ preserves the submodule 
$Q_{S}\subset P_{S}$.  As we already know that $p^k\Gamma[k]$ is a quasi-endomorphism of $(P,Q,F,V^{-1})_{/S}$ some $\Z_p$-
multiple of $p^k\Gamma[k]$ commutes with both $F$ and $V^{-1}$.  But $W(S)$ is $\Z_p$-torsion-free and $P_S$ is free over 
$W(S)$;  thus $P_S$ has no $\Z_p$-torsion and we deduce that $p^k\Gamma[k]$ itself commutes with $F$ and $V^{-1}$.  In other 
words $p^k\Gamma[k]$ is an endomorphism of $(P,Q,F,V^{-1})_{/S}$.
\end{proof}

Explicitly computing $\alpha_k$ and $\beta_k$ by solving the recursion (\ref{recursion}) is prohibitively difficult; luckily everything we 
need to know about $\alpha_k$ and $\beta_k$ can be deduced from (\ref{structure I}), (\ref{structure II}), and (\ref{structure III}) without 
knowing the actual values of the $u_\ell$'s and $g_\ell$'s.  Let $\mathcal{A}$ denote the completed local ring of $\ZZ_p[[x_1,x_2]]$ at 
the prime ideal $(p,x_2)$, and note that the constants $u_\ell$ appearing in (\ref{structure I}) satisfy 
$$
u_\ell\in \mathcal{A}^\times
$$
as they have nonzero image in the residue field $\mathcal{A}/\mathfrak{m}_\mathcal{A}\iso \F((x_1))$.

\begin{Lem}\label{Lem:unr length}
Suppose that $p$ is odd and $k\ge 1$.  The $\mathcal{A}$-module  $Q_k=\mathcal{A}/(\alpha_k,\beta_k)$ is Artinian of length
$$
\length_{\mathcal{A}} (Q_k) = 2p^{k-1}+4p^{k-2}+6p^{k-3}+8p^{k-4}+\cdots+ (2k) p^0
$$
and is annihilated by  $ x_2^{  2+2p+2p^2+\cdots+2p^{k-1} }$.
\end{Lem}

\begin{proof}
    As we assume that $p>2$ we have the easy inequality
\begin{equation}\label{estimate}
\sum_{ 0\le i < \ell } 2p^i < p^\ell
\end{equation}
for all $\ell \ge 0$.  Hence  $x_2^{p^\ell}$ is a multiple of $x_2\cdot \epsilon_\pm(\ell)$ and we may define a nonunit  $C_\ell \in A$ by 
the relations
\begin{eqnarray*}
C_{\ell}\cdot \epsilon_-(\ell) =   g_{\ell+1} x_2^{p^\ell} & &  \mathrm{if\ }\ell \mathrm{\ even} \\
 C_{\ell}\cdot \epsilon_+(\ell) =  g_{\ell+1} x_2^{p^\ell} & &  \mathrm{if\ }\ell \mathrm{\ odd}.
\end{eqnarray*}
With this notation (\ref{structure I}), (\ref{structure II}), and (\ref{structure III}) can be rewritten as
\begin{eqnarray}\label{new structure}
\alpha_k &=&
  \sum_{  \substack{ 0 \le \ell \le k  \\ \ell \mathrm{\ even}  } }  u_\ell  p^{k-\ell}  \epsilon_-(\ell)   + 
  \sum_{  \substack{ 0 \le \ell < k  \\ \ell \mathrm{\ odd}  } }  C_\ell   p^{k-\ell} \epsilon_+(\ell)    \\
\beta_k &=&
  \sum_{  \substack{ 0 \le \ell \le k  \\ \ell \mathrm{\ odd}  } }  u_\ell  p^{k-\ell}  \epsilon_+(\ell)    + 
 \sum_{  \substack{ 0 \le \ell < k  \\ \ell \mathrm{\ even}  } }  C_\ell   p^{k-\ell} \epsilon_-(\ell)  .
\nonumber
\end{eqnarray}
For $0\le j\le k$ abbreviate 
$$
Q^{(j)}_k =  x_2^{2+2p+2p^2+\cdots+2p^{j-1}} \cdot Q_k.
$$  
We will construct  $\alpha_k^{(j)} ,\beta_k^{(j)}\in \mathcal{A}$ such that
$$
 \mathcal{A}/(\alpha_k^{(j)} , \beta_k^{(j)} ) \iso Q_k^{(j)}
$$
as $\mathcal{A}$-modules and such that 
\begin{eqnarray*}
\alpha_k^{(j)} &=&
  \sum_{  \substack{ j \le \ell \le k  \\ \ell \mathrm{\ even}  } }  u_\ell^{(j)}  p^{k-\ell} \frac{ \epsilon_-(\ell) }{\epsilon_-(j)}  + 
  \sum_{  \substack{ j \le \ell \le k  \\ \ell \mathrm{\ odd}  } }   C^{(j)}_\ell   p^{k-\ell}  \frac{ \epsilon_+(\ell) }{\epsilon_-(j)}   \\
\beta_k^{(j)} &=&
  \sum_{  \substack{ j \le \ell \le k  \\ \ell \mathrm{\ odd}  } }  u^{(j)}_\ell  p^{k-\ell}   \frac{ \epsilon_+(\ell)  }{ \epsilon_+(j) } + 
 \sum_{  \substack{ j \le \ell \le k  \\ \ell \mathrm{\ even}  } }   C^{(j)}_\ell   p^{k-\ell}  \frac{ \epsilon_-(\ell)  }{ \epsilon_+(j) } 
\end{eqnarray*}
for some units $u_\ell^{(j)}\in \mathcal{A}$  and some nonunits $C_\ell^{(j)}\in \mathcal{A}$. The construction is recursive, beginning  
with $\alpha_k^{(0)}=\alpha_k$ and $\beta_k^{(0)} = \beta_k$.  We now assume that $\alpha_k^{(j)}$ and $\beta_k^{(j)}$ have 
already been constructed and proceed to construct $\alpha_k^{(j+1)}$ and $\beta_k^{(j+1)}$.

Suppose first that $j$ is even.  The $\ell=j$ term in $\alpha_k^{(j)}$ is a unit multiple of $p^{k-j}$ while the $\ell=j$ term in 
$\beta_k^{(j)}$ is a multiple of $p^{k-j}$.  Therefore we may eliminate the $\ell=j$ term from $\beta_k^{(j)}$ by adding a suitable 
multiple of $\alpha_k^{(j)}$.   After collecting terms we find that there are units $u_\ell^{(j+1)}$ and  nonunits $C_\ell^{(j+1)}$  
(for $\ell$ odd and even, respectively) such that 
\begin{eqnarray*}\lefteqn{
\beta_k^{(j)} - \alpha_k^{(j)} \cdot \frac{C_j^{(j)}  \epsilon_-(j )  }{u_j^{(j)}  \epsilon_+(j) }   } \\
& & =   \sum_{  \substack{ j+1  \le \ell \le k  \\ \ell \mathrm{\ odd}  } } 
 u^{(j+1)}_\ell  p^{k-\ell}   \frac{ \epsilon_+(\ell)  }{ \epsilon_+(j) } 
 +  \sum_{  \substack{ j+1 \le \ell \le  k  \\ \ell \mathrm{\ even}  } }  C^{(j+1)}_\ell   p^{k-\ell}  \frac{ \epsilon_-(\ell)  }{ \epsilon_+(j) } .
\end{eqnarray*}
Each term on the right hand side is divisible by $x^{2p^j}$, and dividing out $x^{2p^j}$ leaves 
$$
\beta_k^{ (j+1) }  \define  \sum_{  \substack{ j+1  \le \ell \le k  \\ \ell \mathrm{\ odd}  } } 
  u^{(j+1)}_\ell  p^{k-\ell}   \frac{ \epsilon_+(\ell)  }{ \epsilon_+ ( j+1 ) } 
 +  \sum_{  \substack{ j+1 \le \ell \le  k  \\ \ell \mathrm{\ even}  } } C^{(j+1)}_\ell   p^{k-\ell}  \frac{ \epsilon_-(\ell)  }{ \epsilon_+(j+1) }.
 $$
Note that by construction of $\beta_k^{(j+1)}$ we have
\begin{equation}\label{filter step I}
\left(\alpha_k^{(j)} , \beta_k^{(j)} \right) = \left(\alpha_k^{(j)} ,  x_2^{2p^j} \beta_k^{(j+1)} \right).
\end{equation}
The $\ell=j+1$ term in $\beta_k^{ (j+1) }$ is a unit multiple of $p^{k-j-1}$, while the $\ell=j$ term in $\alpha_k^{ (j) }$ is a unit multiple 
of  $p^{k-j}$.  Therefore we may eliminate the $\ell=j$ term from $\alpha_k^{ (j) }$ by adding a suitable multiple of $\beta_k^{(j+1)}$.  Collecting common terms we find that 
\begin{eqnarray*}\lefteqn{
\alpha_k^{(j)} - \beta_k^{( j+1  )} \cdot \frac{u_j^{(j)} p  }{ u_{ j+1 }^{ (j+1) }   }} \\
& & =  
 \sum_{  \substack{ j+1  \le \ell \le k  \\ \ell \mathrm{\ even}  } } 
  u^{(j+1)}_\ell  p^{k-\ell}   \frac{ \epsilon_-(\ell)  }{ \epsilon_-( j ) } 
        +  \sum_{  \substack{ j+1 \le \ell \le  k  \\ \ell \mathrm{\ odd}  } }    C^{(j+1)}_\ell   p^{k-\ell}  \frac{ \epsilon_+( \ell )  }{ \epsilon_-( j ) } 
\end{eqnarray*}
for some units $u_\ell^{(j+1)}$ and nonunits $C_\ell^{(j+1)}$ (for $\ell$ even and odd, respectively).   As we are assuming that $j$ is 
even $\epsilon_-(j)=\epsilon_-(j+1)$, and the above expression is equal to 
$$
\alpha_k^{ (j+1) }  \define  \sum_{  \substack{ j+1  \le \ell \le k  \\ \ell \mathrm{\ even}  } } 
  u^{(j+1)}_\ell  p^{k-\ell}   \frac{ \epsilon_-(\ell)  }{ \epsilon_- ( j+1 ) } 
 +  \sum_{  \substack{ j+1 \le \ell \le  k  \\ \ell \mathrm{\ odd}  } } C^{(j+1)}_\ell   p^{k-\ell}  \frac{ \epsilon_+(\ell)  }{ \epsilon_-(j+1) }.
$$
By construction of $\alpha_k^{(j+1)}$
\begin{equation}\label{filter step II}
\left( \alpha_k^{( j+1 )} , \beta_k^{( j+1 )} \right) = \left(\alpha_k^{( j )} ,  \beta_k^{(j+1)} \right).
\end{equation}
As we are assuming  $Q_k^{(j)}\iso \mathcal{A}/(\alpha_k^{(j)} , \beta_k^{(j)} )$ the exact sequence
$$
{0} \map{}  { Q_k^{(j+1)} } \map{}  {  Q_k^{(j)} } \map{} { Q_k^{(j)}/Q_k^{(j+1)} } \map{} { 0}   
$$
can be identified with 
$$
{0} \map{} { (x_2^{2p^j})/(x_2^{2p^j})\cap (\alpha_k^{(j)},\beta_k^{(j)}) } \map{}  {  \mathcal{A}/(\alpha_k^{(j)} , \beta_k^{(j)}  ) }
\map{}  {  \mathcal{A}/(x_2^{2p^j} , \alpha_k^{(j)} , \beta_k^{(j)}  )      } \map{}  { 0}.
$$
Using (\ref{filter step I}) and (\ref{filter step II}) and the fact that $\alpha_k^j$ is not divisible by $x_2$ we obtain isomorphisms
$$
\mathcal{A} /( \alpha_k^{(j+1)} , \beta_k^{(j+1)} )   \iso  \mathcal{A} /( \alpha_k^{(j)} , \beta_k^{(j+1)} )  
 \map{x_2^{2p^j} } (x_2^{2p^j})/(x_2^{2p^j})\cap (\alpha_k^{(j)},\beta_k^{(j)}) \iso Q_k^{(j+1)}
$$
as desired.

Now suppose that $j$ is odd.  The method is the same as in the even case: we first eliminate the $\ell=j$ term of $\alpha_k^{(j)}$ by 
adding a multiple of $\beta_k^{(j)}$ to $\alpha_k^{(j)}$, resulting in
\begin{eqnarray*}\lefteqn{
\alpha_k^{(j)} - \beta_k^{(j)}  \frac{ C_j^{(j)} \epsilon_+(j) }{ u_j^{(j)} \epsilon_-(j) }   }  \\
&=& 
  \sum_{  \substack{ j+1 \le \ell \le k  \\ \ell \mathrm{\ even}  } }  u_\ell^{(j+1)}  p^{k-\ell} \frac{ \epsilon_-(\ell) }{\epsilon_-(j)}  + 
  \sum_{  \substack{ j +1 \le \ell \le k  \\ \ell \mathrm{\ odd}  } }   C^{(j+1)}_\ell   p^{k-\ell}  \frac{ \epsilon_+(\ell) }{\epsilon_-(j)}  
\end{eqnarray*}
for some units $u_\ell^{(j+1)}$ and nonunits $C_\ell^{(j+1)}$ (with $\ell$ even and odd, respectively).  Each term on the right is 
divisible by $x_2^{2p^j}$, and dividing gives
$$
\alpha_k^{(j+1)} \define 
  \sum_{  \substack{ j+1 \le \ell \le k  \\ \ell \mathrm{\ even}  } }  u_\ell^{(j+1)}  p^{k-\ell} \frac{ \epsilon_-(\ell) }{\epsilon_-(j+1)}  + 
  \sum_{  \substack{ j +1 \le \ell \le k  \\ \ell \mathrm{\ odd}  } }   C^{(j+1)}_\ell   p^{k-\ell}  \frac{ \epsilon_+(\ell) }{\epsilon_-(j+1)}  .
$$
By construction
$$
( \alpha_k^{(j)} , \beta_k^{(j)} ) = ( x_2^{2p^j} \alpha_k^{(j+1)} , \beta_k^{(j)} ).
$$
One now checks that
\begin{eqnarray*}\lefteqn{
\beta_k^{(j)} - \alpha_k^{(j+1)}  \frac{ u_j^{(j)} p }{ u_{j+1}^{( j +1 )} }   }  \\
&=& 
  \sum_{  \substack{ j+1 \le \ell \le k  \\ \ell \mathrm{\ odd}  } }  u_\ell^{(j+1)}  p^{k-\ell} \frac{ \epsilon_+(\ell) }{\epsilon_+(j)}  + 
  \sum_{  \substack{ j +1 \le \ell \le k  \\ \ell \mathrm{\ even}  } }   C^{(j+1)}_\ell   p^{k-\ell}  \frac{ \epsilon_-(\ell) }{\epsilon_+(j)}  
\end{eqnarray*}
for some units $u_\ell^{(j+1)}$ and nonunits $C_\ell^{(j+1)}$ (with $\ell$ odd and even, respectively).   As $j$ is odd we have 
$\epsilon_+(j)=\epsilon_+(j+1)$, and so the above quantity is equal to
$$
\beta_k^{(j+1)} \define   \sum_{  \substack{ j+1 \le \ell \le k  \\ \ell \mathrm{\ odd}  } }  
u_\ell^{(j+1)}  p^{k-\ell} \frac{ \epsilon_+(\ell) }{\epsilon_+(j+1)}  + 
  \sum_{  \substack{ j +1 \le \ell \le k  \\ \ell \mathrm{\ even}  } }   C^{(j+1)}_\ell   p^{k-\ell}  \frac{ \epsilon_-(\ell) }{\epsilon_+(j+1)}  .
$$
By construction 
$$
( \alpha_k^{(j+1)} , \beta_k^{(j+1)} ) = (  \alpha_k^{(j+1)} , \beta_k^{(j)} )
$$
and exactly in the even case we find isomorphisms
$$
\mathcal{A} /( \alpha_k^{(j+1)} , \beta_k^{(j+1)} )   \iso  \mathcal{A} /( \alpha_k^{(j+1)} , \beta_k^{(j)} )   \map{x_2^{2p^j} } 
(x_2^{2p^j})/(x_2^{2p^j})\cap (\alpha_k^{(j)},\beta_k^{(j)}) \iso Q_k^{(j+1)}.
$$

Both parts of the lemma now follow easily.  As one of $\alpha_k^{(k)}$ and $\beta_k^{(k)}$ is a unit we have
$$
x_2^{2+2p+2p^2+\cdots+2p^{k-1}} \cdot Q_k = Q_k^{(k)} \iso \mathcal{A}/(\alpha_k^{(k)} , \beta_k^{(k)}) =0.
$$
The length of $Q_k$ may be computed as 
\begin{eqnarray*}
\length_{\mathcal{A}}(Q_k) &=& \sum_{ 0\le j < k } \length_{\mathcal{A}}( Q_k^{(j)}/Q_k^{(j+1)}  ) \\
& = &  \sum_{ 0\le j < k } \length_{\mathcal{A}}(   \mathcal{A}/( x_2^{2p^j} , \alpha_k^{(j)} , \beta_k^{(j)} )   ) .
\end{eqnarray*}
If $j$ is even then by (\ref{filter step I})
$$
( x_2^{2p^j} , \alpha_k^{(j)} , \beta_k^{(j)} )= ( x_2^{2p^j} , \alpha_k^{(j)} , x_2^{2p^j} \beta_k^{( j+1 )} ) = ( x_2^{2p^j} , \alpha_k^{(j)} ) 
= (x_2^{2p^j} , p^{k-j} ).
$$
For the final equality we have used the fact  that  for $\ell > j$  
$$
x_2^{2p^j} \mathrm{\ divides}  \begin{cases}
\epsilon_-(\ell)/\epsilon_-(j)  &  \mathrm{if\ } \ell\mathrm{\ is\ even} \\
\epsilon_+(\ell)/\epsilon_-(j)  &  \mathrm{if\ } \ell\mathrm{\ is\ odd}.
\end{cases}
$$
 Similarly if $j$ is odd then
$$
( x_2^{2p^j} , \alpha_k^{(j)} , \beta_k^{(j)} )= ( x_2^{2p^j} , x_2^{2p^j} \alpha_k^{(j+1)} , \beta_k^{( j )} ) 
= ( x_2^{2p^j} , \beta_k^{(j)} ) = (x_2^{2p^j} , p^{k-j} ).
$$
In either case
$$
 \length_{\mathcal{A}}(   \mathcal{A}/( x_2^{2p^j} , \alpha_k^{(j)} , \beta_k^{(j)} )   ) = 2p^j(k-j)
$$
completing the proof of
$$
\length_{\mathcal{A}}(Q_k)  =   \sum_{ 0\le j < k }  2p^j(k-j).
$$
\end{proof}

We continue to let $\mathcal{A}$ denote the completed local ring of $\ZZ_p[[x_1,x_2]]$ at $(p,x_2)$, with maximal ideal 
$\mathfrak{m}_\mathcal{A}=(p,x_2)$.

\begin{Lem}\label{Lem:is thick}
If $p$ is odd and $k\ge 1$ then there is an inclusion of ideals in $\mathcal{A}$
$$
(p^{2k+1} , x_2^{p^k}) \subset \mathfrak{m}_\mathcal{A}\cdot (\alpha_k,\beta_k).
$$
\end{Lem}

\begin{proof}
Let $Q_k=\mathcal{A}/(\alpha_k,\beta_k)$ and for $0\le j \le k$ set $$Q_k^{(j)} = p^{2k-2j} Q_k.$$ We will recursively define 
$\alpha_k^{(j)}$ and $\beta_k^{(j)}$ in $\mathcal{A}$ with the property 
$$
\mathcal{A}/(\alpha_k^{(j)},\beta_k^{(j)}) \iso Q_k^{(j)}
$$
as $\mathcal{A}$-modules and such that
\begin{eqnarray*}
\alpha_k^{(j)} &=&
  \sum_{  \substack{ 0 \le \ell \le j  \\ \ell \mathrm{\ even}  } }  u_\ell^{(j)}  p^{j-\ell} \epsilon_-(\ell)  + 
  \sum_{  \substack{ 0 \le \ell \le j  \\ \ell \mathrm{\ odd}  } }   C^{(j)}_\ell   p^{j-\ell} \epsilon_+(\ell)   \\
\beta_k^{(j)} &=&
  \sum_{  \substack{ 0 \le \ell \le j  \\ \ell \mathrm{\ odd}  } }  u^{(j)}_\ell  p^{j-\ell}  { \epsilon_+(\ell) } + 
 \sum_{  \substack{ 0 \le \ell \le j  \\ \ell \mathrm{\ even}  } }   C^{(j)}_\ell   p^{j-\ell}{ \epsilon_-(\ell) } 
\end{eqnarray*}
for some units $u_\ell^{(j)}\in\mathcal{A}$  and some nonunits $C_\ell^{(j)}\in \mathcal{A}$.   After (\ref{new structure}) we may 
begin by defining 
$$\alpha_k^{(k)}=\alpha_k\qquad\beta_k^{(k)} = \beta_k.$$  Assuming that $\alpha_k^{(j)}$ and $\beta_k^{(j)}$ have 
been constructed we now construct $\alpha_k^{(j-1)}$ and $\beta_k^{(j-1)}$.

Assume first that $j$ is even.  The $\ell=j$ term of $\alpha_k^{(j)}$ is a unit multiple of $\epsilon_-(j)$ while the $\ell=j$ term of 
$\beta_k^{(j)}$ is a multiple of $\epsilon_-(j)$.  Subtracting a suitable multiple of $\alpha_k^{(j)}$ from $\beta_k^{(j)}$ removes the 
$\ell=j$ term from $\beta_k^{(j)}$ resulting in
\begin{eqnarray*}\lefteqn{
\beta_k^{(j)} - \alpha_k^{(j)} \frac{C_j^{(j)}}{u_j^{(j)}}
} \\
&=& 
  \sum_{  \substack{ 0 \le \ell \le j -1 \\ \ell \mathrm{\ odd}  } }  u^{(j-1)}_\ell  p^{j-\ell}  { \epsilon_+(\ell) } + 
 \sum_{  \substack{ 0 \le \ell \le j -1 \\ \ell \mathrm{\ even}  } }   C^{(j-1)}_\ell   p^{j-\ell}{ \epsilon_-(\ell) } 
\end{eqnarray*}
for some units $u^{(j-1)}_\ell $ and nonunits $ C^{(j-1)}_\ell$ (with $\ell$ odd and even, respectively).  Dividing this quantity by $p$ 
results in
$$
\beta_k^{(j-1)} \define \sum_{  \substack{ 0 \le \ell \le j -1 \\ \ell \mathrm{\ odd}  } }  u^{(j-1)}_\ell  p^{j-1-\ell}  { \epsilon_+(\ell) } + 
 \sum_{  \substack{ 0 \le \ell \le j -1 \\ \ell \mathrm{\ even}  } }   C^{(j-1)}_\ell   p^{j-1-\ell}{ \epsilon_-(\ell) } ,
$$
and by construction
\begin{equation}\label{p filter I}
(\alpha_k^{(j)},\beta_k^{(j)}) = (\alpha_k^{(j)} , p \beta_k^{(j-1)} ).
\end{equation}
Subtracting a suitable multiple of $\beta_k^{(j-1)}$  removes the $\ell=j$ term from $\alpha_k^{(j)}$, resulting in
\begin{eqnarray*}\lefteqn{
\alpha_k^{(j)} - \beta_k^{(j-1)} \frac{ u_{ j}^{(j )}\epsilon_-(j)}{u_{j-1}^{(j-1)} \epsilon_+(j-1) }  } \\
&=&   \sum_{  \substack{ 0 \le \ell \le j-1  \\ \ell \mathrm{\ even}  } }  u_\ell^{(j-1)}  p^{j-\ell} \epsilon_-(\ell)  + 
  \sum_{  \substack{ 0 \le \ell \le j-1  \\ \ell \mathrm{\ odd}  } }   C^{(j-1)}_\ell   p^{j-\ell} \epsilon_+(\ell)
\end{eqnarray*}
for some units $u_\ell^{(j-1)}$ and nonunits $C_\ell^{(j-1)}$ (with $\ell$ even and odd, respectively).
Dividing this last expression by $p$ we obtain
$$
\alpha_k^{(j-1)} \define  \sum_{  \substack{ 0 \le \ell \le j-1  \\ \ell \mathrm{\ even}  } }  u_\ell^{(j-1)}  p^{j-1-\ell} \epsilon_-(\ell)  + 
  \sum_{  \substack{ 0 \le \ell \le j-1  \\ \ell \mathrm{\ odd}  } }   C^{(j-1)}_\ell   p^{j-1-\ell} \epsilon_+(\ell).
$$
By construction
\begin{equation}\label{p filter II}
(\alpha_k^{(j)} , \beta_k^{(j-1)} ) = (p\alpha_k^{(j-1)} , \beta_k^{(j-1)} ).
\end{equation}
The exact sequence
$$
0\map{}pQ_k^{(j)} \map{}Q_k^{(j)} \map{} Q_k^{(j)}/pQ_k^{(j)} \map{} 0
$$
can be identified with
$$
0\map{} ( p )/(p  ) \cap ( \alpha_k^{(j)} , \beta_k^{(j)})  \map{}  \mathcal{A} /( \alpha_k^{(j)} , \beta_k^{(j)})
 \map{} \mathcal{A} / ( p , \alpha_k^{(j)} , \beta_k^{(j)})\map{} 0.
$$
As $\alpha_k^{(j)}$ is not divisible by $p$ we find, using (\ref{p filter I}), isomorphisms
$$
 \mathcal{A} /( \alpha_k^{(j)} , \beta_k^{(j-1)} ) \map{p} ( p )/(p  ) \cap ( \alpha_k^{(j)} , p\beta_k^{(j-1)})  \iso pQ_k^{(j)}.
$$
The exact sequence
$$
0\map{} Q_k^{(j-1)} \map{} pQ_k^{(j)} \map{} pQ_k^{(j)}/p^2Q_k^{(j)} \map{}0
$$
is then identified with 
$$
0\map{} ( p )/(p  ) \cap ( \alpha_k^{(j)} , \beta_k^{(j-1)})  \map{}  \mathcal{A} /( \alpha_k^{(j)} , \beta_k^{(j-1)}) 
\map{} \mathcal{A} / ( p , \alpha_k^{(j)} , \beta_k^{(j-1)})\map{} 0
$$
and we find, using (\ref{p filter II}) and the observation that $\beta_k^{(j-1)}$ is not divisible by $p$, isomorphisms
$$
 \mathcal{A} /( \alpha_k^{(j-1)} , \beta_k^{(j-1)} ) \map{p} ( p )/(p  ) \cap ( p\alpha_k^{(j-1)} , \beta_k^{(j-1)})  \iso Q_k^{(j-1)}.
$$

The construction of $\alpha_k^{(j-1)}$ and $\beta_k^{(j-1)}$ for $j$ odd is similar.  We first write
\begin{eqnarray*}\lefteqn{
\alpha_k^{(j)} - \beta_k^{(j)} \frac{C_j^{(j)}}{u_j^{(j)}}
} \\
&=& 
  \sum_{  \substack{ 0 \le \ell \le j -1 \\ \ell \mathrm{\ even}  } }  u^{(j-1)}_\ell  p^{j-\ell}  { \epsilon_-(\ell) } + 
 \sum_{  \substack{ 0 \le \ell \le j -1 \\ \ell \mathrm{\ odd}  } }   C^{(j-1)}_\ell   p^{j-\ell}{ \epsilon_+(\ell) } 
\end{eqnarray*}
for some units $u^{(j-1)}_\ell $ and nonunits $ C^{(j-1)}_\ell$ (with $\ell$ even and odd, respectively).  Dividing 
this quantity by $p$ results in
$$
\alpha_k^{(j-1)} \define \sum_{  \substack{ 0 \le \ell \le j -1 \\ \ell \mathrm{\ even}  } }  u^{(j-1)}_\ell  p^{j-1-\ell}  { \epsilon_-(\ell) } + 
 \sum_{  \substack{ 0 \le \ell \le j -1 \\ \ell \mathrm{\ odd}  } }   C^{(j-1)}_\ell   p^{j-1-\ell}{ \epsilon_+(\ell) }.
$$
We next write
\begin{eqnarray*}\lefteqn{
\beta_k^{(j)} - \alpha_k^{(j-1)} \frac{ u_{ j}^{(j )}\epsilon_+(j)}{u_{j-1}^{(j-1)} \epsilon_-(j-1) }  } \\
&=&   \sum_{  \substack{ 0 \le \ell \le j-1  \\ \ell \mathrm{\ odd}  } }  u_\ell^{(j-1)}  p^{j-\ell} \epsilon_+(\ell)  + 
  \sum_{  \substack{ 0 \le \ell \le j-1  \\ \ell \mathrm{\ even}  } }   C^{(j-1)}_\ell   p^{j-\ell} \epsilon_-(\ell)
\end{eqnarray*}
for some units $u_\ell^{(j-1)}$ and nonunits $C_\ell^{(j-1)}$ (with $\ell$ even and odd, respectively).
Dividing this last expression by $p$ we obtain
$$
\beta_k^{(j-1)} \define  \sum_{  \substack{ 0 \le \ell \le j-1  \\ \ell \mathrm{\ odd}  } }  u_\ell^{(j-1)}  p^{j-1-\ell} \epsilon_+(\ell) 
 +    \sum_{  \substack{ 0 \le \ell \le j-1  \\ \ell \mathrm{\ even}  } }   C^{(j-1)}_\ell   p^{j-1-\ell} \epsilon_-(\ell).
$$
Using
$$
(\alpha_k^{(j)} , \beta_k^{(j)}) = ( p\alpha_k^{( j-1 )}, \beta_k^{( j)} ) 
\qquad
(\alpha_k^{(j-1)} , \beta_k^{(j)}) = ( \alpha_k^{( j-1 )},p \beta_k^{( j-1)} ) 
$$
one proves that
$$
\mathcal{A} /(\alpha_k^{(j-1)}, \beta_k^{(j-1)} ) \iso p^2 Q_k^{(j)} = Q_k^{(j-1)}.
$$

As $\alpha_k^{(0)}$ is a unit we find that
$$
p^{2k}Q_k\iso \mathcal{A}/( \alpha_k^{(0)} , \beta_k^{(0)}) =0
$$
and so $p^{2k}\in (\alpha_k ,\beta_k)$.  This proves that $p^{2k+1}\in \mathfrak{m}_\mathcal{A} \cdot (\alpha_k ,\beta_k)$.  
On the other hand Lemma \ref{Lem:unr length} shows that
$$
x_2^{2+2p+2p^2+\cdots+2p^{k-1}} \in (\alpha_k ,\beta_k),
$$
and so the inequality (\ref{estimate}) shows that
$$
x_2^{p^k} \in \mathfrak{m}_\mathcal{A} \cdot (\alpha_k ,\beta_k).
$$
\end{proof}

\begin{Prop}\label{Prop:unr components}
Assume that $c_0>0$ and that $p>2$.  If $\mathfrak{C}_i^\vertical$ is either of the vertical components of $\mathfrak{Y}$ found in 
Proposition \ref{Prop:two components} then (in the notation of Definition \ref{Def:components})
\begin{eqnarray*}
\mathrm{mult}_\mathfrak{Y}(\mathfrak{C}_i^\vertical) &=& 2p^{c_0-1}+4p^{c_0-2}+6p^{c_0-3}+8p^{c_0-4}+\cdots+ 2c_0p^0 \\
&=& \frac{2p(p^{c_0} -1) -2c_0(p-1)}{(p-1)^2}.
\end{eqnarray*}
\end{Prop}

\begin{proof}
We give the proof for $i=1$, the proof for $i=2$ being entirely similar.   Set $k=c_0$ and let $J\subset R_\mathfrak{M}$ be the ideal of 
definition of the closed formal subscheme $\mathfrak{Y}\map{}\mathfrak{M}$.  As 
$$
\Z_{p^2} [p^k\eta] = \Z_{p^2}+ p^{c_0}\co_E = \Z_{p^2}[\tau]
$$
the quotient $R_\mathfrak{M}/J\iso R_\mathfrak{Y}$ is the maximal quotient over which the endomorphism $j(p^k\eta)$ of 
$\mathfrak{g}$ lifts to an endomorphism of the universal deformation of $\mathfrak{g}$.  Let $\mathfrak{q}\subset R_\mathfrak{Y}$ be 
the ideal of definition of the closed formal subscheme $\mathfrak{C}_1^\vertical\map{}\mathfrak{Y}$, and denote again by 
$\mathfrak{q}$ the kernel of 
$$
R_\mathfrak{M}\map{}R_\mathfrak{Y}\map{}R_\mathfrak{Y}/\mathfrak{q}\iso \F[[x_1]]
$$
so that $\mathfrak{q}=(p,x_2)$.   Proposition \ref{Prop:small locus} tells us that
$$
J+(p^{2k+1} , x_2^{p^k}) = (\alpha_k,\beta_k)+ (p^{2k+1} , x_2^{p^k})
$$
as ideals of $R_\mathfrak{M}\iso \ZZ_p[[x_1,x_2]]$.  If we define $J_\mathcal{A}= J\cdot\mathcal{A}$, where $\mathcal{A}$ is the 
completed local ring of $R_\mathfrak{M}$ at $\mathfrak{q}$, then we obtain the equality of ideals of $\mathcal{A}$
$$
J_\mathcal{A}+(p^{2k+1} , x_2^{p^k}) = (\alpha_k,\beta_k)+ (p^{2k+1} , x_2^{p^k}).
$$
Lemma \ref{Lem:is thick} gives the inclusion of ideals $(p^{2k+1},x_2^{p^k}) \subset (\alpha_k,\beta_k)$ in $\mathcal{A}$ and so 
$$
J_\mathcal{A}+(p^{2k+1} , x_2^{p^k}) = (\alpha_k,\beta_k).
$$
From this we obviously have $J_\mathcal{A}\subset(\alpha_k,\beta_k)$.  To prove the reverse inclusion we apply the inclusion of 
Lemma \ref{Lem:is thick} to obtain
$$
(\alpha_k,\beta_k) \subset  J_\mathcal{A}+(p^{2k+1} , x_2^{p^k})  
\subset  J_\mathcal{A} + \mathfrak{m}_\mathcal{A}(\alpha_k,\beta_k) ,
$$
and an induction argument then proves that 
$$
(\alpha_k,\beta_k)\subset J_\mathcal{A} + \mathfrak{m}_\mathcal{A}^\ell(\alpha_k,\beta_k)
$$
 for every $\ell>0$.  In particular $(\alpha_k,\beta_k)\subset J_\mathcal{A} + \mathfrak{m}_\mathcal{A}^{\ell+1}$ for every $\ell$, and 
 topological considerations prove that $(\alpha_k,\beta_k)\subset J_\mathcal{A}$.   Having proved 
 $J_\mathcal{A}=(\alpha_k,\beta_k)$, Lemma \ref{Lem:unr length} now shows that $\mathcal{A}/J_\mathcal{A}$ is Artinian of length
 $$
 \length_{\mathcal{A}}(\mathcal{A}/J_\mathcal{A}) = 2p^{k-1}+4p^{k-2}+6p^{k-3}+\cdots+(2k)p^0.
 $$
 Using the isomorphisms
 $$
 \mathcal{A}/J_\mathcal{A} \iso R_{\mathfrak{M},\mathfrak{q}}/JR_{\mathfrak{M},\mathfrak{q}} \iso R_{\mathfrak{Y},\mathfrak{q}}
 $$
 we find that $R_{\mathfrak{Y},\mathfrak{q}}$ has length $2p^{k-1}+4p^{k-2}+6p^{k-3}+\cdots+(2k)p^0$, completing the proof of the 
 first equality in the statement of the proposition.  The proof of the second is then an elementary exercise.
\end{proof}


\subsection{Vertical multiplicities: $E_0$ ramified}
\label{ss:ramified multiplicity}


Assume that $E_0/\Q_p$ is ramified.  We assume throughout all of \S \ref{ss:ramified multiplicity} that $p>2$.  This allows us to 
choose a trace-free uniformizing parameter $\varpi_{E_0}\in\co_{E_0}$.  In the notation of \S \ref{ss:dieudonne} $\varpi_{E_0}$ is a 
root of $x^2-bb^\sigma p$ for some $b\in\Z_{p^2}^\times$, and the normalized action of $\varpi_{E_0}$ on the window $(M,N,\Phi)$ of 
the $p$-Barsotti-Tate group $\mathfrak{g}$ (with respect to the frame $\ZZ_p\map{}\F$, as in \S \ref{ss:hasse-witt}) is through the 
matrix
$$
\Gamma = \left(  \begin{matrix}   & pb \\ b^\sigma  \\ & &  & pb \\ & & b^\sigma &  \end{matrix}  \right).
$$
Let $(\mathfrak{G}_1^\vertical,\rho_1^\vertical)$ be the restriction of the universal deformation of $\mathfrak{g}$ over 
$\mathfrak{M}\iso \Spf(\ZZ_p[[x_1,x_2]])$ to the closed formal subscheme 
$$
\mathfrak{C}_1^\vertical = \Spf(\F[[x_1]])
$$
and let $(M_1^\vertical,N_1^\vertical,\Phi)$ be the window of $\mathfrak{G}_1^\vertical$ with respect to the frame 
$\ZZ_p[[x_1]]\map{}\F[[x_1]]$ so that, by the calculations of \S \ref{ss:hasse-witt}, the endomorphism $\Gamma$ of $(M,N,\Phi)$ lifts to 
the quasi-endomorphism of $(M_1^\vertical, N_1^\vertical,\Phi)$ given by
$$
\Gamma_1^\vertical = \left(  \begin{matrix}  Y_1^\vertical \\ &  Z_1^\vertical  \end{matrix}  \right)
$$
with
\begin{eqnarray*}
Y_1^\vertical &=& 
\left(  \begin{matrix}     &  pb  \\   b^\sigma &  \end{matrix} \right)  +  b^\sigma
 \left(  \begin{matrix}   f(x_1)    &   -  g(x_1)  \\    & - f(x_1) \end{matrix} \right)   \\
Z_1^\vertical &=&  
\left(  \begin{matrix}    &  pb   \\ b^\sigma  \end{matrix} \right)  +  b \left(  \begin{matrix}    - f(x_1^p) &     \\ -\frac{g(x_1^p)}{p}   &  f(x_1^p)
 \end{matrix} \right).
\end{eqnarray*}
We now imitate the method of \S \ref{ss:unramified multiplicity}.  Define $R[k]$ and $A[k]$ by (\ref{thick-frame}), let $\mathfrak{G}[k]$ 
be the base change of the universal $p$-Barsotti-Tate group through  $\ZZ_p[[x_1,x_2]]\map{}R[k]$, and let $(M[k], N[k],\Phi)$ be the 
window of $\mathfrak{G}[k]$ with respect to the frame $A[k]\map{}R[k]$.  Viewing each $\mathfrak{G}[k]$ as a deformation of 
$\mathfrak{G}[0]=\mathfrak{G}_1^\vertical$ the quasi-endormorphism $\Gamma[0]=\Gamma$ of $$(M[0],N[0],\Phi)=(M,N,\Phi)$$ 
lifts to  the quasi-endomorphism of $(M[k],N[k],\Phi)$ given by the matrix
$$
\Gamma[k] = \left(\begin{matrix} Y[k] & \\ & Z[k] \end{matrix}\right)
$$
in which $Y[k],Z[k]\in M_2(A[k])\otimes_{\ZZ_p}\QQ_p$ satisfy the recursion (\ref{recursion II}) and the initial conditions
$$
Y[0]=Y_1^\vertical\qquad Z[0]=Z_1^\vertical.
$$
As in \S \ref{ss:unramified multiplicity} we lift each $Y[k]$ and $Z[k]$ to a matrix in $M_2(\ZZ_p[[x_1,x_2]])\otimes_{\ZZ_p}\QQ_p$ in 
such a way that (\ref{recursion II}) continues to hold.  For $k\ge 0$ define 
$$
Y_k,Z_k\in M_2(\ZZ_p[[x_1,x_2]])\otimes_{\ZZ_p}\QQ_p
$$ 
by (\ref{init}) and set $Y_0=Y[0]$ and $Z_0=Z[0]$.  Explicitly computing $Y[1]$ and $Z[1]$ we find that 
\begin{eqnarray*}
Y_1  &=&  0  \\
Z_1 &=&      x_2 \left(\begin{matrix}  -b^\sigma g(x_1^p) &  b^\sigma  x_2 g(x_1^p) \\ & b^\sigma  g(x_1^p)\end{matrix}\right)  
+   p x_2  \left(\begin{matrix} b^\sigma &  2 b  f(x_1^p) -b^\sigma x_2 \\ & -b^\sigma \end{matrix}\right).
\end{eqnarray*}
For $k\ge 1$ we again let $y_k, z_k \in M_2(\ZZ_p[[x_1,x_2]])$ denote the upper right entries of $Y_k$ and $Z_k$, and let 
$\alpha_k, \beta_k\in M_2(\ZZ_p[[x_1,x_2]])$ denote the upper right entries of $p^kY[k]$ and $p^kZ[k]$.   By the same reasoning as in 
\S \ref{ss:unramified multiplicity} (that is, by  computing the images of $Y_k$ and $Z_k$ in $\F[[x_1,x_2]]$ and using the fact that 
$Z_1$ is divisible by $x_2$) we find that $y_k$ and $z_k$  satisfy (\ref{structure I}) for some $g_k\in \ZZ_p[[x_1,x_2]]$ and some 
$u_k \in \ZZ_p[[x_1]]$ with nonzero image in $\F[[x_1]]$, and that (\ref{structure II}) and (\ref{structure III}) hold.

\begin{Prop}\label{Prop:ram components}
Assume that $c_0>0$ and that $p>2$.  If $\mathfrak{C}_i^\vertical$ is either of the vertical components of $\mathfrak{Y}$ found in 
Proposition \ref{Prop:two components} then (in the notation of Definition \ref{Def:components})
\begin{eqnarray*}
\mathrm{mult}_\mathfrak{Y}(\mathfrak{C}_i^\vertical) &=& 2p^{c_0-1}+4p^{c_0-2}+6p^{c_0-3}+8p^{c_0-4}+\cdots+ 2c_0p^0 \\
&=& \frac{2p(p^{c_0} -1) -2c_0(p-1)}{(p-1)^2}.
\end{eqnarray*}
\end{Prop}

\begin{proof}
The statement is exactly that of Proposition  \ref{Prop:unr components}, which dealt with the case of $E_0/\Q_p$ unramified.  In the 
ramified situation the statements of Lemmas \ref{Lem:unr length} and \ref{Lem:is thick} continue to hold, as the proofs only require 
that $\alpha_k$ and $\beta_k$ satisfy (\ref{structure I}), (\ref{structure II}), and (\ref{structure III}).  Similarly the  statement and proof of 
Proposition \ref{Prop:small locus} hold verbatim in the ramified case (taking $\eta=\varpi_{E_0}$).  The proof of the proposition is now 
the same, word-for-word, as that  of  Proposition \ref{Prop:unr components}  (again taking $\eta=\varpi_{E_0}$).
\end{proof}


\section{Horizontal components}
\label{s:horizontal}


In this subsection we will compute all horizontal components of $\mathfrak{Y}$ in the sense of Definition \ref{Def:components}.    The 
strategy is to start from the components of $\mathfrak{Y}_0$, all of which are known to us by the Gross-Keating theory of  quasi-
canonical lifts described in \S \ref{ss:quasi-canonical}, and apply the action of $\co_E^\times$ on $\mathfrak{Y}$ as described in \S 
\ref{ss:functors} to produce more components.  In the case in which $E_0/\Q_p$ is unramified we will show that this construction 
accounts for all horizontal components of $\mathfrak{Y}$.  More precisely, Proposition \ref{Prop:unr orbits} shows that every 
$\co_E^\times$-orbit of horizontal components of $\mathfrak{Y}$ contains a unique horizontal component of $\mathfrak{Y}_0$.  When 
$E_0/\Q_p$ is ramified the situation is slightly more complicated.  Taking $\co_E^\times$-orbits of components of $\mathfrak{Y}_0$ 
yields exactly half of the horizontal components of $\mathfrak{Y}$.  The components constructed in this way we call \emph{standard} 
components.  To construct the remaining \emph{nonstandard} components we first construct a closed formal subscheme 
$\mathfrak{Y}_0'\map{}\mathfrak{M}_0$ different from $\mathfrak{Y}_0$, but whose image in $\mathfrak{M}$ is still contained in 
$\mathfrak{Y}$.  Taking $\co_E^\times$-orbits of components of $\mathfrak{Y}_0'$ yields the nonstandard components of 
$\mathfrak{Y}$ (Proposition \ref{Prop:ram orbits}).  Having thus constructed all horizontal components $\mathfrak{C}$ of 
$\mathfrak{Y}$, we use results of Keating on the endomorphism rings of reductions of quasi-canonical lifts to compute the intersection 
number  $I_\mathfrak{M}(\mathfrak{C},\mathfrak{M}_0)$ for those components which meet $\mathfrak{M}_0$ properly 
(the \emph{proper} components of Definition \ref{Def:components}).
 We know from Corollary \ref{Cor:horizontal reduced} that every horizontal component $\mathfrak{C}$ satisfies 
 $\mathrm{mult}_\mathfrak{Y}(\mathfrak{C})=1$, and so our results give a complete picture of the horizontal part of $\mathfrak{Y}$ (at 
 least for $p>2$, a hypothesis which will first appear in Proposition \ref{Prop:ram intersection II}).


\subsection{Quasi-canonical lifts}
\label{ss:quasi-canonical}


Recall from \S \ref{ss:dieudonne} that $M_0$ is the fraction field of $W_0$.  Viewing $E_0$ as a subfield of $M_0$ via $\Psi$, the 
reciprocity map of class field theory provides an isomorphism
$$
\mathrm{rec}:\co_{E_0}^\times\map{} \Gal(\mathcal{E}_0^\mathrm{ab}/M_0)
$$
where $\mathcal{E}_0^\mathrm{ab}$ is the completion of the maximal abelian extension of $E_0$ (inside the completion of an 
algebraic closure of $\Q_p$ containing $M_0$), and for each nonnegative integer $k$ we let 
$M_0\subset M_k\subset \mathcal{E}_0^\mathrm{ab}$ be the intermediate extension satisfying
$$
\co_{E_0}^\times/(\Z_p+p^k\co_{E_0})^\times\iso \Gal(M_k/M_0).
$$ 
Note that $M_k/\QQ_p$ is Galois, as $(\Z_p+p^k\co_{E_0})^\times$ is stable under the action of $\Gal(E_0/\Q_p)$.   Let $W_k$ 
denote the ring of integers of $M_k$.

Recall that $\Z_p[\gamma_0]=\Z_p+p^{c_0}\co_{E_0}$. For each $0\le k \le c_0$ Gross \cite{gross86} (see also \cite{ARGOS-8}) has 
constructed  $[M_k:\QQ_p]$ distinct surjective $\ZZ_p$-algebra homomorphisms $R_{\mathfrak{Y}_0}\map{}W_k$.   A deformation  
$$
(\mathfrak{G}_0,\rho_0) \in \mathfrak{Y}_0(W_k)\iso \Hom_{\ProArt}(  R_{\mathfrak{Y}_0}, W_k)
$$
corresponding to such a surjection is called  a \emph{quasi-canonical} lift of level $k$.  A quasi-canonical lift of level $0$ is called a 
\emph{canonical} lift.      The action of $\Z_p[\gamma_0]$ on a quasi-canonical lift  $\mathfrak{G}_0$ of level $k$ can be extended to 
an action of the larger order $\Z_p+p^k\co_{E_0}$, and  this enlarged action makes the $p$-adic Tate module  
$\mathrm{Ta}_p(\mathfrak{G}_0)$ into a free $\Z_p+p^k\co_{E_0}$-module of rank one.   Furthermore, if 
$\phi_0 : R_{\mathfrak{Y}_0} \map{}W_k$ is the homomorphism defining a quasi-canonical lift of level $k$ then  for any 
$\xi\in\co_{E_0}^\times$   there is a commutative diagram 
$$
\xymatrix{
{ R_{\mathfrak{Y}_0}  } \ar[r]^{ \phi_0 } \ar[d]_{\xi}   &  {  W_k  } \ar[d]^{\mathrm{rec}(\xi^{-1})}  \\
{ R_{\mathfrak{Y}_0}  } \ar[r]_{\phi_0}   &  {  W_k  } 
}
$$
where the vertical arrow  on the left is the automorphism of (\ref{E_0 action}).

\begin{Prop}[Gross, Keating]\label{Prop:gross-keating}
 The closed immersion of formal schemes $\mathfrak{Y}_0\map{}\mathfrak{M}_{0}$ can be identified with
$$
\Spf( \ZZ_p[[x]]/I_{0} )\map{}\Spf( \ZZ_p[[x]])
$$
where the ideal $I_{0}$ is generated by a power series  of the form
$$
g_{c_0}(x)=\prod_{k=0}^{c_0} \varphi_k(x)
$$
with each $\varphi_k(x)$ an Eisenstein polynomial  satisfying $\ZZ_p[[x]]/(\varphi_k)\iso W_k$.
\end{Prop}

\begin{proof}
By \cite[Theorem 3.8]{ARGOS-7} we may fix an isomorphism $R_{\mathfrak{M}_0}\iso \ZZ_p[[x]]$.  By \cite[Theorem 5.1]{ARGOS-8} 
the quotient $R_{\mathfrak{Y}_0}$ then has the form $\ZZ_p[[x]]/(g)$ for some power series $g(x)$.  By the Weierstrass preparation 
theorem we may choose $g(x)$ to be a polynomial of degree $f$ satisfying $g(x)\equiv x^f\pmod{p}$.  The quasi-canonical lifts 
described above give, for each $0\le k\le c_0$, a surjection $R_{\mathfrak{Y}_0}\map{}W_k$ whose kernel is generated by an  
Eisenstein polynomial $\varphi_k(x)$ which must divide $g(x)$.  This implies that we may factor 
$g(x)=\varphi_0(x)\cdots \varphi_{c_0}(x) h(x)u(x)$ with $u(x)$ a unit in $\ZZ_p[[x]]$ and $h(x)$ a polynomial of degree $e$ satisfying 
$h(x)\equiv x^e\pmod{p}$.  Thus 
$$
  \length_{\F} (R_{\mathfrak{Y}_0}/(p)) = f = e+\sum_{k=0}^{c_0}[M_k:\QQ_p].
$$
But calculations of  Keating give an explicit formula for  $  \length_{\F} (R_{\mathfrak{Y}_0}/(p))$. Using the remarks following 
\cite[Theorem 1.1]{keating88} one can show that
$$
 \length_{\F} (R_{\mathfrak{Y}_0}/(p))  =  
 \begin{cases}
  2(1+p+\cdots +p^{c_0-1}) +p^{c_0}  & \mathrm{if\ } E_0/\Q_p \mathrm{\ is\ unramified} \\
 2(1+p+\ldots+p^{c_0})  & \mathrm{if\ } E_0/\Q_p \mathrm{\ is \ ramified.} 
 \end{cases}
$$
 Comparing with the values (for $k>0$)
$$
[M_k:\QQ_p] =  \begin{cases}  (p+1) p^{k-1} & \mathrm{if\ } E_0/\Q_p \mathrm{\ is\ unramified} \\
2p^k & \mathrm{if\ } E_0/\Q_p \mathrm{\ is \ ramified}
   \end{cases}
$$  
 we deduce that $e=0$, and so $h(x)$ is a unit.  
\end{proof}


\subsection{$E_0$ unramified}
\label{ss:unr horizontal}


Throughout all of \S \ref{ss:unr horizontal} we assume that $E_0/\Q_p$ is unramified. The calculation of the horizontal components of 
$\mathfrak{Y}$ will be an easy consequence of Proposition \ref{Prop:CM orbits} and  Proposition \ref{Prop:first reflex} below.  First we 
need a simple lemma.

\begin{Lem}\label{Lem:degenerate deformation}
If   $\Z_{p^2}[\gamma]=\co_E$ then $\mathfrak{Y}\iso \mathfrak{Y}_0\iso\Spf(\ZZ_p).$
  \end{Lem}

\begin{proof}
As we assume that $E_0$ is unramified $\co_{E_0} \iso \Z_{p^2}$ as $\Z_p$-algebras, and $\co_E\iso \co_{E_0}\times\co_{E_0}.$
The two idempotents on the right determine a splitting 
$$
\mathfrak{g}\iso \mathfrak{g}_0\otimes_{\Z_p} {\Z_{p^2}} \iso   \mathfrak{g}_0\otimes_{\co_{E_0}} \co_E 
   \iso \mathfrak{g}_0\times\mathfrak{g}_0
$$ 
and similarly for any deformation of $\mathfrak{g}$ for which the $\co_E$-action lifts.  This splitting of deformations of $\mathfrak{g}$ 
induces an isomorphism of formal schemes over $\ZZ_p$
$$
\mathfrak{Y}\iso\mathfrak{Y}_0\times_{\Spf(\ZZ_p)}\mathfrak{Y}_0.
$$ 
 But according to Proposition \ref{Prop:gross-keating}, applied with $\Z_p[\gamma_0]=\co_{E_0}$ and $c_0=0$,  we have 
 $\mathfrak{Y}_0\iso \Spf(\ZZ_p)$.  The claim follows.
\end{proof}

\begin{Prop}\label{Prop:first reflex}
Let $R$ be the ring of integers of a finite extension of $\QQ_p$.  In the terminology of Definition \ref{Def:invariants}, if two elements of 
$\mathfrak{Y}(R)$ have the same geometric CM order, then they have the same reflex type.
\end{Prop}

\begin{proof}
Fix a deformation  $(\mathfrak{G}, \rho )\in \mathfrak{Y}(R)$ and  abbreviate $\co=\co(\mathfrak{G})$ for the geometric CM order of 
$\mathfrak{G}$, so that $\mathrm{Ta}_p(\mathfrak{G})$ is free of rank one over $\co$.   Let $s\in\Z$ be defined by 
$$\co = \Z_{p^2}+p^s\co_E.$$   As in Remark \ref{Rem:reflex pairs}  the reflex type of $\mathfrak{G}$ is determined by the action of 
$\Z_p[\gamma_0]$ on $\mathrm{Lie}(\mathfrak{G})$.  We will prove that this action is through  $\Psi:\co_{E_0}\map{}\ZZ_p$ if $s$ is 
even and through  $\overline{\Psi} : \co_{E_0}\map{}\ZZ_p$ if $s$ is odd, where $\Psi$ is the homomorphism defined in
 \S \ref{ss:dieudonne} and $\overline{\Psi}(x)=\Psi(x^\sigma)$.

  Using \cite[Theorem 1.3]{waterhouse72} we see that there is  $p$-Barsotti-Tate group $\mathfrak{G}^*$  with $\co_E$-action 
  satisfying 
  $$
  \mathrm{Ta}_p(\mathfrak{G}^*)\iso \mathrm{Ta}_p(\mathfrak{G})\otimes_\co \co_E
  $$  
and that the $\co$-linear inclusion of $\mathrm{Ta}_p(\mathfrak{G})$ into $\mathrm{Ta}_p(\mathfrak{G}^*)$ arises from an  
$\co$-linear isogeny of $p$-Barsotti-Tate groups $f:\mathfrak{G}\map{}\mathfrak{G}^*$ of degree $p^{2s}$.  Let $\mathfrak{g}^*$ 
denote the reduction of $\mathfrak{G}^*$ to $\F$ with its action of $\co_{E}$.  The claim is that one can find a $p$-Barsotti-Tate group 
$\mathfrak{g}_0^*$ over $\F$ equipped with an action of $\co_{E_0}$ in such a way that $\mathfrak{g}^*$ is $\co_{E}$-linearly 
isomorphic to $\mathfrak{g}_0^*\otimes\Z_{p^2}$.  Let $(D^*,F,V)$ be the covariant Dieudonn\'e module of  $\mathfrak{g}^*$ and 
identify (Lemma \ref{Lem:standard display})  the covariant Dieudonn\'e module of $\mathfrak{g}_0$ with its $\co_{E_0}$-action with 
the standard superspecial Dieudonn\'e module $(D_0,F,V)$  with its normalized action of $\co_{E_0}$, as defined in 
\S \ref{ss:dieudonne}.  The Dieudonn\'e module of $\mathfrak{g}$ is then  $D=D_0\otimes_{\Z_p}\Z_{p^2}$ with its $\Z_{p^2}$-linear 
extensions of $F$ and $V$.   Let $\{\epsilon_1,\epsilon_2\}$ be the idempotents in  
$\Z_{p^2}\otimes_{\Z_p} \ZZ_p\iso \ZZ_p\times\ZZ_p$ indexed  so that 
$$
(\alpha\otimes 1)\epsilon_1=(1\otimes\alpha)\epsilon_1
\qquad
(\alpha\otimes 1)\epsilon_2=(1\otimes\alpha^\sigma)\epsilon_2.
$$
 The $\ZZ_p$-module $D$ then has
$$
e_1=\epsilon_1 e_0\qquad e_2=\epsilon_2 e_0\qquad  f_1=\epsilon_1f_0 \qquad f_2=\epsilon_2 f_0
$$
as a basis, the action of $\alpha\in\Z_{p^2}$ is via  (\ref{unr display action}) and the action of $F$ (and $V$) is via
$$
e_1\mapsto f_2\qquad e_2 \mapsto f_1\qquad f_1\mapsto pe_2\qquad f_2  \mapsto pe_1.
$$
Using the reduction of the isogeny $f$ to identify $D^*$ with an $\co_E$-stable superlattice of $D$ in $D\otimes_{\ZZ_p}\QQ_p$ we 
find that the action of $\co_E\otimes_{\Z_p}\ZZ_p\iso (\ZZ_p)^4$ decomposes $D^*$ into a direct sum of four rank one $\ZZ_p$-
submodules, each spanned by some $\QQ_p$-multiple of one of the four basis elements $\{e_1,e_2,f_1,f_2\}$.  The condition that 
$D^*$ be stable under $F$ and $V$ then forces $D^*$ to have the form
$$
D^*= \frac{1}{p^a} \ZZ_p e_1 \oplus \frac{1}{p^{b}} \ZZ_p e_2 \oplus  \frac{1}{p^{b+\delta}}\ZZ_p f_1 \oplus 
\frac{1}{p^{a+\delta}}\ZZ_p f_2
$$
with $\delta\in \{0,1\}$,  $a,b\ge 0$, and $a+b+\delta =s$.  If we define $D_0^*\subset D^*$ to be the $\ZZ_p$-span of 
$$
e_0^*= \frac{1}{p^a} e_1 +\frac{1}{p^b} e_2 \qquad f_0^*= \frac{1}{p^{b+\delta}} f_1 +\frac{1}{p^{a+\delta}} f_2
$$
then $D_0^*$ is stable under the actions of $\co_{E_0}$, $F$, and $V$, and satisfies $D_0^*\otimes_{\Z_p}\Z_{p^2}\iso D^*$.  Hence 
we may take $\mathfrak{g}_0^*$ to be the $p$-Barsotti-Tate group associated to the Dieudonn\'e module $(D_0^*,F,V)$.

Fix an isomorphism $\mathfrak{g}^*\iso \mathfrak{g}_0^*\otimes\Z_{p^2}$ as in the previous paragraph and  let $\mathfrak{Y}_0^*$ 
and $\mathfrak{Y}^*$ be the functors on $\Art$ classifying, respectively,  deformations of $\mathfrak{g}_0^*$ with its $\co_{E_0}$-
action and $\mathfrak{g}^*$ with its $\co_E$-action.  By Lemma  \ref{Lem:degenerate deformation}  we have 
$$
\mathfrak{Y}^*\iso\mathfrak{Y}_0^*\iso\Spf(\ZZ_p)
$$ 
and so there is an $\co_{E}$-linear isomorphism $\mathfrak{G}^*\iso \mathfrak{G}^*_0\otimes\Z_{p^2}$ where $\mathfrak{G}_0^*$ is 
the unique deformation to $R$ of $\mathfrak{g}_0^*$ with its $\co_{E_0}$-action.    Fix an isomorphism of $\co_{E_0}$-modules 
$$
\mathrm{Ta}_p(\mathfrak{G}^*_0)\iso \co_{E_0}
$$
and let $\mathfrak{G}_0'$ be the $p$-Barsotti-Tate group over $R$ whose $p$-adic Tate module is the sublattice
$$
\mathrm{Ta}_p(\mathfrak{G}_0')\iso  \co  \subset \co_{E_0}\iso \mathrm{Ta}_p(\mathfrak{G}^*_0).
$$
Let $f_0':\mathfrak{G}_0'\map{}\mathfrak{G}_0^*$ be the associated isogeny of degree $p^s$,  set 
$\mathfrak{G}'=\mathfrak{G}'_0\otimes\Z_{p^2}$, and let $f':\mathfrak{G}'\map{}\mathfrak{G}^*$ be the isogeny  
$f'=f_0'\otimes\mathrm{id}$.  Then the $p$-adic Tate modules of $\mathfrak{G}$ and $\mathfrak{G}'$ are each free rank one $\co$-
submodules of $\mathrm{Ta}_p(\mathfrak{G}^*)\iso \co_E$ of index $p^{2s}$,  and such submodules are permuted simply transitively 
by $\co_E^\times/\co^\times$.  Thus   there is a $\xi\in \co_E^\times$ such that the automorphism of $\mathrm{Ta}_p(\mathfrak{G}^*)$ 
induced by $\xi$ carries the $p$-adic Tate module of   $\mathfrak{G}$  isomorphically onto  the $p$-adic Tate module of 
$\mathfrak{G}'$.  This isomorphism of $p$-adic Tate modules arises from an isomorphism $\xi:\mathfrak{G}\map{}\mathfrak{G}'$ 
making the diagram
$$
\xymatrix{
{  \mathfrak{G}  }  \ar[r]^\xi  \ar[d]_f &   { \mathfrak{G}'   } \ar[d]^{f'}  \\
{  \mathfrak{G}^* }  \ar[r]_{ \xi }  &  {   \mathfrak{G}^* }
}
$$
commute.  Reducing these isogenies to  $\F$ we obtain a commutative diagram of $\co_{E}$-linear isogenies
$$
\xymatrix{
{  \mathfrak{g}  }  \ar[r]^{\xi\circ \rho}  \ar[d]_{f \circ \rho} &   { \mathfrak{g}'   } \ar[d]^{f'}  \ar[r]^\iso &
  {\mathfrak{g}_0'\otimes \Z_{p^2}} \ar[d]^{f_0' \otimes\mathrm{id}}  \\
{  \mathfrak{g}^* }  \ar[r]_{ \xi }  &  {   \mathfrak{g}^* } \ar[r]^\iso &  {\mathfrak{g}_0^*\otimes \Z_{p^2}}.
}
$$
As the horizontal arrows are isomorphisms  there are  $\co_E$-linear isomorphisms 
$$
\mathrm{Lie}(\mathfrak{g})\iso \mathrm{Lie}(\mathfrak{g}')\iso \mathrm{Lie}(\mathfrak{g}_0')\otimes_{\Z_p}\Z_{p^2}.
$$
The action of $\co_{E_0}$ on $\mathrm{Lie}(\mathfrak{g})$ is through $\psi$, and hence so is the action of $\co_{E_0}$ on 
$\mathrm{Lie}(\mathfrak{g}_0')$.   Applying Lemma \ref{Lem:flip action} to the reduction of the isogeny $f_0'$,  the action of 
$\co_{E_0}$ on $\mathrm{Lie}(\mathfrak{g}_0')$ is therefore through
$$
\phi= \begin{cases}
\psi &\mathrm{if\ }s\mathrm{\ even} \\
\overline{\psi} &\mathrm{if\ }s\mathrm{\ odd.}
\end{cases}
$$
Let $\Phi:\co_{E_0}\map{}\ZZ_p$ be the unique (recall that  $E_0/\Q_p$ is unramified) $\Z_p$-algebra homomorphism lifting $\phi$, 
so that
$$
\Phi= \begin{cases}
\Psi &\mathrm{if\ }s\mathrm{\ even} \\
\overline{\Psi} &\mathrm{if\ }s\mathrm{\ odd}
\end{cases}
$$
By what we have shown, the  action of $\co_{E_0}$ on $\mathrm{Lie}(\mathfrak{G}_0^*)$ is through $\Phi$, and hence so is the 
action of $\co_{E_0}$ on 
$$
\mathrm{Lie}(\mathfrak{G}^*)\iso \mathrm{Lie}(\mathfrak{G}_0^*)\otimes_{\Z_p}\Z_{p^2}.
$$
Finally, the isogeny $f$ induces an $\Z_p[\gamma_0]$-linear injection 
$$
\mathrm{Lie}(\mathfrak{G})\map{ }\mathrm{Lie}(\mathfrak{G}^*)
$$
and we at last deduce that $\Z_p[\gamma_0]$ acts on $\mathrm{Lie}(\mathfrak{G})$ through $\Phi$.
\end{proof}

\begin{Cor}\label{Cor:improper orbits}
Let $R$ be the ring of integers of a finite extension of $\QQ_p$.  
Every $\co_{E}^\times$-orbit in $\mathfrak{Y}(R)$ contains an element of $\mathfrak{Y}_0(R)$.
\end{Cor}

\begin{proof}
Fix a deformation $(\mathfrak{G},\rho)\in\mathfrak{Y}(R)$ and let 
$$
\co=\Z_{p^2}+p^s\co_E
$$ 
be the geometric CM-order of $\mathfrak{G}$.  As $\co\supset \Z_{p^2}[\gamma]=\Z_p[\gamma_0]\otimes_{\Z_p}\Z_{p^2}$ we must 
have  $0\le s\le c_0$.  Let $R'$ be the ring of integers in a finite extension of $\mathrm{Frac}(R)$ large enough to contain a subfield 
isomorphic to $M_s$, and fix a  quasi-canonical lift $(\mathfrak{G}'_0,\rho_0')\in \mathfrak{Y}_0(W_s)$  of $\mathfrak{g}_0$ of level 
$s$.  The deformations $(\mathfrak{G},\rho)_{/R'}$ and $(\mathfrak{G}_0',\rho_0')_{/R'}\otimes\Z_{p^2}$ have the same geometric 
CM-order $\co$, and hence have the same reflex type by Proposition \ref{Prop:first reflex}.  By Proposition \ref{Prop:CM orbits} there 
is a $\xi\in\co_E^\times$ such that 
$$
\xi* (\mathfrak{G},\rho)_{/R'}  \iso (\mathfrak{G}_0',\rho_0')_{/R'}\otimes\Z_{p^2}.
$$
The pro-representability of $\mathfrak{Y}_0$ and $\mathfrak{Y}$ imply that 
$$
\mathfrak{Y}(R)\cap \mathfrak{Y}_0(R')=\mathfrak{Y}_0(R)
$$
 and  so $\xi*(\mathfrak{G},\rho) \in \mathfrak{Y}_0(R)$.
\end{proof}

The homomorphism $\co_{E}^\times \map{} \Aut_{\ProArt}(R_{\mathfrak{Y} })$ constructed in \S \ref{ss:functors} determines an action 
of $\co_{E}^\times$ on the set of components of $\mathfrak{Y}$.  To be explicit, if $\mathfrak{C}$ is the component corresponding to a 
minimal prime $\mathfrak{p}$ of $R_\mathfrak{Y}$ then $\xi*\mathfrak{C}$ corresponds to the kernel $\xi (\mathfrak{p})$  of 
$$
R_{\mathfrak{Y}} \map{\xi^{-1}}R_{\mathfrak{Y}}\map{}R_{\mathfrak{Y}}/\mathfrak{p}.
$$
Recalling Proposition \ref{Prop:gross-keating}, for each $0\le s\le c_0$ let $\mathfrak{p}_{0,s} \subset R_{\mathfrak{Y}_0}$ be the 
unique minimal prime for which $ R_{\mathfrak{Y}_0}/\mathfrak{p}_{0,s} \iso  W_s $ and define a component of $\mathfrak{Y}_0$ by 
$$
\mathfrak{C}_0(s) = \Spf(  R_{\mathfrak{Y}_0}/\mathfrak{p}_{0,s}  ).
$$ 
Let $q:R_{\mathfrak{Y}}\map{}R_{\mathfrak{Y}_0}$ be the homomorphism inducing the closed immersion 
$\mathfrak{Y}_0\map{}\mathfrak{Y}$ and write $\mathfrak{C}(s)$ for $\mathfrak{C}_0(s)$ viewed as a closed subscheme of 
$\mathfrak{Y}$, so that
$$
\mathfrak{C}(s)\iso \Spf( R_{\mathfrak{Y}}/\mathfrak{p}_s  )
$$
where $\mathfrak{p}_s=q^{-1}(\mathfrak{p}_{0,s})$.     Define a subgroup $H_s\subset\co_E^\times$ by 
$$
H_s=\co_{E_0}^\times \cdot  (\Z_{p^2}+p^s\co_E)^\times.
$$

\begin{Rem}\label{Rem:improper list}
As Proposition \ref{Prop:gross-keating} gives a complete list of the minimal primes of $R_{\mathfrak{Y}_0}$, it follows that  
$\mathfrak{C}(0),\ldots,\mathfrak{C}(c_0)$  is a complete list of the improper components of $\mathfrak{Y}$.
\end{Rem}

\begin{Prop}\label{Prop:unr orbits}
Every horizontal component of $\mathfrak{Y}$ can be expressed uniquely in the form 
$$
 \mathfrak{C}(s,\xi)\define \xi* \mathfrak{C}(s) 
$$ 
with  $0\le s\le c_0$ and  $\xi\in \co_E^\times/H_s$.
\end{Prop}

\begin{proof}
Fix a horizontal component $\mathfrak{C}=\Spf(R_{\mathfrak{Y}}/\mathfrak{p} )$ of $\mathfrak{Y}$ and   let $R$ be the ring of integers 
in the fraction field of $R_{\mathfrak{Y}}/\mathfrak{p}$, a finite extension of $\QQ_p$ by   Corollary \ref{Cor:horizontal reduced}. 
Applying   Corollary \ref{Cor:improper orbits}  to the deformation $(\mathfrak{G},\rho)\in\mathfrak{Y}(R)$ determined by the 
composition  map $R_\mathfrak{Y}\map{}  R_\mathfrak{Y}/\mathfrak{p}  \map{}R$ we find a $\xi\in \co_E^\times$ such that 
$$
\xi*(\mathfrak{G},\rho) \iso (\mathfrak{G}_0,\rho_0)\otimes\Z_{p^2}
$$
where $(\mathfrak{G}_0,\rho_0)$ is obtained as the base change of a quasi-canonical lift of level $s$ through some embedding 
$W_s\map{}R$.    The integer $s$ is determined by the condition that   $\Z_{p^2}+p^s\co_E$ is the geometric CM-order of 
$(\mathfrak{G},\rho)$.   This can be restated as saying that the composition
$$
R_\mathfrak{Y}\map{\xi^{-1}}R_\mathfrak{Y}\map{}R
$$
factors through $R_{\mathfrak{Y}_0}$, and the kernel of this composition is $\mathfrak{p}_s$.  Therefore 
$\xi(\mathfrak{p})=\mathfrak{p}_s$ and so $\xi*\mathfrak{C}=\mathfrak{C}(s)$.

The only thing left to prove is that $H_s\subset\co_E^\times$ is the stabilizer of $\mathfrak{C}(s)$.   Fix a quasi-canonical lift    
$(\mathfrak{G}_0,\rho_0 )\in \mathfrak{Y}_0(W_s)$ corresponding to some surjection of $\ZZ_p$-algebras 
$f_0:R_{\mathfrak{Y}_0} \map{} W_s$ and let
$$
(\mathfrak{G},\rho)=(\mathfrak{G}_0,\rho_0)\otimes\Z_{p^2}\in\mathfrak{Y}(W_s)
$$
be the deformation corresponding to $f=f_0\circ q$.   The action of  $\co_{E_0}^\times$  on $R_{\mathfrak{Y}_0}$ must stabilize the 
prime $\mathfrak{p}_{0,s}$, as $\mathfrak{p}_{0,s}$ is characterized as the unique minimal prime of $R_{\mathfrak{Y}_0}$ with the 
property $R_{\mathfrak{Y}_0}/\mathfrak{p}_{0,s} \iso W_s$.   Thus the action of $\co_{E_0}^\times$ on $R_{\mathfrak{Y}}$ leaves 
$\mathfrak{p}_s$ fixed,  proving that $\co_{E_0}^\times$ stabilizes $\mathfrak{C}(s)$.   Now  suppose we are given some 
$\xi\in (\Z_{p^2}+p^s\co_E)^\times$.  By the comments preceding Proposition \ref{Prop:gross-keating} the action of $\Z_p[\gamma_0]$ 
on $\mathfrak{G}_0$ can be extended to an action of $\Z_p+p^s\co_K$, and hence the action of $\Z_{p^2}[\gamma]$ on 
$\mathfrak{G}$ can be extended to an action of $\Z_{p^2}+p^s\co_E$.  The existence of a lifting of $\xi\in\Aut(\mathfrak{g})$ to an 
automorphism of $\mathfrak{G}$ is equivalent to the existence of an isomorphism of deformations
$$
(\mathfrak{G},\rho  ) \iso (\mathfrak{G},\rho \circ \xi^{-1}),
$$
which proves that $f=f\circ\xi^{-1}$.  As $\mathfrak{p}_s=\mathrm{ker}(f)$ we conclude that $\mathfrak{p}_s=\xi(\mathfrak{p}_s)$.  We 
have now shown that both $\co_{E_0}^\times$ and $(\Z_{p^2}+p^s\co_E)^\times$ are contained in the stabilizer of $\mathfrak{C}(s)$.

Now start with some $\xi\in\co_E^\times$ which satisfies $\mathfrak{p}_s=\xi(\mathfrak{p}_s)$.   This implies that the two 
homomorphisms $f,f\circ\xi^{-1}: R_\mathfrak{Y}\map{} W_s$   have the same kernel, and as  $f=f_0\circ q$ factors through the 
quotient map $q:R_{\mathfrak{Y}}\map{}R_{\mathfrak{Y}_0}$ we may also factor $f\circ\xi^{-1} =f_0'\circ q$ for some 
$f_0':R_{\mathfrak{M}_0} \map{}W_s$.  As $f_0,f_0': R_{\mathfrak{Y}_0}\map{}W_s$ are  surjective homomorphisms of $\ZZ_p$-
algebras with the same kernel $\mathfrak{p}_{0,s}$, there is a 
$$
\tau\in\Aut_{\ZZ_p}(W_s)\iso \Gal(M_s/ M_0)
$$
such that $f_0=\tau\circ f_0'$.  Recalling the discussion after Proposition \ref{Prop:gross-keating} we may write 
$\tau=\mathrm{rec}(\xi_0^{-1})$ for some $\xi_0\in\co_{E_0}^\times$ to obtain 
$$
f=f_0\circ q =f_0'\circ\xi_0 \circ q =f_0'\circ q\circ\xi_0 =f\circ\xi^{-1}\circ\xi_0.
$$
In other words there is an isomorphism of deformations
$$
(\mathfrak{G},\rho) \iso (\mathfrak{G},\rho\circ \xi^{-1} \xi_0).
$$
This implies that the automorphism $\xi\xi_0^{-1} \in \co_{E}^\times$ of $\mathfrak{g}$  lifts to an automorphism of $\mathfrak{G}$, and 
so acts on the $p$-adic Tate module of $\mathfrak{G}$.
But $\mathfrak{G}$ has geometric CM order 
 $$
 \co(\mathfrak{G})  = \Z_{p^2}+p^s\co_E
 $$
 and we conclude that $\xi\xi_0^{-1}\in (\Z_{p^2}+p^s\co_E)^\times$.  Hence $\xi\in H_s$.
\end{proof}

Note that for each $0\le s\le c_0$ the group $\co_E^\times/H_s$ has a decreasing filtration
$$\{1\}= H_s/H_s\subset H_{s-1}/H_s\subset \cdots H_0/H_s=\co_E^\times/H_s.$$

\begin{Prop}\label{Prop:unr intersection}
Suppose $0\le t < s\le c_0$, $\xi\in H_t/H_s$, and $\xi\not\in H_{t+1}/H_s$.  The horizontal component $\mathfrak{C}(s,\xi)$ of 
$\mathfrak{Y}$ defined  in Proposition \ref{Prop:unr orbits} satisfies
$$
I_{\mathfrak{M}}(\mathfrak{C}(s,\xi) , \mathfrak{M}_0) = 1+ \frac{(p+1)(p^t-1)}{p-1}.
$$
\end{Prop}

\begin{proof}
Let $(\mathfrak{G}_0,\rho_0)\in \mathfrak{Y}_0(W_s)$ be a quasi-canonical lift of level $s$ and let 
$(\mathfrak{G},\rho)=(\mathfrak{G}_0,\rho_0)\otimes\Z_{p^2}$ be its image in $\mathfrak{Y}(W_s)$.  If we write $\mathfrak{m}$ for the 
maximal ideal of $W_s$ and $R_k=W_s/\mathfrak{m}^{k+1}$ then 
$$
I_{\mathfrak{M}}(\mathfrak{C}(s,\xi) , \mathfrak{M}_0) = k+1
$$
where $k$ is the largest nonnegative  integer for which 
 $$
 (\mathfrak{G},\rho\circ\xi^{-1})_{/R_k} \in \mathrm{Image}\big( \mathfrak{M}_0(R_k) \map{}\mathfrak{M}(R_k) \big).
 $$
 The coset $\xi\in H_t/H_s$ admits a representative $\xi\in (\Z_{p^2}+p^t\co_E)^\times$ not contained in 
 $(\Z_{p^2} + p^{t+1}\co_E)^\times$, which we now fix.
 
 For any $k\ge 0$ abbreviate
 $$
 a(k)=\frac{(p+1)(p^k-1)}{p-1}.
 $$
 Keating \cite[Theorem 5.2]{keating88} has computed the endomorphism ring of $\mathfrak{G}_{0/R_k}$, and found that the largest 
 order of $\co_{E_0}$ for which the action $j_0:\co_{E_0}\map{}\End(\mathfrak{g}_0)$ lifts to the deformation 
 $(\mathfrak{G}_0,\rho_0)_{/W_k}$ is
 \begin{eqnarray*}
 \co_{E_0} &\mathrm{if} & 0\le k < a(0)+1 \\
 \Z_p + p\co_{E_0} & \mathrm{if}& a(0)+1 \le k< a(1)+1 \\
  \Z_p + p^2\co_{E_0} & \mathrm{if}& a(1)+1 \le k <  a(2)+1 \\
   & \vdots& \\
    \Z_p + p^{s-1} \co_{E_0} & \mathrm{if}& a(s-2)+1 \le k< a(s-1)+1 \\
     \Z_p + p^s\co_{E_0} & \mathrm{if}& a(s-1)+1\le k.
 \end{eqnarray*}
 We remark that Keating refers the reader to his unpublished thesis for the proof of  \cite[Theorem 5.2]{keating88}, but a complete 
 proof may be found in \cite{ARGOS-11}.   Using $\End(\mathfrak{G}_{/R_k})\iso \End(\mathfrak{G}_{0/R_k})\otimes_{\Z_p}\Z_{p^2}$ 
 we find that the largest order of $\co_{E}$ for which the action $j:\co_{E}\map{}\End(\mathfrak{g})$ lifts to the deformation 
 $(\mathfrak{G},\rho)_{/R_k}$ is
 \begin{eqnarray*}
 \co_{E} &\mathrm{if} & 0\le k <a(0)+1 \\
 \Z_{p^2} + p\co_{E} & \mathrm{if}& a(0)+1 \le k< a(1)+1 \\
  \Z_{p^2} + p^2\co_{E} & \mathrm{if}& a(1)+1 \le k <  a(2)+1 \\
   & \vdots& \\
    \Z_{p^2} + p^{s-1} \co_{E} & \mathrm{if}& a(s-2)+1 \le k< a(s-1)+1 \\
     \Z_{p^2} + p^s\co_{E} & \mathrm{if}& a(s-1)+1\le k.
 \end{eqnarray*}

Set $k=a(t)$ so that the automorphism $\xi\in (\Z_{p^2}+p^t\co_E)^\times$ of $\mathfrak{g}$ lifts to the deformation 
$(\mathfrak{G},\rho)_{/R_k}$ but not to $(\mathfrak{G},\rho)_{/R_{k+1}}$.  The existence of a lift of $\xi$ over $R_k$ is equivalent to the 
existence of  an isomorphism of deformations 
$$
(\mathfrak{G},\rho)_{/R_k}\iso(\mathfrak{G},\rho\circ\xi^{-1})_{/R_k},
$$ 
and as  $(\mathfrak{G},\rho)_{/R_k}$ lies in the image of $\mathfrak{M}_0(R_k)\map{}\mathfrak{M}(R_k)$ so does 
$(\mathfrak{G},\rho\circ\xi^{-1})_{/R_k}$.

Now set $k=a(t)+1$ and $\Delta_0=\End(\mathfrak{g}_0)$ so that the largest order $\co\subset \co_{E_0}$ for which the restriction to 
$\co$ of the action  $j_0:\co_{E_0}\map{}\Delta_0$ lifts to $(\mathfrak{G}_0,\rho_0)_{/R_k}$ is 
$$
\co=\Z_{p}+p^{t+1}\co_{E_0}.
$$  
As in \S \ref{ss:functors} we fix an embedding $\Z_{p^2}\map{}M_2(\Z_p)$ and identify the functor 
$ \mathfrak{G}'_0\mapsto \mathfrak{G}'_0\otimes\Z_{p^2}$  with the functor 
$\mathfrak{G}'_0\mapsto \mathfrak{G}'_0\times\mathfrak{G}'_0$.  This identification determines  an isomorphism 
$$
\End(\mathfrak{G}'_0\otimes\Z_{p^2})\iso M_2(\End(\mathfrak{G}'_0))
$$
and in  particular $\End(\mathfrak{g})\iso M_2(\Delta_0)$.   Moreover, one easily checks that the closed formal subscheme 
$\mathfrak{M}_0\map{}\mathfrak{M}$ is the locus of deformations for which the action of $M_2(\Z_p)$ on $\mathfrak{g}$ lifts.  To 
obtain a contradiction let us suppose that 
$$
(\mathfrak{G},\rho\circ\xi^{-1})_{/R_k} \in \mathrm{Image}\big( \mathfrak{M}_0(R_k)\map{}\mathfrak{M}(R_k) \big)
$$
so that the action of $M_2(\Z_p)$ on $\mathfrak{g}$ lifts to the deformation $(\mathfrak{G},\rho\circ\xi^{-1})_{/R_k}$, and, equivalently,  
the endomorphism $\xi \circ \mu \circ \xi^{-1}$ of $\mathfrak{g}$  lifts to the deformation $(\mathfrak{G},\rho)$ for every 
$\mu\in M_2(\Z_p)$.     Using $j_0:\co_{E_0}\map{}\Delta_0$ we view $\co_{E_0}$ as a subring of $\Delta_0$ and $\co_E$ as a 
subring of $M_2(\co_{E_0})$.  We also view 
$$
\End(\mathfrak{G}_{/R_k})\iso M_2( \End(\mathfrak{G}_{0/R_k}))
$$
as a subring of $\End(\Delta_0)$ using the reduction map $R_k\map{}\F$.  Under these identifications 
$$
\End(\mathfrak{G}_{/R_k})\cap M_2(\co_{E_0}) = M_2(\co)
$$
and we deduce   that $\xi\circ \mu\circ\xi^{-1}\in M_2(\co)$ for every $\mu\in M_2(\Z_p)$.  As the subalgebra of $M_2(\Delta_0)$ 
generated by $M_2(\Z_p)$ and $\co$ is $M_2(\co)$, we further deduce that $\xi\circ \mu\circ\xi^{-1}\in M_2(\co)$ for every 
$\mu\in M_2(\co)$.  Some elementary linear algebra then shows that  $\xi\in\co_{E_0}^\times \cdot \mathrm{GL}_2(\co)$.  Thus there 
is a $\xi_0\in\co_{E_0}$ for which $\xi\xi_0^{-1}\in M_2(\co)$, and so 
$$
\xi\xi_0^{-1}\in \End(\mathfrak{G}_{/R_k})\cap \co_E^\times = (\Z_{p^2}+p^{t+1}\co_E)^\times.
$$
  But this means that $\xi\in H_{t+1}$, a contradiction.

We have shown that $(\mathfrak{G},\rho\circ\xi^{-1})_{/R_k}$  lies in the image of $\mathfrak{M}_0(R_k)$ for $k=a(t)$ but not for 
$k=a(t)+1$, completing the proof.
\end{proof}

\begin{Cor}\label{Cor:unr proper horizontal}
We have
$$
\sum_{\mathfrak{C} } I_\mathfrak{M}(\mathfrak{C},\mathfrak{M}_0) = 
\frac{  -4p(p^{c_0} - 1) + 2c_0(p-1) + c_0p^{c_0}(p^2-1)  }{ (p-1)^2 }
$$
where the sum is over all proper horizontal components  $\mathfrak{C}\map{}\mathfrak{Y}$.
\end{Cor}

\begin{proof}
If $0\le t < s \le c_0$   then combining
$$
|H_t/H_s| = \begin{cases}
p^{s-t} & \mathrm{if\ } t\not=0 \\
p^{s-1}(p-1) & \mathrm{if\ }t=0
\end{cases}
$$
with Proposition \ref{Prop:unr intersection} shows that
$$
\sum_{ \substack{  \xi\in\co_E^\times/H_s  \\ \xi\not=1   } } I_{\mathfrak{M}}(\mathfrak{C}(s,\xi) , \mathfrak{M}_0 ) =
  p^{s-1}(p-2) + p^{s-1}(p+1)(s-1)  - \frac{  2(p^{s-1} -1 )}{p-1}.
$$
Summing over $0 \le s\le c_0$ and using Proposition \ref{Prop:unr orbits} yields the desired result.
\end{proof}


\subsection{$E_0$ ramified}
\label{ss:ram horizontal}


Throughout all of \S \ref{ss:ram horizontal} we assume that $E_0/\Q_p$ is ramified.  Suppose $R$ is the ring of integers of a finite 
extension of $\QQ_p$ and  $(\mathfrak{G},\rho)\in \mathfrak{Y}(R)$.  As in Remark \ref{Rem:reflex pairs} there is a  decomposition of 
the Lie algebra
$$
\mathrm{Lie}(\mathfrak{G}) \iso \Lambda_1\oplus\Lambda_2
$$ 
in which each $\Lambda_i$ is free of rank one over $R$, $\Z_{p^2}$ acts through $\Z_{p^2}\map{}\ZZ_p\map{}R$ on $\Lambda_1$ 
and through the conjugate embedding on $\Lambda_2$.   The action of $\Z_{p^2}[\gamma]$ on $\Lambda_i$ is through some 
homomorphism $\Phi_i:\co_{E}\map{}R$.  We have  $\Phi_1\not=\Phi_2$ after restriction to $\Z_{p^2}$, and the restriction of the pair 
$(\Phi_1,\Phi_2)$ to $\co_{E_0}$ determines the reflex type of $(\mathfrak{G},\rho)$.

\begin{Def}
If $\Phi_1=\Phi_2$ when restricted to $\co_{E_0}$ then we say that the   deformation $(\mathfrak{G}, \rho)$ has  a \emph{standard} 
reflex type.  If $\Phi_1\not=\Phi_2$ when restricted to $\co_{E_0}$  then we say that the reflex type is \emph{nonstandard}.  
\end{Def}

There is a disjoint union 
$$
\mathfrak{Y}(R)=\mathfrak{Y}^+(R)\bigcup \mathfrak{Y}^-(R)
$$
where $\mathfrak{Y}^+(R)$ is the subset of deformations having standard reflex type and $\mathfrak{Y}^-(R)$ is the subset of 
deformations with nonstandard reflex type.  Each subset is $\co_{E}^\times$-stable, as the reflex type is constant on $\co_{E}^\times$-
orbits.  Note that the above decomposition makes sense for any objects $R$ of $\ProArt$, but  is not functorial: after base change 
through a homomorphism $R\map{}R'$ one may have elements of $\mathfrak{Y}^-(R)$ whose image lies in $\mathfrak{Y}^+(R')$.    In 
this subsection we will only consider those $R$ which are integer rings of finite extensions of $\QQ_p$, and for such rings  any 
(necessarily injective) map $R\map{}R'$ induces an  inclusion $\mathfrak{Y}(R)\map{}\mathfrak{Y}(R')$ which respects the 
decompositions into standard and nonstandard reflex types.

\begin{Lem}\label{Lem:standard reflex orbits}
Let $R$ be the ring of integers of a finite extension of $\QQ_p$. Then $\mathfrak{Y}_0(R)\subset\mathfrak{Y}^+(R)$  and every 
$\co_{E}^\times$-orbit in $\mathfrak{Y}^+(R)$ contains an element of $\mathfrak{Y}_0(R)$.
\end{Lem}

\begin{proof}
First start with some $(\mathfrak{G}_0,\rho_0) \in \mathfrak{Y}_0(R)$.  The action of $\Z_p[\gamma_0]$ on 
$\mathrm{Lie}(\mathfrak{G}_0)$ is through some $\Z_p$-algebra map $\Phi_0:\co_{E_0}\map{}R$, and so the action of 
$\Z_p[\gamma_0]$ on the Lie algebra $\mathrm{Lie}(\mathfrak{G}_0)\otimes_{\Z_p}\Z_{p^2}$ of
$$
(\mathfrak{G}_0,\rho_0)\otimes \Z_{p^2} \in \mathfrak{Y}(R)
$$
is  again through $\Phi_0$.  After restriction to $\co_{E_0}$ we now have $\Phi_1=\Phi_0=\Phi_2$, and hence 
$(\mathfrak{G}_0,\rho_0)\otimes \Z_{p^2}$ has a standard reflex type, proving $\mathfrak{Y}_0(R)\subset\mathfrak{Y}^+(R)$.

Now suppose we start with some $(\mathfrak{G},\rho)\in\mathfrak{Y}^+(R)$,   so that  $\Z_p[\gamma_0]$ acts on 
$\mathrm{Lie}(\mathfrak{G})$ through some embedding $\Phi_0:\co_{E_0}\map{}R$.   Let $\Z_{p^2}+p^s\co_E$ be the geometric 
CM-order of $(\mathfrak{G},\rho)$  and  consider a quasi-canonical lift $(\mathfrak{G}_0,\rho_0)\in \mathfrak{Y}_0(W_s)$ of level $s$.   
The action of $\Z_p[\gamma_0]$ on $\mathrm{Lie}(\mathfrak{G}_0)$ is through some embedding 
$\Phi_{00}: \Z_p[\gamma_0] \map{}W_s$, and after  enlarging $R$ as in the proof of Corollary \ref{Cor:improper orbits} we assume 
that there is an  embedding of $\ZZ_p$-algebras $i:W_s\map{}R$.   As $M_s/\QQ_p$ is Galois we may choose a 
$\tau\in\Gal(M_s/\QQ_p)$ whose restriction to $\Gal(M_0/\QQ_p)$ is nontrivial, and then  
$i\circ\Phi_{00}$ and $ i\circ \tau \circ\Phi_{00}$ give the two distinct embeddings of $\Z_p[\gamma_0]$ into $R$.  Thus after possibly 
replacing $i$ by $i\circ\tau$ we may assume that $i\circ\Phi_{00}=\Phi_0$.  The action of  $\Z_p[\gamma_0]$ on the Lie algebra of 
$(\mathfrak{G}_0,\rho_0)_{/R}\in \mathfrak{Y}_0(R)$ is now through $\Phi_0$, and so the  deformation
$$
(\mathfrak{G}_0,\rho_0)_{/R} \otimes\Z_{p^2}\in\mathfrak{Y}(R)
$$
has  the same reflex type and same geometric CM-order as $(\mathfrak{G},\rho)$.  Applying Proposition \ref{Prop:CM orbits} 
completes the proof.
\end{proof}

Again let $R$ be the ring of integers in a finite extension of $\QQ_p$.  We now explain how to construct elements in 
$\mathfrak{Y}^-(R)$.  As we assume that $E_0/\Q_p$ is a ramified field extension, $E$ is a biquadratic field extension of $\Q_p$ 
containing a unique quadratic subfield $E_0'$ which is ramified over $\Q_p$ and not equal to $E_0$.  Choose an action  
$j_0':\co_{E_0'}\map{}\End(\mathfrak{g}_0)$ of $\co_{E_0'}$ on $\mathfrak{g}_0$ and let $j'$ be the induced $\Z_{p^2}$-linear action
 $$
 j':\co_E\map{}\End(\mathfrak{g})
 $$
of $\co_E=\co_{E_0'}\otimes\Z_{p^2}$ on $\mathfrak{g}$.

\begin{Lem}
There is a $\Z_{p^2}$-linear automorphism $w:\mathfrak{g}\map{}\mathfrak{g}$ satisfying
$$
w\circ  j'(x)  =j(x) \circ w 
$$
for every $x\in\co_E$.
\end{Lem}

\begin{proof}
Abbreviate $\Delta_0=\End(\mathfrak{g}_0)$ and $\Delta=\End(\mathfrak{g})$ so that $\Delta=\Delta_0\otimes_{\Z_p}\Z_{p^2}$.  Thus 
$\Delta_0$ is the maximal order in a quaternion division algebra over $\Q_p$ and $\Delta$ is isomorphic to an Eichler order of level 
one in $M_2(\Z_{p^2})$.  Fix such an isomorphism, so that $\Delta$ acts on a two-dimension $\Q_{p^2}$-vector space  in a way 
which stabilizes a $\Z_{p^2}$-lattice $\Lambda_1$ and  sublattice $\Lambda_2\subset\Lambda_1$ with 
$\Lambda_1/\Lambda_2\iso\F_{p^2}$.  Let $L_i$ denote $\Lambda_i$ viewed as an $\co_{E}$-module via $j$, and let $L_i'$ denote 
$\Lambda_i$ viewed as an $\co_E$-module via $j'$.  If $\mathfrak{q}\subset\co_E$ is the maximal ideal then $L_2=\mathfrak{q}L_1$, 
$L_2'=\mathfrak{q}L_1'$, and there is an isomorphism of torsion free rank one $\co_E$-modules $w: L_1'\map{}L_1$.  As $w$ takes 
$L_2'$ to $L_2$, we have found the desired $w\in\End(\Lambda_2)\cap\End( \Lambda_1) = \Delta$.
\end{proof}

Choose any $\gamma_0'\in \co_{E_0'}$ for which  $\Z_p[\gamma_0']=\Z_p+p^{c_0}\co_{E_0'}$ and set 
$$
\gamma'=\gamma_0'\otimes 1\in \co_E\iso \co_{E_0'} \otimes_{\Z_p} \Z_{p^2}
$$ 
so that 
$$
\Z_{p^2}[\gamma'] =\Z_{p^2}+p^{c_0}\co_E = \Z_{p^2}[\gamma].
$$  
Let $\mathfrak{Y}_0' \map{}\mathfrak{M}_0$ be the closed formal subscheme classifying deformations of $\mathfrak{g}_0$ for which 
the action  $j_0': \Z_p[\gamma_0'] \map{}\End(\mathfrak{g}_0)$ lifts, and let $\mathfrak{Y}'\map{}\mathfrak{M}$ be the closed formal 
subscheme classifying deformations of $\mathfrak{g}$ for which the action $j':\Z_{p^2}[\gamma'] \map{}\End(\mathfrak{g})$ lifts.  There 
is then an isomorphism of functors 
$$
\mathfrak{Y}' \map{w}\mathfrak{Y}
$$
 defined by
$$
(\mathfrak{G},\rho)\mapsto    (\mathfrak{G},\rho\circ w^{-1})
$$
which respects with the action of $\co_E^\times$ on each side (the actions on the left and right being defined using the actions   $j'$ 
and  $j$ of  $\co_{E}^\times$     on $\mathfrak{g}$, respectively).  Define $R_{\mathfrak{Y}'_0}$ and $R_{\mathfrak{Y}'}$ by
$$
\mathfrak{Y}_0'=\Spf(R_{\mathfrak{Y}'_0}) \qquad
\mathfrak{Y}'=\Spf(R_{\mathfrak{Y}' } ).
$$

We point out that  Proposition \ref{Prop:gross-keating} applies equally well with $\mathfrak{Y}_0$ replaced by $\mathfrak{Y}'_0$, and 
the theory of quasi-canonical lifts applies equally well to the deformation problem $\mathfrak{Y}_0'$ as it does to $\mathfrak{Y}_0$.  
We also  point out that the isomorphism $w:\mathfrak{Y}'\map{} \mathfrak{Y}$ does not restrict to an isomorphism 
$\mathfrak{Y}_0'\map{}\mathfrak{Y}_0$, as the automorphism $w$ of $\mathfrak{g}$ does not arise from an automorphism of 
$\mathfrak{g}_0$.  Instead we have the following pretty situation.

\begin{Lem}\label{Lem:nonstandard reflex orbits}
Let $R$ be the ring of integers of a finite extension of $\QQ_p$.  The bijection $w:\mathfrak{Y}'(R)\map{}\mathfrak{Y}(R)$ restricts to 
an injection $w: \mathfrak{Y}_0'(R)\map{}\mathfrak{Y}^-(R)$, and every $\co_E^\times$-orbit in $\mathfrak{Y}^-(R)$ contains an 
element  in the image of $\mathfrak{Y}_0'(R)$.
\end{Lem}

\begin{proof}
We will repeatedly use the fact that for any two distinct embeddings of $E$ into a field there is a unique quadratic subfield of $E$ on 
which those embeddings agree.  Fix  $(\mathfrak{G}'_0,\rho'_0)\in \mathfrak{Y}_0'(R)$ and define
$$
(\mathfrak{G}',\rho') = (\mathfrak{G}'_0,\rho'_0)\otimes\Z_{p^2} \in \mathfrak{Y}'(R).
$$
 Let $\Phi'_0: \co_{E_0'}  \map{}R$ be the $\ZZ_p$-algebra homomorphism giving the action of $\Z_p[\gamma_0']$  on 
 $\mathrm{Lie}(\mathfrak{G}_0')$.   Decomposing the Lie algebra 
$$
\mathrm{Lie}(\mathfrak{G}')\iso \Lambda_1\oplus\Lambda_2
$$
as before, the order $\Z_{p^2}[\gamma']$ acts on $\Lambda_i$ through an embedding $\Phi_i: \co_E \map{}R$ which restricts to 
$\Phi'_0$ on $\co_{E_0'}$.   The   restrictions of $\Phi_1$ and $\Phi_2$ to $\Z_{p^2}$ are distinct, and it follows that the restrictions of 
$\Phi_1$ and $\Phi_2$ to $\co_{E_0}$ are also distinct.

 If we now define 
$$
(\mathfrak{G},\rho)=(\mathfrak{G}',\rho'\circ w^{-1})\in\mathfrak{Y}(R)
$$
and let $\mathrm{id}:\mathfrak{G}\map{}\mathfrak{G}'$ be the identity map then $\mathrm{id}$ is a $\Z_{p^2}+p^{c_0}\co_E$-linear 
isomorphism of $p$-Barsotti-Tate groups $\mathfrak{G}\iso \mathfrak{G}'$ (where the action of $ \Z_{p^2}[\gamma]$  on 
$\mathfrak{G}$ lifts the action  $j $  on $\mathfrak{g}$, and the action of $\Z_{p^2}[\gamma']$ on $\mathfrak{G}'$ lifts the action $ j' $ 
on  $\mathfrak{g}$), and so  $\mathrm{Lie}(\mathfrak{G})\iso \mathrm{Lie}(\mathfrak{G}')$ as modules over
$$
\Z_{p^2}[\gamma]\otimes_{\Z_p}R \iso \Z_{p^2}[\gamma']\otimes_{\Z_p}R.
$$
 By the previous paragraph the reflex type of  $(\mathfrak{G},\rho)$ is nonstandard, proving $(\mathfrak{G},\rho)\in\mathfrak{Y}^-(R)$.

Now start with any $(\mathfrak{G},\rho)\in \mathfrak{Y}^-(R)$, and let $\Z_{p^2}+p^s\co_E$ be the geometric CM-order of 
$(\mathfrak{G},\rho)$.   Decompose 
$$
\mathrm{Lie}(\mathfrak{G})\iso \Lambda_1\oplus\Lambda_2
$$
as above, so that $\Z_{p^2}[\gamma]$ acts on $\Lambda_i$ through some $\Phi_i: \co_E \map{}R$.  As $\Phi_1$ and $\Phi_2$ are 
distinct when restricted  to $\Z_{p^2}$ and to $\co_{E_0}$ (by definition of nonstandard reflex type), they must agree when restricted 
to $\co_{E_0'}$.  Call the common restriction $\Phi_0':\co_{E_0'} \map{}R$.  Choose a quasi-canonical lift 
$(\mathfrak{G}'_0,\rho_0')\in\mathfrak{Y}_0'(W_s)$ of level $s$ of $\mathfrak{g}_0$ with its action of $j_0'$, and (after enlarging $R$ 
as in the proof of Corollary \ref{Cor:improper orbits}) an embedding $W_s\map{}R$ of $\ZZ_p$-algebras.  As in the proof of Lemma 
\ref{Lem:standard reflex orbits} this can be done in such a way that the action of $\Z_p[\gamma_0']$ on the Lie algebra of 
$\mathfrak{G}'_{0/R}$ is through $\Phi_0'$.   The image of the deformation $(\mathfrak{G}'_0,\rho'_0)$ under
$$
\mathfrak{Y}_0'(W_s)\map{} \mathfrak{Y}_0'(R)\map{w} \mathfrak{Y}^-(R)
$$
  then has the same geometric CM-order and reflex type as $(\mathfrak{G},\rho)$, and so lies in the same $\co_E^\times$-orbit by 
  Proposition \ref{Prop:CM orbits}.
  \end{proof}

The isomorphism of functors $w:\mathfrak{Y}'\map{}\mathfrak{Y}$ induces an isomorphism 
$w:R_{\mathfrak{Y}'}\map{}R_{\mathfrak{Y}}$ in such a way that the composition
$$
\Hom_{\ProArt}(R_{\mathfrak{Y}'}, - )\iso \mathfrak{Y}'  (-)  \map{w}\mathfrak{Y}(-) \iso \Hom_{\ProArt}(R_\mathfrak{Y}, - )
$$
has the form $f\mapsto f\circ w^{-1}$.  Let
$$
q:R_{\mathfrak{Y}}\map{}R_{\mathfrak{Y}_0}
\qquad
q':R_{\mathfrak{Y}'}\map{} R_{\mathfrak{Y}_0'}
$$
be the natural surjections.  As in \S \ref{ss:unr horizontal}, for each $0\le s\le c_0$ let $\mathfrak{p}_{0,s}$ be the unique minimal 
prime of $R_{\mathfrak{Y}_0}$ for which  $R_{\mathfrak{Y}_0}/\mathfrak{p}_{0,s} \iso W_s$ and  let 
$\mathfrak{p}_s=q^{-1}(\mathfrak{p}_{0,s})$ be the corresponding prime of $R_{\mathfrak{Y}}$.   Similarly let $\mathfrak{p}_{0,s}'$ be 
the minimal prime of $R_{\mathfrak{Y}'_0}$ corresponding to a level $s$ quasi-canonical lift of the action 
$j_0':\co_{E_0}\map{}\End(\mathfrak{g}_0)$, so that    $R_{\mathfrak{Y}'_0}/\mathfrak{p}'_{0,s} \iso W'_s$ with $W'_s$ constructed 
exactly as in \S \ref{ss:quasi-canonical} but with $E_0$ and $j_0$ replaced everywhere by $E_0'$ and $j_0'$.  Let 
$\mathfrak{p}'_s=q'^{-1}(\mathfrak{p}'_{0,s})$ be the corresponding prime of $R_{\mathfrak{Y}'}$.  
For each $0\le s\le c_0$ we now define two distinguished components of $\mathfrak{Y}$,
$$
\mathfrak{C}(s) = \Spf(R_\mathfrak{Y}/\mathfrak{p}_s )  
= \Spf(R_{\mathfrak{Y}_0}/\mathfrak{p}_{0,s} )
$$
and
$$
\mathfrak{C}'(s) =  \Spf(R_\mathfrak{Y}/w(\mathfrak{p}'_s) )  
\iso  \Spf(R_{\mathfrak{Y}'}/\mathfrak{p}'_{0,s} ) 
= \Spf(R_{\mathfrak{Y}_0'}/\mathfrak{p}'_{0,s} ),
$$
and   two subgroups of $\co_E^\times$
$$
H_s=\co_{E_0}^\times  \cdot (\Z_{p^2}+p^s\co_E)^\times
\qquad
H_s'=\co_{E_0'}^\times \cdot (\Z_{p^2}+p^s\co_E)^\times.
$$

\begin{Rem}
Exactly as in  Remark \ref{Rem:improper list},   $\mathfrak{C}(0),\ldots,\mathfrak{C}(c_0)$ is a complete list of the improper 
components of $\mathfrak{Y}$.
\end{Rem}

We will say that a horizontal component $\mathfrak{C}=\Spf(R_{\mathfrak{Y}}/\mathfrak{p})$  of $\mathfrak{Y}$ is \emph{standard} if 
the deformation corresponding to the quotient map  $R_{\mathfrak{Y}}\map{} R_{\mathfrak{Y}}/\mathfrak{p} $ has standard reflex type, 
and say that $\mathfrak{C}$ is \emph{nonstandard} otherwise.

\begin{Prop}\label{Prop:ram orbits}
 Every standard horizontal component of $\mathfrak{Y}$ can be written uniquely in the form 
$$
\mathfrak{C}(s,\xi)  \define \xi * \mathfrak{C}(s)
$$
 for $0\le s\le c_0$ and $\xi\in \co_E^\times/H_s$, and every nonstandard horizontal component of $\mathfrak{Y}$ can be written 
 uniquely in the form 
 $$
 \mathfrak{C}'(s,\xi)   \define \xi * \mathfrak{C}'(s)
 $$
 for $0\le s\le c_0$ and $\xi\in \co_E^\times/ {H'_s}$. To be clear: in both instances the action of  $\xi*$ on $\mathfrak{Y}$ is that of 
 \S \ref{ss:functors} defined using the action $j:\co_E^\times\map{}\End(\mathfrak{g})$.
 \end{Prop}

\begin{proof}
The proof is the same as the proof of Proposition \ref{Prop:unr orbits}, almost verbatim.  In place of Corollary \ref{Cor:improper orbits} 
one uses Lemma \ref{Lem:standard reflex orbits} in the standard case and Lemma \ref{Lem:nonstandard reflex orbits} in the 
nonstandard case.
\end{proof}

\begin{Prop}\label{Prop:ram intersection I}
Suppose $0\le t < s\le c_0$, $\xi\in H_t/H_s$, and $\xi\not\in H_{t+1}/H_s$.  The horizontal component 
$\mathfrak{C}(s,\xi)\map{} \mathfrak{Y}$ defined in Proposition \ref{Prop:ram orbits}  satisifes
$$
I_{\mathfrak{M}}(\mathfrak{C}(s,\xi) , \mathfrak{M}_0) = 1+ p^t + \frac{(p+1)(p^t-1)}{p-1}.
$$
\end{Prop}

\begin{proof}
Let $(\mathfrak{G}_0,\rho_0)\in \mathfrak{Y}_0(W_s)$ be a quasi-canonical lift of level $s$.   For any $k\ge 0$ set 
$R_k=W_s/\mathfrak{m}^{k+1}$ and abbreviate
 $$
 b(k)= p^k+ \frac{(p+1)(p^k-1)}{p-1}.
 $$
 Keating \cite[Theorem 5.2]{keating88} has computed the endomorphism ring of $\mathfrak{G}_{0/R_k}$, and found that the largest 
 order of $\co_{E_0}$ for which the action $j_0:\co_{E_0}\map{}\End(\mathfrak{g}_0)$ lifts to the deformation 
 $(\mathfrak{G}_0,\rho_0)_{/R_k}$ is
 \begin{eqnarray*}
 \co_{E_0} &\mathrm{if} & 0\le k < b(0)+1 \\
 \Z_p + p\co_{E_0} & \mathrm{if}& b(0)+1 \le k< b(1)+1 \\
  \Z_p + p^2\co_{E_0} & \mathrm{if}& b(1)+1 \le k <  b(2)+1 \\
   & \vdots& \\
    \Z_p + p^{s-1} \co_{E_0} & \mathrm{if}& b(s-2)+1 \le k< b(s-1)+1 \\
     \Z_p + p^s\co_{E_0} & \mathrm{if}& b(s-1)+1\le k.
 \end{eqnarray*}
 Using these formulas, the proof is a direct imitation of that of Proposition \ref{Prop:unr intersection}.
\end{proof}

\begin{Prop}\label{Prop:ram intersection II}
Suppose $0  \le s\le c_0$ and $\xi\in \co_E^\times/H_s$.    If $\ord_p(\mathrm{disc}(E_0))=1$ then 
the horizontal component $\mathfrak{C}'(s,\xi) \map{} \mathfrak{Y}$ defined in Proposition \ref{Prop:ram orbits} satisifes
$$
I_{\mathfrak{M}}(\mathfrak{C}'(s,\xi) , \mathfrak{M}_0) =1.
$$
Note that  $\ord_p(\mathrm{disc}(E_0))=1$ is equivalent to $p>2$.
\end{Prop}

\begin{proof}
Fix an isomorphism $\mathfrak{C}'(s,\xi) \iso \Spf(W_s')$, let $\mathfrak{m}$ be the maximal ideal of $W_s'$, and set 
$R=W_s'/\mathfrak{m}^2$ so that $R$ is isomorphic to the ring of dual numbers $\F[\epsilon]$ of $\F$.    The resulting closed 
immersion of formal schemes $\Spf(W_s')\map{}\mathfrak{Y}$ is given by the composition
\begin{equation}\label{twisty-twisty}
R_\mathfrak{Y} \map{\xi^{-1}} R_\mathfrak{Y}\map{w^{-1}}R_{\mathfrak{Y}'} 
\map{q'} R_{\mathfrak{Y}_0'} \map{f_s}W_s'
\end{equation}
where the final arrow $f_s$ defines a level $s$ quasi-canonical lift $(\mathfrak{G}_0',\rho_0')\in\mathfrak{Y}_0(W_s')$ of  
$\mathfrak{g}_0$ with respect to the embedding $j_0':\co_{E_0'}\map{} \End(\mathfrak{g}_0)$.    Setting 
$(\mathfrak{G}',\rho')=(\mathfrak{G}'_0,\rho'_0)\otimes \Z_{p^2}$, the composition (\ref{twisty-twisty}) corresponds to the deformation 
$$
(\mathfrak{G}'', \rho'')= (\mathfrak{G}',\rho'\circ w^{-1}\circ \xi^{-1} ) \in \mathfrak{Y}(W_s').
$$
  It follows from the formulas of Keating cited in the proof of Proposition \ref{Prop:ram intersection I} that the action of $\Z_{p^2}[\tau]$ 
  on the deformation $\mathfrak{G}'_{/R}$ extends (uniquely) to an action of $\co_{E}$, and hence the same is true of the deformation 
  $\mathfrak{G}''_{/R}$.We will show first that the resulting actions of $\co_{E_0}$ on the two summands 
$$
\mathrm{Lie}(\mathfrak{G}''_{/R})\iso \Lambda_1\oplus\Lambda_2
$$
are through distinct homomorphisms $\co_{E_0}\map{}R$.

 Consider the diagram
$$
\xymatrix{
{ R_{\mathfrak{Y}_0'} } \ar[d]\ar[r]^{f_s}  &   { W_s' }  \ar[d]  \\
{  W_0'  }  \ar@{-->}[r]_? & R
}
$$
in which the vertical arrow on the left corresponds to a  canonical lift 
$(\mathfrak{G}^\dagger_{0},\rho^\dagger_{0})\in\mathfrak{Y}'_0(W_0')$  of $\mathfrak{g}_0$ with respect to the action 
$j_0':\co_{E'_0} \map{}\End(\mathfrak{g}_0)$.  Choose an isomorphism $R_{\mathfrak{M}_0}\iso \ZZ_p[[x]]$ and let $\alpha$ be the 
image of $x$ under 
$$
R_{\mathfrak{M}_0}\map{}R_{\mathfrak{Y}'_0}\map{f_s}W_s'\map{}R.
$$  
There is an isomorphism $W_0'\iso \ZZ_p[[x]]/(\varphi_0')$ with $\varphi_0'$ Eisenstein of degree two, and as $\varphi_0'(\alpha)=0$ 
we find that there is a surjection $W_0'\map{}R$ making the above diagram commute.  In other words, if we set 
$(\mathfrak{G}^\dagger,\rho^\dagger) =(\mathfrak{G}^\dagger_{0} ,\rho_{0}^\dagger)\otimes\Z_{p^2}$ then there is an isomorphism of 
deformations in $\mathfrak{Y}(R)$
$$
(\mathfrak{G}'',\rho'')_{/R} \iso (\mathfrak{G}^\dagger,\rho^\dagger \circ w^{-1}\circ \xi^{-1})_{/R}
$$
and so  $(\mathfrak{G}'',\rho'')_{/R}$ admits a lift 
$$
(\mathfrak{G}^\dagger,\rho^\dagger \circ w^{-1}\circ \xi^{-1}) \in \mathfrak{Y}^-(W_0')
$$
to $W_0'$  of nonstandard reflex type with the property that  the action $j :\co_E\map{} \End(\mathfrak{g})$ also lifts.  The action of 
$\co_{E_0}$  on each  $\Lambda_i$  is therefore through some   $\co_{E_0}\map{\phi_i}W_0' \map{}R$, and $\phi_1\not=\phi_2$.  
The hypothesis that $\ord_p(\mathrm{disc}(E_0))=1$ implies that the nontrivial Galois automorphism of $W_0'/\ZZ_p$ (which 
interchanges $\phi_1$ and $\phi_2$) remains nontrivial modulo the square of the maximal ideal of $W_0'$, and so the actions of 
$\co_{E_0}$ on  $\Lambda_1$ and $\Lambda_2$ are through distinct homomorphisms $\co_{E_0} \map{} R$.

Is it possible that $(\mathfrak{G},\rho)_{/R}$ lies in the image of $\mathfrak{M}_0(R)\map{}\mathfrak{M}(R)$?  If so, 
then $(\mathfrak{G},\rho)_{/R}\iso (\mathfrak{G}_0^\sim,\rho_0^\sim)\otimes\Z_{p^2}$ for some 
$ (\mathfrak{G}_0^\sim,\rho_0^\sim)\in\mathfrak{Y}_0(R)$ corresponding to a surjective $\ZZ_p$-algebra map 
$R_{\mathfrak{Y}_0} \map{}R$.  As above, the fact that $W_0/\ZZ_p$ is ramified implies that this map can be factored as 
$R_{\mathfrak{Y}_0}\map{}W_0\map{}R$ where the first arrow corresponds to a canonical lift in $ \mathfrak{Y}_0(W_0)$.  This implies 
that $(\mathfrak{G}'',\rho'')_{/R}$ admits a lift to $\mathfrak{Y}^+(W_0)$ with the property that the action 
$j:\co_E\map{}\End(\mathfrak{g})$ also lifts.  We deduce    that the action of $\co_{E_0}$ on $\mathrm{Lie}(\mathfrak{G}''_{/R})$  is 
through a \emph{single} homomorphism $\co_{E_0}\map{}R$.  This contradicts what was said in the previous paragraph.  Thus  the 
reduction of $(\mathfrak{G}'',\rho'')$ to $R$ is not in the image of $\mathfrak{M}_0(R)\map{}\mathfrak{M}(R)$, and so 
$I_{\mathfrak{M}}(\mathfrak{C}'(s,\xi) , \mathfrak{M}_0) =1.$
\end{proof}

\begin{Cor}\label{Cor:ram proper horizontal} 
If $p\not=2$ then
$$
\sum_{\mathfrak{C} } I_\mathfrak{M}(\mathfrak{C},\mathfrak{M}_0) = 
\frac{ -4p^{c_0+1}  +2p + 2  }{ (p-1)^2 } - \frac{   (2c_0+1) p^{c_0+1} + 2c_0 +1 }{ p-1 }
$$
where the sum is over all proper horizontal components  $\mathfrak{C}\map{}\mathfrak{Y}$.
\end{Cor}

\begin{proof}
If $0\le t <  s \le c_0$ and  then combining $|H_t/H_s| = p^{s-t}$ with Proposition \ref{Prop:ram intersection I} shows that
$$
\sum_{ \substack{  \xi\in\co_E^\times/H_s  \\ \xi\not=1   } } I_{\mathfrak{M}}(\mathfrak{C}(s,\xi) , \mathfrak{M}_0 ) =  
2sp^s - \frac{2p^s-2}{p-1}
$$
while Proposition \ref{Prop:ram intersection II} shows that
 $$
\sum_{   \xi\in\co_E^\times/H_s } I_{\mathfrak{M}}(\mathfrak{C}'(s,\xi) , \mathfrak{M}_0 ) = p^s .
$$
Summing over $0 \le s\le c_0$ and using Proposition \ref{Prop:ram orbits} yields the desired result.
\end{proof}


\section{Appendix}


In this appendix we consider a slightly modified version of the deformation problem of the main text.

\subsection{A ramified variant}

Let $\mathfrak{g}_0$, $E_0$, $\Z_p[\gamma_0]=\Z_p+p^{c_0}\co_{E_0},$ $j_0:\co_{E_0}\map{}\End(\mathfrak{g}_0),$
and $\mathfrak{Y}_0\map{}\mathfrak{M}_0$ be  as in \S \ref{ss:functors},  but now choose a ramified quadratic extension 
$F/\Q_p$ whose ring of integers $\co_F$ will assume the role played earlier by $\Z_{p^2}$.  Thus we define a $p$-Barsotti-Tate 
group $\mathfrak{g}=\mathfrak{g}_0\otimes\co_F$ over $\F$ of height $4$ and dimension $2$ and let $\mathfrak{M}$ be the
 formal $\ZZ_p$-scheme classifying deformations  of $\mathfrak{g}$, with its $\co_F$ action, to objects of $\Art$.  Set 
 $E=E_0\otimes_{\Q_p} F$, let $\gamma=\gamma_0\otimes 1\in E$, and let $\mathfrak{Y}$ be the closed formal subscheme 
 of $\mathfrak{M}$ classifying those deformations for which the action of $\gamma$ lifts.  We again have the cartesian 
 diagram of closed immersions (\ref{functor diagram}) in which the horizontal arrows are now defined by the functor $\otimes \co_F$.

Many of the methods used in the main body of the article can be applied to study this new formal scheme $\mathfrak{Y}$.  
Virtually everything said in \S \ref{s:prelims} holds simply by replacing $\Z_{p^2}$ everywhere by $\co_F$
 (except that the action of $\co_{E_0}\otimes_{\Z_p} \co_F$ on $\mathfrak{g}$ deduced from $j_0$ may not extend to all 
 of $\co_E$.)  The methods used in \S \ref{s:vertical} to study vertical components should apply in this new setting 
 (although we have not checked this in any detail).  The more difficult problem seems to be extending the methods of 
 \S \ref{s:horizontal} to study the horizontal components of $\mathfrak{Y}$.  The issue is that there may be deformations
  of $\mathfrak{g}$ with its $\co_F[\gamma]$-action to characteristic $0$ whose full endomorphism ring is an $\co_F$-order 
  in $\co_E$ which is not of the form $\co\otimes_{\Z_p}\co_F$ for any $\Z_p$-order $\co\subset \co_{E_0}$.  One should
   not expect the component of $\mathfrak{Y}$ containing such  a deformation to come from $\mathfrak{Y}_0$ in the way 
   described in Propositions \ref{Prop:unr orbits} and \ref{Prop:ram orbits}.  

We will content ourselves for the moment with the following simple case, which is needed for the global intersection
 theory  of \cite{howardC}.

\begin{Prop}
If  $E_0/\Q_p$ is unramified and $c_0=0$ then each of $\mathfrak{Y}_0$ and $\mathfrak{Y}$ are isomorphic to
 $\Spf(\ZZ_p)$.  In particular the closed immersion $\mathfrak{Y}_0\map{}\mathfrak{Y}$ is an isomorphism and 
  $\mathfrak{Y}$ is contained in $\mathfrak{M}_0$.
\end{Prop}

\begin{proof}
That $\mathfrak{Y}_0\iso \Spf(\ZZ_p)$ is a special case of Proposition \ref{Prop:gross-keating}, and so we turn 
to $\mathfrak{Y}$.  Our hypotheses imply that 
$$
\co_E=\co_{E_0}\otimes_{\Z_p}\co_F=\co_F[\gamma].
$$
Let $\Psi,\overline{\Psi}:\co_{E_0}\map{}\ZZ_p$ be as in \ref{ss:dieudonne}.  For any object $R$ of $\Art$ 
we denote again by $\Psi$ and $\overline{\Psi}$ the ring homomorphisms obtain by composition with the canonical 
map $\ZZ_p\map{}W(R)$.  Define a display $\mathbf{D}_R=(P_R,Q_R,F,V^{-1})$ over $R$ as follows.  
The $W(R)$-module  $P_R$ is free  on the generators $\{e_1,e_2,f_1,f_2\}$, the submodule $Q_R\subset P_R$ is 
$$
Q_R=I_R e_1 + I_R e_2 + W(R) f_1+W(R) f_2,
$$
and $F$ and $V^{-1}$ are determined by the stipulation that the displaying matrix  \cite[(9)]{zink02}  of 
$\mathbf{D}_R$ with respect to the basis $\{e_1,e_2,f_1,f_2\}$ is
$$
\left(\begin{matrix}
  &  & 1 \\
  & & & 1 \\
  1\\ & 1
\end{matrix}\right).
$$
Define an action of $\co_E$ on $\mathbf{D}_R$ as follows.  An element $r\in \co_{E_0}$ acts as
$$
\left(\begin{matrix}
\Psi(r)   & 0  \\
 0 & \Psi(r)  \\
  & & \overline{\Psi}(r) & 0 \\
  & & 0 & \overline{\Psi}(r)
\end{matrix}\right).
$$
Choose a uniformizer $\varpi_F\in\co_F$ with minimal polynomial $x^2-ax-b$ and let $\varpi_F$ act
 as 
 $$
\left(\begin{matrix}
 0 & b   \\
 1 & a  \\
  & & 0 & b \\
  & &1 & a 
\end{matrix}\right).
$$
These rules define commuting actions of $\co_{E_0}$ and $\co_F$, and so determine an action of $\co_E$.

One easily checks that $\mathbf{D}_\F$ is $\co_E$-linearly isomorphic to the display of $\mathfrak{g}$. 
 Indeed, in the notation of \S \ref{ss:dieudonne} there is an isomorphism $\mathbf{D}_\F\iso \mathbf{d}_0\otimes\co_F$ defined by
$$
e_1\mapsto e_0\otimes 1\qquad
e_2\mapsto e_0\otimes \varpi_F \qquad
f_1\mapsto f_0\otimes 1\qquad
f_2\mapsto f_0\otimes\varpi_F.
$$
Thus  each  $\mathbf{D}_R$ is a deformation to $R$ of the display of $\mathfrak{g}$ with its $\co_E$-action.  
Suppose we have a morphism $R'\map{}R$ in $\Art$ whose kernel $J$ satisfies $J^2=0$.   By Zink's deformation
 theory \cite[Theorem 48]{zink02} the set of all  deformations to $R'$ of $\mathbf{D}_{R}$ with its $\co_E$-action 
 is in bijection with the set of $\co_E$-stable direct summands  $T\subset P_{R'}/I_{R'} P_{R'}$ which lift the Hodge filtration
$$
Q_R/I_RP_R\subset P_R/I_RP_R.
$$
The condition that $T$ be stable under the action of $\co_{E_0}\otimes_{\Z_p}\ZZ_p\iso \ZZ_p\times\ZZ_p$ implies 
that $T=T(\Psi)\oplus T(\overline{\Psi})$ with $T(\Psi)\subset R' e_1\oplus R'e_2$ and $T(\overline{\Psi})\subset R'f_1\oplus R'f_2$.  
As $\co_{E_0}$ acts on $Q_R/I_RP_R$ through $\overline{\Psi}$ we must have $T(\Psi)=0$.  Thus 
$$
T = T(\overline{\Psi}) \subset R'f_1\oplus R'f_2 = Q_{R'}/I_{R'}Q_{R'}
$$
from which $T=Q_{R'}/I_{R'}Q_{R'}$ follows.  We deduce that $\mathbf{D}_{R'}$ is the unique deformation of 
$\mathbf{D}_R$ with its $\co_E$-action.   By induction on the length of the local ring $R$ we conclude that 
$\mathbf{D}_R$ is the unique deformation to $R$ of $\mathbf{D}_\F$ with its $\co_E$-action, and it follows that
 $\mathfrak{Y}\iso \Spf(\ZZ_p)$.
\end{proof}


\def\cprime{$'$}

\end{document}